\documentclass[11pt,leqno]{article}
\usepackage{amssymb,amsfonts}
\usepackage{amsmath,amsthm,amsxtra}
\usepackage{epsfig}
\usepackage{color}
\usepackage{verbatim}

\newcommand\bigcheck[1]{#1 \raise1ex\hbox{$\hspace{-1ex}{}^\vee$}}
\newcommand\sucheck[1]{#1 \raise0.5ex\hbox{$\hspace{-1ex}{}^\vee$}}

\newcommand{\ch}{{\rm ch}}

\renewcommand{\Im}{\mathop{\rm Im  \, }}

\renewcommand{\sl}{s\ell}

\renewcommand{\span}{{\rm span}}

\newcommand{\sdim}{\mathop{\rm sdim \, }}

\newcommand{\tr}{\mathrm{tr} \, }
\newcommand{\tw}{\mathrm{tw}}

\newcommand{\bz}{\bar{0}}

\renewcommand{\Re}{\mathop{\rm Re  \, }}

\newcommand{\CC}{\mathbb{C}}

\newcommand{\QQ}{\mathbb{Q}}

\newcommand{\ZZ}{\mathbb{Z}}

\newcommand{\fg}{\mathfrak{g}}
\newcommand{\fh}{\mathfrak{h}}

\newcommand{\fn}{\mathfrak{n}}

\newcommand\bl{(\, . \, | \, . \, )}

\renewcommand{\tilde}{\widetilde}
\renewcommand{\hat}{\widehat}

\makeatletter
\renewcommand\section{\@startsection {section}{1}{\z@}%
                                   {-3.5ex \@plus -1ex \@minus -.2ex}%
                                   {2.3ex \@plus.2ex}%
                                   {\normalfont\large\bfseries}}
\renewcommand\subsection{\@startsection{subsection}{2}{\z@}%
                                     {-3.25ex\@plus -1ex \@minus -.2ex}%
                                     {0ex \@plus .0ex}%
                                     {\normalfont\normalsize\bfseries}}

\@addtoreset{equation}{section}
\makeatother

\newtheorem{theorem}{Theorem}[section]

\newtheorem{lemma}[theorem]{Lemma}
\newtheorem{corollary}[theorem]{Corollary}
\newtheorem{proposition}[theorem]{Proposition}

\newtheorem*{lemma*}{Lemma}

\theoremstyle{remark}
\newtheorem{remark}[theorem]{Remark}

\makeatletter
\def\@maketitle{\newpage
 \null
 \vskip 2em
 \begin{center}%
 \vskip 3em
  {\Large\bf \@title \par}%
  \vskip 1.5em
  {\normalsize
   \lineskip .5em
   \begin{tabular}[t]{c}\@author
   \end{tabular}\par}%
  \vskip 2em

 \end{center}%
 \par
 \vskip 2.5em}
\makeatother

\renewcommand{\epsilon}{\varepsilon}

\definecolor{light}{gray}{.9}

\newcommand{\half}{\frac{1}{2}}
\newcommand{\thalf}{\tfrac{1}{2}}
\newcommand{\zp}{\ZZ_{>0}}

\newcommand{\la}{\lambda}
\newcommand{\La}{\Lambda}
\newcommand{\al}{\alpha}

\newcommand{\wg}{\widehat{\fg}}
\newcommand{\wh}{\widehat{\fh}}
\newcommand{\wrh}{\widehat{\rho}}

\newcommand{\tz}{(\tau, z)}
\newcommand{\tzt}{(\tau, z, t)}
\newcommand{\tzzt}{(\tau, z_1, z_2, t)}
\newcommand{\tzzzt}{(\tau, z_1, z_2, z_3, t)}
\newcommand{\tuvt}{(\tau, u, v, t)}
\newcommand{\LLa}{L(\Lambda)}
\newcommand{\tot}{\frac{\tau}{2}}
\newcommand{\tof}{\frac{\tau}{4}}
\newcommand{\tch}{\tilde{\ch}}
\newcommand{\tph}{\tilde{\Phi}}
\newcommand{\tps}{\tilde{\Psi}}
\newcommand{\tef}{\tilde{F}}
\newcommand{\tbe}{\tilde{B}}
\newcommand{\HR}{\mathrm{HR}}

\begin{document}

\title{Representations of superconformal algebras and mock theta functions}


\author{Victor G.~Kac\thanks{Department of Mathematics, M.I.T, 
Cambridge, MA 02139, USA. Email:  kac@math.mit.edu~~~~Supported in part by an NSF grant.} \ 
and Minoru Wakimoto\thanks{Email: ~~wakimoto@r6.dion.ne.jp~~~~.
Supported in part by Department of Mathematics, M.I.T.}}

\maketitle

\setcounter{section}{-1}

\noindent To  Ernest Borisovich Vinberg on his $ 80^{th} $ birthday

\section*{Abstract}
It is well known that the normaized characters of integrable highest weight modules of given level over an affine Lie algebra
$\hat{\fg}$
span an $SL_2(\ZZ)$-invariant space. This result extends to admissible $\hat{\fg}$-modules, where $\fg$ is a simple Lie algebra or $osp_{1|n}$. Applying the quantum Hamiltonian reduction (QHR) to admissible $\hat{\fg}$-modules when $\fg =sl_2$ (resp. $=osp_{1|2}$) one obtains minimal series modules over the Virasoro (resp. $N=1$ superconformal algebras), which form modular invariant families.

Another instance of modular invariance occurs for boundary level admissible modules, including when $\fg$ is a basic Lie superalgebra. For example, if $\fg=sl_{2|1}$ (resp. $=osp_{3|2}$), we thus obtain modular invariant families of $\hat{\fg}$-modules, whose QHR produces the minimal series modules for the $N=2$ superconformal algebras (resp. a modular invariant family of $N=3$ superconformal algebra modules).

However, in the case when $\fg$ is a basic Lie superalgebra different from a simple Lie algebra or $osp_{1|n}$, modular invariance of normalized supercharacters
 of admissible $\hat{\fg}$-modules holds outside of boundary levels only after their modification in the spirit of Zwegers' modification of mock theta functions. Applying the QHR, we obtain families of representations of $N=2,3,4$ and big $N=4$ superconformal algebras, whose modified (super)characters span an $SL_2(\ZZ)$-invariant space. 

\section{Introduction}

Let $ \fg $ be a basic finite-dimensional Lie superalgebra, i.e. $ \fg $ is 
simple, its even part $ \fg_{\bar{0}} $ is reductive and $ \fg $ carries a non-degenerate invariant supersymmetric bilinear form $ \bl. $ 
Given a nilpotent element $ f \in  \fg_{\bar{0}} $ and $ K \in \CC, $ one associates to the triple $ (\fg, f, K) $ a vertex algebra $ W^K (\fg, f) $, 
obtained from the universal affine vertex algebra $ V^K (\fg) $ of level 
$ K $ by the quantum Hamiltonian reduction (QHR)  \cite{FF90}, 
\cite{KRW03}, \cite{KW04}. 
The vertex algebra $ W^K (\fg, f) $ is freely and strongly generated by 
quantum fields, labeled by a basis of the centralizer  $ \fg^f $ of $f$ in $ \fg$.  

Moreover, QHR provides a functor from the category of $ V^K (\fg) $-modules to the category of $ W^K (\fg, f) $-modules (some modules may go to zero). This 
functor allows, in particular, derive the characters 
of $ W^K (\fg, f) $-modules and their modular transformations from that of $ V^K (\fg) $-modules 
\cite{FKW92}, \cite{KRW03}, \cite{A05}.

The simplest subclass of the $ W $-algebras 
$ W^K (\fg, f) $, studied in detail in \cite{KW04}, 
corresponds to $ f = e_{-\theta} $, the lowest root vector of $ \fg  $ for some choice of positive roots. The vertex algebra $ W^K (\fg, e_{-\theta}) $
carries a conformal vector $L$ with central charge
\begin{equation}
\label{0.1}
 c(K)=\frac{K \sdim \fg}{K+h^\vee} -6K+h^\vee-4,
\end{equation}
provided that $K\neq- h^\vee$. Throughout the paper $h^\vee$, called the dual Coxeter number of $\fg$, is the half of the eigenvalue in the adjoint representation of the Casimir element 
of $\fg$, for the normalization of the bilinear form by
\begin{equation}
\label{0.2}
(\theta|\theta)=2. 
\end{equation} 
This class of $ W $-algebras, called \textit{minimal}, covers all well-known superconformal algebras:
\begin{itemize}
\item[] $ W^K (s\ell _2, e_{-\theta}) $  is the Virasoro vertex algebra, 
\item[] $ W^K (spo_{2|3}, e_{-\theta}) $ is the Neveu-Schwarz algebra, 
\item[] $ W^K (s\ell _{2|1}, e_{-\theta}) $ is the $ N =2 $ superconformal algebra,
\item[] $ W^K (ps\ell _{2|2} , e_{-\theta}) $ is the $ N= 4 $ superconformal algebra, 
\item[] $ W^K (spo_{2|3}, e_{-\theta}) $ (tensored with one fermion) is the $ N =3 $ superconformal algebra, 
\item[] 
$ W^K (D (2,1; a), e_{-\theta}) $ 
(tensored with four fermions and one boson) is the big $ N=4 $ superconformal algebra.
\end{itemize}
Throughout the paper by $N$ superconformal algebra we always mean the corresponding minimal $W$-algebra.

Recall that an irreducible highest weight module $ L (\La) $ over the affine Lie superalgebra $ \wg $ is called \textit{integrable} if all root vectors of $ \wg $, attached to roots $ \al, $ such that $ (\al | \al) > 0, $ act locally nilpotently on $ L (\La)$ (throughout the paper we keep the usual normalization
(\ref{0.2}) of the invariant bilinear form). It was established in \cite{KP84} that if $ \fg $ is a Lie algebra, then the  span of the normalized characters of integrable $ \wg $-modules $ L(\La) $ of given level $ K $ is $ SL(2, \ZZ) $-invariant, since normalized numerators can be expressed in terms of theta functions (= Jacobi forms), due to the Weyl-Kac character formula \cite{K90}. (The $ SL(2, \ZZ) $-invariance of the normalized denominator easily follows from the Jacobi triple product identity.)

As we have discovered in \cite{KW88}, \cite{KW89}, modular invariance of normalized characters for affine Lie algebra $ \wg $ holds for a much larger class of modules $ L(\La) $, which we called \emph{admissible} modules (and we conjectured that these are all $ L(\La) $ with the modular invariance property, which we were able to verify only for $ \fg = s\ell_2 $). Roughly speaking, a $ \wg $-module $ L(\La) $ is called admissible, if the $ \QQ $-span of coroots of $ \wg$ coincides with that of the $ \La $-integral coroots, and with respect to the corresponding affine Lie algebra $ \wg_\La $ the weight $ \La $ becomes integrable after a shift by the Weyl vectors. We showed in \cite{KW88} that a formula, similar to the Weyl-Kac character formula holds if $ \La $ is an admissible weight: one just has to replace in the numerator the Weyl group 
of $ \wg $ by the subgroup, generated by reflections with respect to non-isotropic $ \La $-integral coroots. It follows that the numerators of normalized admissible characters are again expressed as linear combinations of Jacobi forms, which again implies modular invariance of normalized characters of admissible $ \wg $-modules. 

It was shown in \cite{FF90} and \cite{FKW92} that the QHR of admissible representations of the affine Lie algebra $\hat{ \fg}, $ associated to a simple
Lie algebra $\fg$,
 yields all minimal series representations of the corresponding $ W $-algebra 
$ W^K(\fg, f), $ when $ f $ is a principal nilpotent element. In particular, for $ \fg = s\ell_2$, 
$f$ =  pincipal (= minimal) nilpotent element, one obtains the minimal series representations of the Virasoro algebra, and at the same time shows that
their  normalized characters are modular invariant.
(Note that all integrable $\hat {\fg} $-modules go to zero under the QHR
if $\fg$ is a Lie algebra.) 

If $ \fg $ is a Lie superalgebra, which is not a Lie algebra, the situation is different in several respects. First, though the normalized affine superdenominator is modular invariant, the normalized denominator is not. However, on
the span of the normalized denominator and superdenominator and their Ramond twisted analogues the group $SL(2,\ZZ)$ acts by monomial matrices. Hence the question of modular invariance reduces to that of the span of the normalized numerators and supernumerators and their Ramond twisted analogues. 

Recall that the \emph{defect} of a basic Lie superalgebra $ \fg $ is the cardinality of a maximal isotropic set of roots $ T $, i.e. a set of linearly independent pairwise orthogonal isotropic roots of $ \fg $ (= dimension of a maximal isotropic subspace in the $ \QQ $-span of roots) \cite{KW94}. As has been explained in \cite{KW16}, one can expect modular invariance of numerators of 
normalized (super)characters of integrable (and the corresponding admissible)
$\hat{\fg}$-modules, even after their modification, only if $ (\La + \widehat{\rho} | T) = 0 $ for a set T, consisting of defect($\fg$) linearly independent isotropic pairwise orthogonal roots. 
Such $ L(\La) $ are called \emph{maximally atypical.} 
Provided that $T$ is contained in a set of simple roots of 
$\hat{\fg}$, such $ L(\La) $ are called \emph{tame}. 
The normalized numerators of (super)characters of admissible $\hat{\fg}$- modules , corresponding to tame integrable modules,
and hence their Ramond twisted analogues, 
can be expressed in terms of mock theta functions. This is a conjecture, which has been verified in many cases  
\cite{KW01}, \cite{KW16}, 
\cite{GK15}. Moreover these mock theta functions can be modified, so that the resulting normalized (super)characters span a modular invariant space \cite{KW16}.

In the case of $ \fg = spo _{n|1}, $ which is the only basic non Lie algebra of defect 0, the normalized (super)numerators and their Ramond twisted analogues are again linear combinations of theta functions, and their span is modular invariant for all the admissible respresentations \cite{KW88}. 
In particular, this holds for the Lie superalgebra $ \fg = spo_{2|1},$ when the QHR of admissible $ \hat{\fg} $-modules and their twisted analogues produces all minimal series representations of the Neveu-Schwarz and Ramond algebra \cite{KRW03}, whose normalized
characters and supercharacters span a modular invariant space. 

The level $ K  $ of an admissible $ \wg $-module is a rational number, 
expressed via the level $m$ of the corresponding integrable module by
\begin{equation}
\label{0.3} 
K = \frac{m+h^\vee}{M}-h^\vee,\,\, \hbox{where}\quad M \in \ZZ_{\geq 1}. 
\end{equation}
Of special interest are the so called boundary level $ \wg $-modules, i.e. 
the admissible $ \wg $-modules, corresponding to the trivial (integrable)
module; their levels are 
\begin{equation}
\label{0.4} 
K = \frac{h^\vee}{M}-h^\vee, \,\, \hbox{where} \quad M \in \ZZ_{\geq 1}. 
\end{equation}
\noindent In this case the (super)characters and their twisted analogues can be expressed via the affine (super)denominator, see \cite{KW89} (resp. \cite{GK15}) in the Lie algebra (resp. superalgebra) case. As a result, modular invariance still holds, and therefore it holds after the QHR. 

In particular, in the case $ \fg = s \ell_{2|1}, $ when the corresponding minimal $ W $-algebra is the $ N =2  $ superconformal algebra, the boundary levels are $K = \frac{1}{M} -1, \mbox{ where } M \in \ZZ_{\geq 2}$.
Then, by the QHR, one gets modular invariant 
representations of the $N=2$ superconformal algebras with central charge $ c = 3 - \frac{6}{M} $ \cite{RY87}.
(For $M=1$ the $\hat{\fg}$-module is trivial, hence goes to zero.)    These are the well-known $ N =2 $ minimal series representations, whose modular invariance is well known \cite{KW94}, \cite{KRW03}. However, there 
are many more tame integrable and the corresponding admissible $ \hat{\fg} $-modules; their levels are \cite{KW14}, \cite{KW15}:
\begin{equation}
\label{0.5}
 K = \frac{m+1}{M}-1, \mbox{ where } M \in \ZZ_{\geq 2}, \ m \in \ZZ_{> 0}, \ gcd(2m+2, M)=1.
\end{equation}
\noindent As shown in \cite{KW14}, \cite{KW15}, though the 
normalized (super)characters of the tame integrable $\hat{\fg}$-modules and their Ramond twisted analogues are not modular invariant, their modifications in the spirit of Zwegers 
\cite{Z08} are. Consequently, the modified (super)characters of the
corresponding, via the QHR, representations of the $ N =2 $ superconformal 
algebras with central charge
$c(K)=3(1-\frac{2m+2}{M})$ 
span a modular invariant subspace.

Likewise, in the case $ \fg = spo_{2|3}, $ 
for the boundary level 
\[ K = \frac{1}{2M} - \half, \quad \hbox{where}\,\, M \in \ZZ_{\geq 1}, \mbox{ odd}, \]
\noindent the QHR produces modular invariant representations of the $N=3$
superconformal algebras with central charge $ c = -\frac{3}{M} $, 
see \cite{KW15}, Section 6. 
However, there are many more admissible $ \wg $-modules, corresponding to tame integrable modules; their levels are
\begin{equation}
\label{0.6}
 K = \frac{2m+1}{2M} - \half, \mbox{ where } M \in \ZZ_{\geq 1}, \ m \in \ZZ_{\geq 0}, \  \gcd (M, 4m+2) = 1. 
\end{equation}
\noindent As shown in \cite{KW15}, Section 6, though the corresponding, via the QHR, representations of the $ N=3 $ superconformal algebra with central charge $ c(K) = -\frac{3(2m+1)}{M} $ are not modular invariant for $m>0$, the Zwegers like modifications of their characters are. 

It turns out that for $ \fg = spo_{2|3} $ the corresponding affine Lie superalgebra $ \wg  $ has important \textit{complementary integrable} highest weight modules $ L(\La), $ namely those, for which all root vectors of $ \wg, $ attached to roots $ \al $ with $ (\al | \al ) <0,  $ act locally nilpotently on $ L(\La). $ Such tame $ \wg $-modules were studied in \cite{KW16}, Section 6.4, and \cite{GK15}. The levels of such non-critical modules are
$K=-\frac{m+2}{4}$, where  $m\in \ZZ_{\geq 1}$.
Consequently, by (\ref{0.3}), the corresponding admissible modules have level 
\begin{equation}
\label{0.7}
K=-\frac{m}{4M}
-\frac{1}{2}, \,\,\hbox{where}\,\, m, 
M \in \ZZ_{\geq 1}.
\end{equation}
These $\hat{\fg}$-modules are studied in Section 3. Furthermore, in Section 4
 we construct, via the QHR,  the corresponding family of representations 
of the $N=3$ Neveu-Schwarz and Ramond type superconformal algebras, of 
central charge 
$c(K)=\frac{3m}{2M}-\half$ (cf. \eqref{0.1}) , where $M$ is a 
positive odd integer, coprime to $m$, such that  
their modified (super)characters span a modular invariant space 
(for $m=1$ modular invariance holds without modification), see Theorem \ref{th4.12}.

Next, for $ \fg = ps\ell_{2|2} $, the corresponding affine Lie superalgebra $ \wg $ has a family of tame integrable modules of negative integer level $ m $ (see Section 6), for which we construct in Section 7, via the QHR, $N=4$ 
superconformal algebra modules, whose modified (super)characters form a modular invariant family. In Section 8 we study the associated  family of principal admissible modules of level 
\begin{equation}
\label{0.8}
K = \frac{m}{M}, \mbox{ where }\, M \in \ZZ_{\geq 1},
\end{equation}
and compute their modified (super)characters and those of their twists. They do not span a modular invariant space. However, we show in Section 9 (see Theorems \ref{th9.6}-\ref{th9.8}) that the modified (super)characters of their QHR, which have central charge $ c(K) = -6 \left(\frac{m}{M}+1 \right), $ do span a modular invariant space, provided that $\gcd(M,2m)=1$ if $m\leq -2$. 

Finally, in Section 10 we consider $ \fg = D(2,1; a), $ a family of 17-dimensional exceptional Lie  superalgebras. The tame integrable $ \wg $-modules exist only for $ a $ of the form $ a = -\frac{p}{p+q}, $ where $ p,q \in \ZZ_{\geq 1} $ are coprime. Then the level of such a module is of the form
\begin{equation}
\label{0.9}
K = -\frac{pqn}{p+q}, \mbox{ where } n \in \ZZ_{\geq 1}.
\end{equation}
We construct four families of tame integrable $ \wg $-modules of level $ K, $ compute their modified supercharacters (there are actually two types of modifications) and show in Corollary \ref{cor10.22} that they span an $ SL_2 (\ZZ) $-invariant space. Applying the QHR to these $\hat{\fg}$-modu;es, we obtain in Section 11 a family of positive energy irreducible representations of the big 
$N=4$ Neveu-Schwarz and Ramond type superconformal algebras with central charge
$c(K)=6\frac{pqn}{p+q}$, where $n$ is a positive integer, whose modified
(super)characters span a modular invariant subspace. 

In Sections 5 and 12 we consider some other cases when the modular invariance holds without modification.

\section{Some important functions and their transformation properties}
Fix a positive integer $ m. $ A theta function (=Jacobi form) of rank 1 and degree (=index) $ m $ is defined by the following series:
\begin{equation}
\label{1.1}
\Theta_{j,m} (\tau, z, t) = 
e^{2 \pi i m t} \sum_{n \in \ZZ + \frac{j}{2m}} q^{mn^2} e^{2 \pi i m n z},\,\hbox{where}\, j \in \ZZ / 2m \ZZ,
\end{equation}
\noindent which converges in the domain $ (\tau, z, t) \in \CC^3, \Im \tau > 0, $ to a holomorphic function. 

These functions have nice elliptic and modular transformation formulas (cf., e.g., Appendix in \cite{KW14}. 
The  elliptic transformation formulas are as follows for 
$am\in \ZZ$ and $k\in \ZZ$: 
\begin{equation}
  \label{1.2}
  \Theta_{j,m} (\tau,z+a,t) = e^{\pi i ja} \Theta_{j,m}(\tau,z,t); \,\, 
\,\, \Theta_{j,m} (\tau,z+\frac{k\tau}{m},t) = q^{-\frac{k^2}{4m}}
e^{-\pi ik z} \Theta_{j+k,m}(\tau,z,t).
\end{equation}
The modular transformation formulas are:
\begin{equation}
  \label{1.3}
  \Theta_{j,m} \left( -\frac{1}{\tau}\,,\, \frac{z}{\tau}\,,\, t-
\frac{z^2}{4\tau} \right)
     = \left( \frac{-i\tau}{2m}\right)^{\tfrac12} \sum_{j'\in \ZZ   /2m\ZZ} 
         e^{-\frac{\pi i jj'}{m}} \Theta_{j',m} (\tau ,z,t)\, , 
\end{equation}
\begin{equation}
  \label{1.4}
  \Theta_{j,m}(\tau +1,z,t) =e^{\frac{\pi i j^2}{2m}}
    \Theta_{{j,m}}(\tau,z,t).
\end{equation}
We often write $ \Theta_{j,m} (\tau, z) = \Theta_{j,m} (\tau, z, 0). $
Especially important are the celebrated four Jacobi theta functions of degree two (we put $t=0$ here):
\begin{eqnarray*}
  \vartheta_{00} = \Theta_{2,2} + \Theta_{0,2} \, , \, 
     \vartheta_{01} =-\Theta_{2,2} + \Theta_{0,2}\, , \,
  \vartheta_{1 0} = \Theta_{1,2} + \Theta_{-1,2} \, , \, 
      \vartheta_{11} = i \Theta_{1,2} -i\Theta_{-1,2}\, .
\end{eqnarray*}
Formula \eqref{1.2} for $m=2$ implies that for $a,b =0$ or $1$ one has:
\begin{equation}
\label{1.5}
  \vartheta_{ab} (\tau,z+1) = (-1)^{a} \vartheta_{ab} (\tau ,z)\, ;\, 
\vartheta_{ab} (\tau,z+\tau) = (-1)^{b} q^{-\frac{1}{2}} e^{-2\pi i z} \vartheta_{ab} (\tau ,z)\, .
\end{equation}
Formula \eqref{1.3} for $ m = 2 $ implies that for $ a,b = 0 $ or $1$ one has:
 \begin{eqnarray}
 \label{1.6}
    \vartheta_{ab} \left( - \frac{1}{\tau} , \frac{z}{\tau} \right)
        &=& (-i)^{ab} (-i\tau)^{\frac12} e^{\frac{\pi i
            z^2}{\tau}} \vartheta_{ba} (\tau ,z )\, ;\\
 \label{1.7}
    \vartheta_{0a} (\tau +1,z) &=& \vartheta_{0b} (\tau,z)\, , \, {\,
      \hbox{where\,\,} } a \neq b \, ; \quad 
     \vartheta_{1a} (\tau +1,z)=  e^{\frac{\pi i}{4}}
     \vartheta_{1a}(\tau ,z)\, .
  \end{eqnarray}
  
A mock theta function of rank 1 and degree $ m\in \ZZ_{>0} $ is defined by the following series:
\begin{equation}
\label{1.8}
\Phi^{[m;s]}_{1} (\tau, z_1, z_2) = \sum_{j \in \ZZ} \ \frac{e^{2\pi i m j (z_1 + z_2) + 2 \pi i sz_1} q^{j^2 m + js}}{1 - e^{2 \pi i z_1} q^j},\,\hbox{where}\, s\in \half \ZZ,
\end{equation}
\noindent which converges in the domain $ (\tau, z_1, z_2) \in \CC^3, \Im \tau > 0, $ to a meromorphic function. 

As in \cite{KW14}-\cite{KW16}, we will extensively use in the paper the functions
\begin{equation}
\label{1.9}
 \Phi^{[m;\,s]} (\tau, z_1, z_2, t) = e^{2 \pi i m t} \left(\Phi^{[m;\,s]}_1 (\tau, z_1, z_2) -\Phi^{[m;\,s]}_1 (\tau, -z_2, -z_1) \right), 
\end{equation}
\noindent and 
\begin{equation}
\label{1.10}
\Psi_{a,b; \epsilon'} ^{[M, m, s;\epsilon]}
(\tau, z_1, z_2, t) = q^{\frac{mab}{M}} e^{\frac{2 \pi i m }{M}(bz_1 + az_2)} \Phi^{[m;\,s]} (M\tau, z_1 + a \tau + \epsilon, z_2 + b\tau + \epsilon, \frac{t}{M}), 
\end{equation}
\noindent where $ M $ is a positive integer, $ \epsilon, \epsilon' = 0 $ or 
$ \half,$
and $a,b \in \epsilon' +\ZZ.  $ 

We will often use notation 
$\Phi^{[m;s]}(\tau,z_1,z_2)=\Phi^{[m;s]}(\tau,z_1,z_2,0)$, and similarly for
$\Psi$.

The following property of the functions $ \Phi^{[m;s]} $ will be useful in the sequel. 
\begin{lemma}
\label{lemma1.1}
\[ 2 \Phi^{[m;\,s]} (2 \tau, z_1, z_2, t) = \Phi^{[2m;\,2s]} (\tau, \frac{z_1}{2}, \frac{z_2}{2}, \frac{t}{2}) + (-1)^{2s} \Phi^{[2m;\,2s]} (\tau, \frac{z_1+1}{2}, \frac{z_2-1}{2}, \frac{t}{2}). \]
\end{lemma}
\begin{proof}
A straightforward verification. 
\end{proof}

The functions $ \Phi^{[m;\,s]}_1 $ and $ \Phi^{[m;\,s]} $ have neither good elliptic transformation nor modular transformation properties. In order to achieve these properties one needs to add to these functions a real analytic correction, defined as follows \cite{Z08}, \cite{KW15}. Let $ E(x) = 2 \int_{0}^{x} e^{-\pi u^2} du. $ For each $ j \in \half \ZZ $ let
\[ R_{j;\,m}(\tau,\, v) := \]
\[ \sum\limits_{n \in \half\ZZ \atop n  \equiv
 j \!\!\!\!
\mod\!\! 2m}\left({\rm sgn} (n-\frac{1}{2} - j + 2m) \right. \left.- E \left(\left(n-2m \frac{{\rm Im}\,v}{{\rm Im}\,\tau} \right) \sqrt{\frac{{\rm Im} \,\tau}{m}}\right) \right) 
e^{-\frac{\pi i n^2}{2m}\tau + 2\pi i n v}. \]
This series converges to a real analytic function for all $ v, \tau \in \CC, \ \Im \tau > 0. $

Introduce the real analytic ``correcting'' function
\begin{equation}
	\label{1.11}
\Phi^{[m;\,s]}_{\mathrm{add}} 
(\tau, z_1, z_2, t) = \frac{1}{2} e^{2\pi i mt} 
\sum\limits^{s+2m-1}_{j=s}{R_{j;\,m} \left(\tau, \frac{z_1-z_2}{2}\right) (\Theta_{-j,\,m}-\Theta_{j,\,m}) (\tau, z_1+z_2)},
\end{equation}
\noindent and let 
\begin{equation}
\label{1.12}
\tilde{\Phi}^{[m;\,s]} := \Phi^{[m;s]} + \Phi^{[m;\,s]}_{\mathrm{add}}
\end{equation}
\noindent be the \emph{modification}  of the function $ \Phi^{[m;s]}. $ 

The function $ \tilde{\Phi}^{[m;\,s]} $ has the following elliptic and modular transformation properties \cite{Z08}, \cite{KW15}:
\begin{equation}
\label{1.13}
 \tilde{\Phi}^{[m;\,s]} (\tau, z_1+a, z_2+b,t) = \tilde{\Phi}^{[m;\,s]} (\tau, z_1, z_2, t) \mbox{ if } a, b \in \ZZ,
\end{equation}
\begin{equation}
 \label{1.14}
 \tilde{\Phi}^{[m;\,s]} (\tau, z_1 + a\tau, z_2+b\tau, t) = q^{-mab} 
e^{-2\pi i m (b z_1+a z_2)}  \tilde{\Phi}^{[m;\,s]} (\tau, z_1, z_2, t)\, 
\mbox{ if }\, a, b \in \ZZ,
\end{equation}
\begin{equation}
\label{1.15}
\tilde{\Phi}^{[m;\,s]} \left(-\frac{1}{\tau}, \frac{z_1}{\tau}, \frac{z_2}{\tau}, t-\frac{z_1 z_2}{\tau} \right) = \tau \, \tilde{\Phi}^{[m;\,s]} (\tau, z_1, z_2, t),
\end{equation}
\begin{equation}
\label{1.16}
\tilde{\Phi}^{[m;\,s]} (\tau+1, z_1, z_2,t) = \tilde{\Phi}^{[m;\,s]} (\tau, z_1, z_2, t).
\end{equation}
\begin{remark}
\label{rem1.2}
Properties \eqref{1.13}--\eqref{1.16} hold for 
$ \tilde{\Phi}_1^{[m;s]}= \Phi_1^{[m;s]} 
+ \Phi_{1, \mathrm{add}}^{[m;s]}$, where $ \Phi_{1, \mathrm{add}}^{[m;s]} $ is given by \eqref{1.11} with $ \Theta_{j,m} - \Theta_{-j,m} $ replaced by $ \Theta_{j,m} $ .
\end{remark}

Introduce the modification $\tilde{\Psi}^{[M,\,m,\,s ;\, \epsilon]}_{a,\,b ;\,\epsilon'}$ 
of the function  $\Psi^{[M,\,m,\,s ;\,\epsilon]}_{a,\,b ;\,\epsilon'}$, defined by \eqref{1.10}, by replacing $\Phi^{[m;\,s]}$ in the RHS of \eqref{1.10} by its modification $\tilde{\Phi}^{[m;\,s]}$. Then it is not hard to deduce from \eqref{1.15}, \eqref{1.16} the following modular transformation properties of these modifications, where $\epsilon$, $\epsilon'=0$ or $\frac{1}{2}$, and $j, k  \in\,\epsilon'+\ZZ/M\ZZ$, provided that $M$ is odd and coprime to $m$ if $m>1$ (see \cite{KW15}, Theorem 2.8):
\begin{equation}
\label{1.17}
 \tilde{\Psi}^{[M,\,m,\,s;\,\epsilon]}_{j,\, k;\,\epsilon'} \left(-\frac{1}{\tau}, \frac{z_1}{\tau}, \frac{z_2}{\tau}, t - \frac{z_1 z_2}{\tau} \right)
= \frac{\tau}{M} \sum\limits_{a,\,b \,\in\, \epsilon + \ZZ/M \ZZ}{e^{-\frac{2\pi i m}{M} (ak+bj)}} \;\tilde{\Psi}^{[M,\, m,\, s;\,\epsilon']}_{a,\, b;\,\epsilon} (\tau, z_1, z_2, t),
\end{equation}
\begin{equation}
\label{1.18} 
\tilde{\Psi}^{[M,\, m,\, s;\,\epsilon]}_{j,\,k ;\,\epsilon'} (\tau+1, z_1, z_2, t) = e^{\frac{2\pi i m}{M}\,jk} \;\; \tilde{\Psi}^{[M,\,m,\,s; |\epsilon-\epsilon'|]} (\tau, z_1, z_2, t).
\end{equation}
\begin{remark}
\label{rem1.3}
Since  $\Theta_{j,1}= \Theta_{-j,1} $ for $j\in \ZZ$, we have: $ \tilde{\Phi}^{[1;s]}=\Phi^{[1;s]}$ for $s\in \ZZ$. Consequently, $\Phi^{[1;s]}$  is independent of $s\in\ZZ$. By the denominator identity for $\hat{sl}_{2|1}$\, (see, e.g.
\cite{KW94} or \cite{KW16}, Remark 6.21), we have 
\[\tilde{\Phi}^{[1;0]}(\tau, z_1, z_2)=\Phi^{[1;0]}(\tau, z_1, z_2)=-i \frac{\eta(\tau)^3\vartheta_{11}(\tau,z_1+z_2)}{\vartheta_{11}(\tau,z_1)\vartheta_{11}(\tau,z_2)}.\] 
Hence from the definition \eqref{1.10} we obtain for any positive integer $M$
that $\tilde{\Psi}^{[M,1;0;\epsilon]}_{j,k;\epsilon'}=\Psi^{[M,1;0;\epsilon]}_{j,k;\epsilon'}$ and
\[\tilde{\Psi}^{[M,1;0;\epsilon]}_{j,k;\epsilon'}(\tau, z_1, z_2)
=-i q^{\frac{jk}{M}}e^{\frac{2\pi i}{M}(kz_1+jz_2)}
 \frac{\eta(M\tau)^3\vartheta_{11}(M\tau,z_1+z_2+(j+k)\tau)}{\vartheta_{11}(M\tau,z_1+j\tau+\epsilon)\vartheta_{11}(M\tau,z_2+k\tau-\epsilon)}.
\] 
\end{remark} 
 
\section{Affine Lie superalgebras and their integrable and admissible modules}
 In complete analogy with simple Lie algebras, given a basic Lie superalgebra $ \fg $, we associate to it the affine Lie superalgebra $ \wg = \fg [t,t^{-1}] + \CC \mathbf{K} + \CC d $ with the same commutation relations. Choosing a Cartan algebra $ \fh $ of $ \fg_{\bz}, $ we let the Cartan subalgebra of $ \wg $ be $ \widehat{\fh} = \fh + \CC \mathbf{K} + \CC d. $ Furthermore, we choose a subset of positive 
roots $ \Delta_+ $ in the set of roots of $\Delta$
of $ \fg $, such that the highest root $ \theta $ is even, and extend the invariant bilinear form $ \bl $ on $ \fg, $ normalized by condition (\ref{0.2}), to an invariant bilinear form $ \bl $ on $ \wg $ in the same way as in the affine Lie algebra case. Then the restriction of this bilinear form to $ \wh $ is non-degenerate, hence we can identify $ \wh^* $ with $ \wh. $ The element of $ \wh^*, $ corresponding to $\mathbf{K} $ (resp. $ d $) under this identification, is denoted by $ \delta $ (resp. $ \La_0 $). We use the following coordinates on $ \wh = \wh^*: $
  \begin{equation}
 \label{2.1}
 h = 2 \pi i  (- \tau \La_0 + z + t \delta), \mbox{ where } \Im \tau > 0,\, 
z \in \fh,\, t \in \CC. 
 \end{equation}
  
Let $ \hat{\Delta}$ be the set of roots of $\hat{\fg}$. The corresponding to
the choice of the set of positive roots $\Delta_+$ of $\fg$,  
the set of positive roots of $ \wg $ is 
$ \hat{\Delta}_+ 
= \Delta_+ \cup ( \underset{\substack{n \in \ZZ_{>0} \\ \al \in \Delta \cup {0}}}{\cup} \{ \al + n \delta \}),   $ and the corresponding subalgebra of $ \wg  $ is $ \hat{\fn}_+ =\fn_+ + \fg [t] t, $ where $ \fn_+ $ is the subalgebra of $ \fg,  $ corresponding to $ \Delta_+ $.  The set of simple roots is $ \hat{\Pi} = \{ \al_0 = \delta - \theta \}\cup \Pi, $ where $ \Pi = \{ \al_1, \ldots, \al_n \}  $ is the set of simple roots for $ \Delta_+. $ As in the affine Lie algebra case, we let $ \hat{\rho} = h^\vee \La_0 + \rho$, where
$\rho = \half \sum_{\alpha \in \Delta_+}(-1)^{p(\alpha)}\alpha$.

Recall that the Weyl group $\hat{W}$ of $\hat{\fg}$ is the subgroup of $GL(\hat{\fh})$, generated by reflections in the roots $\alpha \in \hat{\Delta}$,
such that $(\alpha|\alpha)\neq 0$. One has
$\hat{W}=W\ltimes T_L$, where $W$ is the (finite) Weyl group of 
$\fg_{\bar{0}}$ and $T_L$ is the group, consisting of translations $t_\beta ,\beta \in L$, where $L$ is the coroot lattice of $\fg_{\bar{0}}$, and, for $\beta \in \fh =\fh^*$ the translation $t_\beta$ is defined by
\begin{equation}
\label{2.2}
t_\beta (\La) = \La + (\La|\delta) \beta - \left( (\La | \beta) +\thalf (\La|\delta) |\beta|^2 \right)\delta, \,\,\Lambda\in \hat{\fh}^*.
\end{equation}

 Given $ \La \in \wh^* $, one defines, as in the affine Lie algebra case, the irreducible highest weight $ \wg$-module $ L(\La) $ by the property that it admits a non-zero even weight vector $ v_\La $ with weight $\Lambda$, such that $\hat{\fn}_+ v_\La = 0.$
 
 A root $ \al $ of $ \wg $ is called \emph{integrable} for $ L(\La) $ if it is even and the root spaces 
$ \wg_\al $ and $ \wg_{-\al} $ 
act locally nilpotently on $ L(\La). $ A $ \wg $-module $ L(\La) $ is called \emph{integrable} (resp. \emph{complementary integrable}) if all roots $ \al  $ of $ \wg $, such that $ (\al | \al) > 0 $ (resp. $ (\al | \al ) <0 $) are integrable for $ L(\La). $
 
 We define the normalized character $ \ch^+_\La  $ and supercharacter  
$ \ch^-_\La $ of the $ \wg $-module $ L(\La ) $ by the series
 \begin{equation}
 \label{2.3}
 \ch^\pm_\La \tzt = q^{m_{\La}} \, \tr^\pm_{L(\La)}e^h,
 \end{equation}
 \noindent where $ h $ is as in \eqref{2.1}, $ \tr^+ $ and $ \tr^- $ denote the trace and supertrace respectively, $ m_\La = \frac{|\La + \widehat{\rho}|^2}{2 (K+h^\vee)} - \frac{\sdim \fg}{24}, $ and $ K=\Lambda(\mathbf{K}) $ is the level of $ \La. $ We shall always assume that $ \La $ is non-critical, i.e. $ K + h^\vee \neq 0. $ The series \eqref{2.2}, given by the weight space decomposition of $ L(\La), $ converge to a holomorphic function in the domain 
$ \{ h \in \wh |\, \Re \al_i (h)>0,\,\al_i \in \hat{\Pi}\}$.  
The simple proof of this fact, as well as more background in the affine Lie algebra case, may be found in \cite{K90} (the proof in the affine Lie superalgebra case is the same).
 
 By the Gorelik-Kac theorem, these functions extend to meromorphic functions in the domain $ \left\{ \tzt | \Im \tau > 0 \right\} $  in the affine Lie algebra case, but this is not known for an arbitrary $ \La $ in the affine Lie superalgebra case (though this holds in all examples at hand).
 
For tame integrable (and complementary integrable, if $\fg =spo_{2m|2m+1}$)
$\hat{\fg}$-modules $L(\Lambda)$ one has the conjectural Kac-Wakimoto (super)character formula \cite{KW01}, \cite{GK15}:
\begin{equation} 
\label{2.4} 
\widehat{R}^\pm \ch^\pm_\La =j_\Lambda ^{-1} 
q^{-\frac{|\Lambda +\hat{\rho}|^2}{2(K+h^\vee)}} 
\sum_{w \in \hat{W}^\#} \epsilon_\pm (w) 
w \frac{e^{\La + \widehat{\rho}}}{\prod_{\beta\in T}(1\pm e^{-\beta})} . 
\end{equation}
 Here $\hat{W}^\#$ is the subgroup of the group $\hat{W}$, generated by reflections with respect to the roots $\alpha \in \hat{\Delta}$, which are integrable
for $L(\Lambda)$, $j_\Lambda$ is a positive integer,
$ T  $ is a set, consisting of defect($ \fg $) positive isotropic pairwise orthogonal roots, such that $(\Lambda +\hat{\rho}|T)=0$, and 
$\epsilon_{\pm}(w)=(-1)^{\ell_\pm(w)}$, where 
$\ell_+(w)$ (resp. $\ell_-(w)$) is the number of reflections $r_\alpha$ with respect to even (resp. indivisible even) roots $\alpha\in\Delta$ in a decomposition of $w$. Finally,
   $\widehat{R}^\pm  $ is the normalized (super)denominator:
 \begin{equation}
 \label{2.5}
 \widehat{R}^\pm = q^{\frac{\sdim \fg}{24}} e^{\widehat{\rho}} \prod_{\al \in \widehat{\Delta}_+} (1-(\mp1)^{p(\al)}e^{-\al})^{(-1)^{p (\al)}}.
 \end{equation}
Formula (\ref{2.4}) has been checked in \cite{GK15}  in many, but not all
cases, considered in this paper. 

Denote by $ \Delta^{\#} $ the set of even roots $ \al $ of $ \fg, $ such that both $ \al  $ and $ \delta - \al $ are integrable (resp. complementary integrable) for $ \LLa, $ let $ L^{\#} $ be the lattice spanned over $ \ZZ $ by $ \al^\vee = 2 \al / (\al | \al), $ where $ \al \in \Delta^\#; $ this lattice is positive (resp. negative) definite. One has: $ \hat{W}^\# = W^\# \ltimes t_{L^\#} $,
where $W^\#$ is a (finite) subgroup of the group $\hat{W}^\#$, generated by reflections with respect to exactly one integrable root from the set $\{\alpha , \delta -\alpha\}$, and $ t_{L^\#} = \{ t_\al | \al \in L^\#  \}. $    
Let $ W^\#_0 $ be the subgroup of $ W^\#, $ generated by reflections $ r_\al, \al \in \Delta^\#, $ and let $ \hat{W}^\#_0 $ be the corresponding affine Weyl group, i.e. the group, generated $ r_\al $ and $ r_{\delta - \al}, \al \in \Delta^\#; $ equivalently, 
$ \hat{W}^\#_0 = W^\#_0 \ltimes t_{L^\#} $.
The integer $ j_\Lambda $ in \eqref{2.4} can be computed using the following.
\begin{proposition}
\label{prop2.1}
Let 
\[ B'_\La = \sum_{w \in \hat{W}_0^\#} \epsilon_- (w) w \frac{e^{\La + \wrh}}{\prod_{\beta \in T}(1-e^{-\beta})}, \]
and $j'_{\Lambda}$ be the coefficient of $e^{\Lambda+\hat{\rho}}$ in $B'_\Lambda$.
Then 
\[j'_\Lambda =\sum_w \epsilon_-(w),\]
where the summation is taken over 
$w \in \hat{W}_0^\#$,
such that 
 $w(\Lambda +\hat{\rho})=\Lambda +\hat{\rho}$ and $w(T)>0$.
\end{proposition}
\begin{proof}
Each term in the sum, defining $ B'_\La $ is, up to a sign, of the form
\[ \frac{e^{yt_{-\gamma} (\La + \wrh)}}{\prod_{\beta \in T} (1-e^{-y(\beta)} q^{(\beta|\gamma)})},\, \mbox{ where } \gamma \in L^\#, y \in W_0^\#. \]
We expand this in a formal power series in $ q^a $ and $ e^{-\omega}, $ where $ a>0 $ and $ \omega \in \Delta_+ $, by the following rules:
\[ 
\begin{aligned}
\frac{1}{1 -e^{-\omega}q^a} & = \sum_{j \in \ZZ_{\geq 0}} e^{-j\omega} q^{ja} \mbox{ if } a>0, \mbox{ or } a=0,\, \omega > 0,\\
\frac{1}{1 -e^{-\omega}q^a} & = -\frac{e^\omega q^{-a}}{1 - e^\omega q^{-a}} =- \sum_{j \in \ZZ_{\geq 1}} e^{j\omega} q^{-ja} \mbox{ if } a<0, \mbox{ or } a =0,\, \omega<0.\\
\end{aligned} \]
Then we obtain that $ B'_\La $ is a sum of terms of the following form
\[\pm  e^{\La + \wrh + (yt_{-\gamma} (\La + \wrh)-\La - \wrh)-\sum_{\beta \in T} c_\beta y(\beta)} q^{\sum_{\beta \in T} c_\beta (\beta|\gamma)}, \]
where $ y \in W_0^\#, \gamma \in L^\#, c_\beta \in \CC. $ 
But $ yt_{-\gamma} (\La + \wrh) -(\La + \wrh)) \in L^\# + \QQ \delta $ and 
$ \sum_{\beta \in T} c_\beta y(\beta)
 \in y(T) $.
Since the lattice $ L^\# $ is positive (or negative) definite and $ (T|T) = 0,  $ the intersection of $ L^\# $ with $ y(T) $ is zero, hence the above term can be equal to $ e^{\La + \wrh} $ only if 
$w$ fixes $\Lambda +\hat{\rho}$ and all $ c_\beta = 0.  $ This implies the formula for $j'_\Lambda$. 
\end{proof}

Using (\ref{2.2}), formula (\ref{2.4}) can be rewritten as follows:
 \begin{equation} \label{2.6} 
\widehat{R}^\pm \ch^\pm_\La = j_\Lambda^{-1}\sum_{w \in W^\#} \epsilon_\pm (w) w (\Theta^\pm_{\La + \widehat{\rho}, T}), 
\end{equation}
where, in coordinates \eqref{2.1}, we have
 \begin{equation}
 \label{2.7}
 \Theta^\pm_{\la, T}\tzt = e^{2 \pi i (K+h^\vee)t} 
\sum_{\gamma \in \frac{\bar{\la}}{K + h^\vee}+L^\#} \frac{q^{\half (K+h^\vee)|\gamma|^2} e^{2 \pi i (K+h^\vee) \gamma (z)}}{\prod_{\beta \in T} (1 \pm q^{-(\gamma | \beta)} e^{-2 \pi i \beta (z)})}.
 \end{equation}
 Here $ \bar{\la} $ is the restriction of $ \la \in \wh^* $ of level $ K + h^\vee $ to $ \fh$.
 The function (\ref{2.7}) is a \emph{mock theta function of defect = defect $ (\fg) $} \cite{KW14}.
  Note that the series \eqref{2.7} converges to a meromorphic function for
 $ \Im \tau > 0, $ provided that $ (K+h^\vee) |\gamma |^2 >0 $ in \eqref{2.7},
 which will hold in all examples. (In general, there could be an extra sign in front of each fraction (these are called the signed mock theta functions), but in all cases, considered in this paper the sign is $+$.) 
 
Of course, for $|T|=0$, (\ref{2.7})
is a classical Jacobi theta function. For $|T|=1$ and $|L^\#|=1$, (\ref{2.7})
is an Appell function \cite{Ap}.


Now we turn to the description of \emph{admissible weights,} associated to a given integrable or complementary integrable weight $ \La^0 $, cf. \cite{KW88}, \cite{KW15}. First, recall that a subset $ S $ of $ \hat{\Delta}_+ $ is called \emph{simple} if 
\[\al - \beta \notin \hat{\Delta} \mbox{ for } \al, \beta \in S; \ \QQ S = \QQ \hat{\Delta}.\]
An element $ \La \in \wh^* $ is called \emph{admissible} for a simple set $ S $ if $ \La^0 := \La + \hat{\rho} -\hat{\rho}_S $ is integrable for the affine Lie superalgebra $ \wg_S $. We then say that $ \La $ is an \emph{admissible weight}, associated to the integrable weight $ \La^0. $

There is a general formula, relating to the characters of the $ \wg $-module 
$ \LLa $ and the $ \wg_S $-module $ L(\La^0), $ see e.g. \cite{GK15}, formula (1). This formula has been established by efforts of many authors for an arbitrary symmetrizable Kac-Moody algebra, in particular, for affine Lie algebras, and non-critical $ \La. $ It is still unclear in what generality this formula holds in the affine Lie superalgebra case, see the discussion in \cite{GK15}, where it was proved in many cases.

For example, given a positive integer $ M, $ a subset $ S_{(M)} = \{  (M-1)\delta + \al_0 \} \cup \Pi $ is simple. Let $ \beta \in \fh^* $ and $ y \in W $ be such that $ S:= t_\beta y S_{(M)} \subset \hat{\Delta}_+, $ where the translation $ t_\beta $ for $ \beta \in \fh^* $ is given by ({2.2}).
Then, obviously, $ S $ is a simple subset, which is called a \emph{principal simple} subset. All \emph{principal admissible} weights, associated to the integrable
weight $ \La^0, $ with respect to this subset $ S,  $ denoted by $ \Pi_\La, $ are of the form (up to adding a multiple of $ \delta $):
\begin{equation}
\label{2.8}
\La = t_\beta y \left( \La^0 - (M-1)(K+h^\vee)\La_0 + \hat{\rho}\right)-\hat{\rho}.
\end{equation}
Of course, a principal admissible weight $\Lambda$ is integrable iff $M=1$.
Note that the level $ K $ of $ \La  $ is related to the level $ m $ of $ \La^0 $ by
\begin{equation}
\label{2.9}
K + h^\vee = \frac{m + h^\vee}{M}.
\end{equation}
The normalized supercharacters of the $ \wg $-modules $ L(\La) $ and $ L(\La^0) $ are related by the following formula, which is a special case of the general formula, mentioned above, see \cite{KW14}, formula (3.28):
\begin{equation}
\label{2.10}
(\hat{R}^-\ch^-_\La) \tzt = q^{\frac{m + h^\vee}{M} |\beta|^2} 
e^{\frac{2 \pi i (m + h^\vee)}{M}(z|\beta)} (\hat{R}^-\ch^-_{\La^0}) \left( M\tau, y^{-1} (z+\tau \beta), \frac{t}{M} \right).
\end{equation}

\noindent This formula, proved in \cite{KW88} in the affine Lie algebra case, was conjectured in \cite{KRW03} to hold in the Lie superalgebra case. In all cases, considered in the present paper, \eqref{2.10} was proved in \cite{GK15}.
In particular, \eqref{2.10} holds in the case of the boundary level, i.e. when $m=0$ in \eqref{2.9}, which implies that $ch^{\pm}_{\Lambda^0}=1$,
hence \eqref{2.10} turns in a product.

 \section{Complementary integrable and admissible $ \wg $-modules, where $ \fg = spo_{2|3} $}
As in \cite{KW15}, Section 4, choose the set of simple roots of $ \fg = spo_{2|3} $ to be $ \Pi = \{ \al_1, \al_2 \}, $  where 
\[ (\al_1 | \al_1) = 0, \quad (\al_2|\al_2) = - \thalf, \quad (\al_1| \al_2) = \thalf. \]
We have:
\[ \theta = 2 (\al_1 + \al_2), \quad \rho = -\thalf \al_1, \quad h^\vee = \thalf, \quad \hat{\rho} = \thalf \La_0 + \rho, \quad \mbox{ defect } \fg = 1, \quad \sdim \fg = 0.  \]
The corresponding set of simple roots of $ \wg $ is $ \hat{\Pi} = \{ \al_0 = \delta - \theta, \al_1, \al_2 \}. $ 

Integrable $ \wg $-modules $ L(\La) $ have been studied in \cite{KW15}. In this section we study the complementary integrable $ \wg $-modules, namely the
$\wg$-modules $L(\Lambda)$, for which for roots 
$\alpha_2$ and  $\delta-\alpha_2$ are integrable. Since we are interested in modular invariant characters, we shall study only atypical such modules.  

Given $ m \in \CC,$  we denote by $ P^m_+ $ (resp. $ \dot{P}^m_+  $) the set $ \bmod \CC \delta $ of non-zero weights $ \La $ of level $ m $ such that $ L(\La)  $ is complementary integrable and $ (\La + \hat{\rho}\ | \ \al_1) =0 $ (resp. $ (\La + \hat{\rho} \ | \ \al_0 + \al_1) = 0$).
Let $n:=-4m-2$. Since any (complementary) integrable $\hat{\fg}$-module $L(\Lambda)$ of level 0 is trivial, and this case has been worked out in \cite{KW15}, we shall assume that $m\neq 0$. We studied these modules in \cite{KW16}, Section 6.4 (where they were called subprincipal integrable), using a different choice of simple roots, which is not suitable for the QHR.
\begin{proposition}
\label{prop3.1}
\begin{itemize}
\item[(a)] The set $ P^m_+ $ (resp. $ \dot{P}^m_+ $) is non-empty iff $ m \in \frac{1}{4} \ZZ_{\leq -2}.$
\item[(b)] Assume that $ m $ is non-critical, i.e. $ m \neq - \half,  $ so that, by (a), $ n $ is a positive integer. Let 
\[ \begin{aligned}
\La^{[m;m_2]} & = m \La_0 - \thalf m_2 \al_1, \\
\dot{\La}^{[m;m_2]} &= m \La_0 + (1+m) \al_1 - (m + \thalf m_2) (\al_1 + 2 \al_2).
\end{aligned} \]
\noindent Then
\[ P^m_+  (\mbox{resp. } \dot{P}^m_+) = \left\{ \La^{[m;m_2]} (\mbox{resp. } \dot{\La}^{[m;m_2]}) \ \big{|} \ m_2 \in \ZZ_{\geq 0}, m_2 \leq n \right\}.  \]
\item[(c)] Integrability of $ L(\La^{[m;m_2]})$ (resp. $L( \dot{\La}^{[m;m_2]}) $) with respect to $ \theta $ holds iff $ m_2 = 0 $ (resp. $ m_2 = n = 2 $). Integrability of $ L(\La^{[m;m_2]}) $ (resp. $ L(\dot{\La}^{[m;m_2]})$) with respect to $ \al_0 $  holds iff $ m + \half m_2 \in \ZZ_{\geq 0}$ (resp. $\in \ZZ_{>0}$).
\item[(d)] Let $ \La \in P^m_+ \cup \dot{P}^m_+. $ Then $ j_\La = 2 $ if 
$ L(\La) $ is integrable with respect to $ \theta $ and $ j_\La = 1 $ otherwise (recall that $ j_\La $ is the factor that appears in \eqref{2.6}).
\end{itemize}
\end{proposition}


\begin{proof}
  It is standard (see e.g. \cite{KW14}). First, it is easy to see that a $ \wg $-module $ L(\La) $ is complementary integrable iff $ -\al_2 $ and $ -(\delta - \al_2) $ are integrable. Next, integrability with respect to $ -\al_2 $ is equivalent to $ m_2 \in \ZZ_{\geq 0} $ since $ \al_2 $ is an even simple root. Using odd reflections, one can bring $ \delta - \al_2 $ to an even simple root $ \gamma, $ and integrability with respect to $ -\gamma $ is equivalent to $ m_2 \leq n. $ This proves (a) and (b). The proof of (c) is similar.
  Claim (d) follows from Proposition \ref{prop2.1}.
\end{proof}

As in \cite{KW15}, we use coordinates \eqref{2.1}, where $ z = -z_1 (\al_1 + 2 \al_2) - z_2 \al_1, $ so that $ (z|z) = 2 z_1 z_2. $ In these coordinates we have the following formula for the affine superdenominator (\ref{2.5}), see \cite{KW15}, formula (5.12) (there is a typo there):
\begin{equation}
\label{3.1}
\hat{R}^- (\tau, z_1, z_2, t) = -ie^{\pi i t} \frac{\eta (\tau)^3 \vartheta_{11} (\tau, z_1 + z_2) \vartheta_{11} (\tau, \frac{z_1-z_2}{2})}{\vartheta_{11} (\tau, z_1) \vartheta_{11}(\tau, z_2) \vartheta_{11} (\tau, \frac{z_1 + z_2}{2})},\end{equation}
\noindent where $ \eta(\tau) $ is the Dedekind eta function.

As for the RHS of \eqref{2.6} and \eqref{2.7}, we have: $ W^\#$ contains the subgroup 
$W_0^\#=\{ 1,r_{\al_2}\}, \ L^\# = 4 \ZZ \al_2, \ T = 
\{ \al_1\}$ (resp. $= \{\al_0+ \al_1\} $). 
In order to apply the supercharacter formula \eqref{2.6} to the complementary 
integrable $ \wg $-modules  $ \LLa, $ where $ \La = \La^{[m;m_2]} \in P^m_+ $ (resp. $ = \dot{\La}^{[m;m_2]} \in \dot{P}^m_+ $ ), we need to introduce the following functions:
\[ \begin{aligned}
F^{[m;m_2]} & = \sum_{j \in \ZZ} t_{4j \al_2} \frac{e^{\La^{[m;m_2]} + \wrh}}{1 - e^{-\al_1}},\\
\dot{F}^{[m;m_2]} & = q^{-(m + \frac{m_2}{2})} \sum_{j \in \ZZ} t_{4j \al_2} \frac{e^{\dot{\La}^{[m;m_2]} + \wrh}}{1 - e^{-\al_0-\al_1}}, \\
B^{[m;m_2]} & = F^{[m;m_2]} - r_{\al_2} F^{[m;m_2]}, \\
\dot{B}^{[m;m_2]} & = \dot{F}^{[m;m_2]} - r_{\al_2} \dot{F}^{[m;m_2]}. \\
\end{aligned} \]

After a direct calculation, we obtain
(cf. \eqref{1.8} and \eqref{1.9}): 
\begin{equation}
\label{3.2}
B^{[m;m_2]}(\tau, z_1, z_2, t) = 
e^{\pi i(2m+1)t} 
\Phi^{\left[-2m-1; \thalf(m_2+1)\right]} (2 \tau, z_1, -z_2). 
\end{equation}
\begin{equation}
\label{3.3}
\dot{B}^{[m;m_2]}(\tau, z_1, z_2, t) = e^{\pi i(2m+1)(t+z_2=z_1)} 
q^{-(m+\thalf)}\Phi^{\left[-2m-1; \thalf(m_2+1)\right]} (2 \tau, z_1+\tau, 
-z_2+\tau).
\end{equation}
Applying Lemma \ref{lemma1.1} to \eqref{3.2} and \eqref{3.3}, we obtain
\begin{proposition}
\label{prop3.2}
\[B^{[m;\,m_2]} \tzzt =  \thalf e^{\pi i(2m+1)t}\]
\[\times \left( \Phi^{[-4m-2;\, m_2+1]} (\tau, \thalf z_1, -\thalf z_2) -(-1)^{m_2} \Phi^{[-4m-2;\,m_2+1]} ( \tau, \thalf (z_1 +1), -\thalf (z_2 +1))  \right);\]
\[\dot{B}^{[m;\,m_2]} \tzzt = \thalf e^{\pi i(2m+1)t} q^{-(m+\thalf)}  e^{\pi i (2m+1) (z_2 - z_1)}\left( \Phi^{[-4m-2;\, m_2+1]} (\tau, \thalf (z_1 + \tau), -\thalf (z_2 - \tau)) \right. \] 
\[ \left. -(-1)^{m_2} \Phi^{[-4m-2;\, m_2+1]} \left( \tau, \thalf (z_1+\tau + 1), -\thalf (z_2-\tau+1) \right) \right). \]
\qed
\end{proposition}


We denote by $ \tilde{B}^{[m;m_2]} $ and $ \tilde{\dot{B}}\vphantom{B}^{[m;m_2]} $ the functions, obtained from the RHS of \eqref{3.2} and \eqref{3.3} by replacing the function $ \Phi $ by its modification $ \tilde{\Phi}. $

It is straightforward to check that Lemma \ref{lemma1.1} still holds if we replace $ \Phi $ by its modification $ \tilde{\Phi} $ on both sides of the equation. Applying this observation to Proposition \ref{prop3.2} and using the facts that $ \tilde{\Phi}^{[n;s]} $ is independent of $ s $ if $ s $ is an integer \cite{KW15}, Corollary 1.6, we obtain the following corollary. 
\begin{corollary}
\label{cor3.3}
\[\tilde{B}^{[m;\,m_2]} \tzzt =  \thalf e^{\pi i(2m+1)t}\]
\[\times \left( \tilde{\Phi}^{[-4m-2;\, 0]} (\tau, \thalf z_1, -\thalf z_2) -(-1)^{m_2} \tilde{\Phi}^{[-4m-2;\,0]} ( \tau, \thalf (z_1 +1), -\thalf (z_2 +1))  \right);\]
\[\tilde{\dot{B}}
^{[m;\,m_2]} \tzzt = \thalf e^{\pi i(2m+1)t} q^{-(m+\thalf)}  e^{\pi i (2m+1) (z_2 - z_1)}\]
\[\times\left(
\tilde {\Phi}^{[-4m-2;\, 0]} (\tau, \thalf (z_1 + \tau), -\thalf (z_2 - \tau)) 
-(-1)^{m_2} 
\tilde{\Phi}^{[-4m-2;\,0]} (\tau, \thalf (z_1+\tau + 1), -\thalf (z_2-\tau+1))\right).\]
\end{corollary}

Note that the $ \wg $-module $ \LLa$ for 
$\La\in P^m_+ $ (resp. $\in \dot{P}^m_+ $) is tame, since the root
$\alpha_1$ (resp. $\alpha_0 +\alpha_1$ is contained in 
$\hat{\Pi}$  (resp. in $r_{\alpha_1}\hat{\Pi}$).  
Now we can write down the normalized supercharacters $ \ch^-_\La  $ for all complementary integrable atypical 
$ \wg $-modules $ \LLa. $
\begin{proposition}
\label{prop3.4}
Let $ \La = \La^{[m;m_2]} \in P^m_+ $ (resp. $ \dot{\La} = 
\dot{\La}^{[m;m_2]} 
\in \dot{P}^m_+ $)
. Then we have the following supercharacter formulas:
\begin{itemize}
\item[(a)] $ \hat{R}^- \ch^-_\La = B^{[m;m_2]} $ (resp. $ \hat{R}^- \ch^-_{\dot{\La}} = \dot{B}^{[m;m_2]} $) if the roots $ \theta $ and $ \al_0 $ are both 
not integrable for $ \LLa $ (resp. $ L(\dot{\La}) $).
\item[(b)] $ \hat{R}^- \ch^-_\La = \frac{1}{j_\La} \left(B^{[m;m_2]} + B^{[m;-m_2]} \right)$ (resp. $ \hat{R}^- \ch^-_{\dot{\La}} = \frac{1}{j_\La} \left(\dot{B}^{[m;m_2]} + \dot{B}^{[m;-m_2]} \right)  $) if the root $ \theta  $ is integrable, but $ \al_0 $ is not for $ \LLa $ (resp. $ L(\dot{\La}) $), where $ j_\La $ (resp. $ j_{\dot{\La}} $) = 2 if $ m_2 = 0 $ and $ j_\La = 1 $ otherwise. 
\item[(c)] $ \hat{R}^- \ch^-_\La = B^{[m;m_2]} - \dot{B}^{[m;n-m_2]} $ (resp. $ \hat{R}^- \ch^-_{\dot{\La}} = \dot{B}^{[m;m_2]} - B^{[m;n-m_2]} $) if the root $ \al_0 $ is integrable but $ \theta $ is not for $ \LLa $ (resp. $ L(\dot{\La}) $).
\end{itemize}
\end{proposition}
\begin{proof}
First note that if a $ \wg $-module $ L(\la) $ is complementary integrable and is integrable with respect to $ \al_2, \delta - \al_2, \theta$ and $ \delta - \theta,  $ then its level $ m = 0,  $ which we excluded from consideration. Hence (a)--(c) cover all cases. Recall also that $ j_\La $ and $ j_{\dot{\La}} $ are given by Proposition \ref{prop3.1}(d).

In case (a) we may apply formula \eqref{2.6} with $ W^\# = \{ 1, r_{\al_2} \}, $ proving the supercharacter formula in this case. 

In case (b) we apply formula \eqref{2.6} with $ W^\# = \{ 1, r_{\al_2}, r_\theta, r_\theta r_{\al_2} \} $, hence we obtain:
\begin{equation}
\label{3.4}
\hat{R}^- \ch^-_\La = \frac{1}{j_\La} \left( B^{[m;m_2]} + r_\theta B^{[m;m_2]} \right).
\end{equation}
\noindent But in coordinates \eqref{2.1} with $ z = -z_1 (\al_1 + 2 \al_2) - z_2 \al_1 $ we have
\[ r_\theta \tzzt = (\tau, -z_2, -z_1, t). \]
\noindent Hence
\[ \begin{aligned}
r_\theta B^{[m;m_2]} \tzzt & = e^{\pi i (2m+1) t} \Phi^{[-2m-1; \frac{m_2+1}{2}]} (2\tau, -z_2, z_1) \\
& = e^{\pi i (2m+1)t} \Phi^{[-2m-1; 1 - \frac{m_2+1}{2}]} (2 \tau, z_1, -z_2) \\
& = B^{[m;-m_2]} \tzzt. \\
\end{aligned} \]
\noindent Hence (b) for $ \La $ follows from \eqref{3.4}. The proof of (b) for $ \dot{\La} $ is similar. 

In case (c) we apply formula \eqref{2.6} with $ W^\# = \{ 1, r_{\al_2}, r_{\al_0}, r_{\al_0} r_{\al_2} \} $. It is immediate to check that
\begin{equation}
\label{3.5}
r_{\al_0} (\La^{[m;m_2]} + \wrh) = \dot{\La}^{[m; n-m_2]} + \wrh \bmod{\CC \delta}. 
\end{equation}
\noindent Hence formula \eqref{2.6} gives:
\[ \hat{R}^- \ch^-_\La = B^{[m;m_2]} - r_{\al_0} B^{[m;m_2]} = B^{[m;m_2]} - \dot{B}^{[m;n-m_2]}. \]
\noindent The proof of (c) for $ \dot{\La} $ is the same. 
\end{proof}
As a corollary, we obtain formulas for modified normalized supercharacters $ \tilde{\ch}^-_\La $ for all complementary integrable atypical $ \wg $-modules $ \LLa $. By Proposition \ref{prop3.4}, all the $ \ch^-_\La  $ are expressed in terms of the functions $ B^{[m;m_2]} $ and $ \dot{B}^{[m;m_2]}, $ which, by Proposition \ref{prop3.2}, are in turn expressed in terms of the function $ \Phi^{[n;s]} $, where $ s \in \ZZ. $  By definition, $ \tilde{\ch}^-_\La $ is obtained from $ \ch^-_\La $ by replacing $ B^{m;m_2} $ and $ \dot{B}^{[m;m_2]} $
by
$ \tilde{B}\vphantom{B}^{[m;m_2]} $ and $ \tilde{\dot{B}}\vphantom{B}^{[m;m_2]} $ respectively, 
as in Corollary \ref{cor3.3}. Since the modified functions $ \tilde{\Phi}^{[n;s]}, s \in \ZZ, $ are independent of $s $ \cite{KW16}, Corollary 1.6, we obtain from Proposition \ref{prop3.4} and Corollary \ref{cor3.3} the following. 
\begin{corollary}
\label{cor3.5} 
Let $ \La $ and $ \dot{\La} $ be as in Proposition \ref{prop3.4}. Then
\begin{itemize}
\item[(a)] $ \hat{R}^- \tilde{\ch}^-_\La  = \tilde{B}^{[m;0]}$ (resp. $ \hat{R}^- \tilde{\ch}^-_{\dot{\La}}  = \tilde{\dot{B}}\vphantom{B}^{[m;0]}$ )  in case (a) of Proposition \ref{prop3.4}.
\item[(b)] $ \hat{R}^- \tilde{\ch}^-_\La  = \frac{2}{j_\La} \tilde{B}^{[m;0]}$ (resp. $ \hat{R}^- \tilde{\ch}^-_{\dot{\La}}  = \frac{2}{j_{\dot{\La}}} \tilde{\dot{B}}\vphantom{B}^{[m;0]}$) in case (b) of Proposition \ref{prop3.4}, where $ j_\La $ (resp. $ j_{\dot{\La}} ) = 2$ if $ m_2 = 0, $ and $= 1$ otherwise.
\item[(c)] $ \hat{R}^- \tilde{\ch}^-_\La  = -\hat{R}^- \,\tilde{\ch}^-_{\dot{\La}} = \tilde{B}^{[m;0]} - \tilde{\dot{B}}\vphantom{B}^{[m;0]}$  in case (c) of Proposition \ref{prop3.4}.
\end{itemize}
\end{corollary}

Next, we describe the principal admissible weights, associated to the complementary integrable weights from $ P^m_+ \cup \dot{P}^m_+. $ The principal simple subsets $ S $ of $ \hat{\Delta}_+ $ are as follows, where the $ k_i $'s are 
non-negative integers, such that all roots from $ S $ are positive , and $ M $ is a positive integer:
\begin{equation}
\label{3.6a}
 \begin{aligned}
\Pi^\mathrm{I}_{M;k_1, k_2} &=  \{k_0 \delta + \al_0, k_1 \delta + \al_1, k_2\delta + \al_2 \}, \ M = k_0 +2 (k_1 + k_2)+1, \\[1ex]
\Pi^\mathrm{II}_{M;k_1, k_2} & = \{k_0 \delta - \al_0, k_1 \delta - \al_1, k_2\delta - \al_2 \}, \ M = k_0 +2 (k_1 + k_2)-1, \\[1ex]
\Pi^\mathrm{III}_{M;k_1, k_2} & =  \{k_0 \delta + \al_0, k_1 \delta + \al_1 + 2\al_2, k_2\delta - \al_2 \}, \ M = k_0 +2 (k_1 + k_2)+1,\\[1ex]
\Pi^\mathrm{IV }_{M;k_1, k_2} & = \{k_0 \delta - \al_0, k_1 \delta - \al_1-2 \al_2, k_2\delta + \al_2 \}, \ M = k_0 +2 (k_1 + k_2)-1. \\
\end{aligned}
\end{equation}
(Formula for $M$ is equivalent to the condition that the sum of the $0$th simple root with twice the sum of the remaining two is equal to $M\delta$.)
The range of $(k_1,k_2)$ 
is obviously as follows:
\begin{equation}
  \label{3.6b}
  \begin{aligned}
    I\, (\hbox{resp.}\,III) : k_1\geq 0 ,\,k_2\geq 0\, 
(\hbox{resp.}\, \geq 1),\, 
2(k_1+k_2)\leq M-1; \\[1ex]
    II\, (\hbox{resp.}\, IV): k_1\geq 1,\, k_2\geq 1\, 
(\hbox{resp.}\, \geq 0),\, 
2(k_1+k_2)\leq M.\\
\end{aligned}
  \end{equation}
The following proposition is straightforward from definitions. 
\begin{proposition}
\label{prop3.6}
Let $ m \in \tfrac{1}{4} \ZZ_{< -2}, $ let $ M $ be a positive integer and let $ K $ be such that \eqref{2.9} holds, i.e. $ M\left(K + \half \right) = m + \thalf. $ Denote by $ \La^{[m; m_2]J}_{k_1, k_2} $ (resp. $ \dot{\La}^{[m; m_2]J}_{k_1, k_2} $) the admissible weight $ \La $ of level $ K $ such that $ \La^0 = \La^{[m;m_2]} $ (resp. $ \La^0 =\dot{\La}^{[m;m_2]} $) and $ \Pi_\La = \Pi^J_{M;k_1, k_2}, $ where $ J = $ I, II, III, or IV. Let $\tilde{K}=K+\half$. 
Then we have:
\[ 
\begin{aligned}
\La^{[m; m_2]\mathrm{I}}_{k_1, k_2}  =  & \ K\La_0 - \tilde{K} k_1 (\al_1 + 2 \al_2) -\left ( \tilde{K} (k_1 + 2k_2) + \frac{m_2 }{2}  \right) \al_1 \\
& - \left (\tilde{K} k_1 (k_1 + 2k_2) + \frac{m_2}{2} k_1\right) \delta ; \\[1ex]
\end{aligned} \]
\[ \begin{aligned}
\dot{\La}^{[m; m_2]\mathrm{I}}_{k_1, k_2} \! = & \,  K\La_0 \!-\! \left(
\tilde{K} k_1 \!+  m \!+ \! \frac{m_2}{2}\right) (\al_1 \!+ \! 2 \al_2) - \left(\tilde{K} (k_1 + 2k_2) -(m+1) \right) \al_1 \\
& -\left( \tilde{K}k_1 (k_1 +2k_2) + \left( m + \frac{m_2}{2}\right) (k_1 + 2 k_2) -(m+\half)k_1  \right) \delta; \\[1ex]
\La^{[m; m_2]\mathrm{II}}_{k_1, k_2}  = & \  K\La_0 + \tilde{K}k_1 (\al_1 + 2 \al_2) + \left(  \tilde{K} (k_1+2k_2) + \frac{m_2}{2}+1 \right) \al_1 \\
& - \left( \tilde{K} k_1 (k_1 + 2k_2) + \frac{m_2+1}{2} k_1 \right) \delta; \\[1ex]
\dot{\La}^{[m; m_2]\mathrm{II}}_{k_1, k_2}  = & \  K\La_0 + 
\left( \tilde{K}k_1 + \left(m + \frac{m_2}{2}\right) \right) (\al_1 + 2 \al_2) 
+ \left(\tilde{K} (k_1 + 2k_2) - m\right) \al_1 \\
&- \left(\tilde{K}k_1 (k_1 + 2k_2) - (m + \half) k_1 + \left( m + \frac{m_2}{2} \right)(k_1 + 2k_2)  \right)\delta   ; \\
\La^{[m; m_2]\mathrm{III}}_{k_1, k_2}  = & \  K\La_0 - \left( \tilde{K}(k_1 + 2k_2) + \frac{m_2+1}{2} \right) (\al_1+2\al_2)  - \left(\tilde{K} k_1 - \half \right) \al_1\\
& -\left( \tilde{K} k_1 (k_1 + 2k_2) + \frac{m_2+1}{2} k_1  \right) \delta; \\[1ex]
\dot{\La}^{[m; m_2]\mathrm{III}}_{k_1, k_2}  = & \  K\La_0  -  
\left( \tilde{K}(k_1 + 2k_2) - \left( m + \half \right) \right) (\al_1 + 2 \al_2)  - \left( \tilde{K} k_1 \! + \! \left( m + \frac{m_2}{2} - \half \right) \right)\al_1 \\
& - \left( \tilde{K}k_1 (k_1 + 2k_2) - \left( m + \half \right)k_1 +
\left(m + \frac{m_2}{2}) (k_1 + 2 k_2 \right) \right) \delta;  \\[1ex]
\La^{[m; m_2]\mathrm{IV}}_{k_1, k_2}  = & \  K\La_0 + \left( \tilde{K} (k_1 + 2k_2) + \frac{m_2+1}{2}\right) (\al_1 +2 \al_2) + \left( \tilde{K}k_1 + \half \right) \al_1\\
& - \left( \tilde{K} k_1 (k_1 + 2k_2) + \frac{m_2 +1 }{2} k_1\right) \delta; \\[1ex]
\dot{\La}^{[m; m_2]\mathrm{IV}}_{k_1, k_2}  = & \  K\La_0 + \left( \tilde{K}
 (k_1 + 2k_2) - \left( m + \half \right)\right) (\al_1 + 2\al_2) 
+ \left(\tilde{K}k_1 + \left( m + \frac{m_2}{2} + \half \right)\right) \al_1 \\
& - \left( \tilde{K} k_1 (k_1 + 2k_2) - \left( m + \half \right)k_1 + \left( m + \frac{m_2}{2}\right) (k_1 + 2k_2)\right) \delta. \\
\end{aligned}
 \]
\qed
\end{proposition}

In view of formula \eqref{2.10} and Proposition \ref{prop3.4}, let, for $ J = $ I or III, 
\[ B^{[m;m_2]J}_{k_1, k_2} (\tau, z, t) = B^{[m;m_2]} \left( M\tau, y^{-1} (z+\tau \beta), \tfrac{1}{M} (t + (z | \beta) + \thalf \tau | \beta|^2) \right), \]
\noindent 
where 
\begin{equation}
\label{3.6}
\begin{split}
y &= 1, \,\beta = -k_1 (\al_1 + 2 \al_2) - (k_1 + 2k_2) \al_1 \mbox{ if } J = \mathrm{I}, \\
y & = r_{\al_2}r_{\al_1+\al_2},\, \beta = 
 k_1 (\al_1 + 2 \al_2) +(k_1+2k_2) \al_1 \mbox{ if } J = \mathrm{II},\\
y & = r_{\al_2},\, \beta = - (k_1 + 2k_2) (\al_1 + 2 \al_2) -k_1 \al_1 \mbox{ if } J = \mathrm{III},\\
y & = r_{\al_1+\al_2},\, \beta =  (k_1 + 2k_2) (\al_1 + 2 \al_2)
+ k_1 \al_2 \mbox{ if } J = \mathrm{IV},
\end{split}
\end{equation}
\noindent and similarly for $ \dot{B}. $
Note that in all 
cases we have: $\beta=-y(k_1(\al_1+2\al_2)+(k_1+2k_2)\al_1)$.

The supercharacter formulas for all these principal admissible modules can be deduced from Proposition \ref{prop3.4} for the complementary integrable $ \wg $-modules by making use of the general formula \eqref{2.10}. We shall give here the results only in cases I and III, since, as we shall see, these are sufficient for the quantum Hamiltonian reduction. Due to Proposition \ref{prop3.4} 
all these supercharacters are expressed in terms of functions
$B^{[m;m_2]J}_{k_1,k_2}$ and $\dot{B}^{[m;m_2]J}_{k_1,k_2}$, which in turn are expressed
in terms of functions $ \Psi^{[M; n; s; \epsilon]}_{a,b; \epsilon'}, $ given by \eqref{1.10}:
\begin{proposition}
\label{prop3.7}
In case $J=I$ we have:
\begin{enumerate}
\item[I(a)] 
$ B^{[m; m_2]\mathrm{I}}_{k_1, k_2}\tzzt = \half e^{\frac{\pi i}{M}(2m+1)t}$ 
  \\[1ex] $ \times \left( \Psi^{[M; -4m-2; m_2+1; 0]}_{\frac{k_1}{2}, - \frac{k_1 + 2k_2}{2}; \epsilon'} \left(\tau, \frac{z_1}{2}, - \frac{z_2}{2}\right)  -(-1)^{m_2} \Psi^{[M; -4m-2; m_2+1; 0]}_{\frac{k_1}{2}, - \frac{k_1 + 2k_2}{2}; \epsilon'}\left(\tau, \frac{z_1}{2}, - \frac{z_2}{2}\right) \right), $\\
\noindent where $ \epsilon'=0 $ if $ \half k_1 \in \ZZ, $ and $ = \half $ otherwise. 
\item[I(b)] 
$\dot{B}^{[m; m_2]\mathrm{I}}_{k_1, k_2}\tzzt = \half 
e^{\frac{\pi i}{M}(2m+1)t}$ 
\\[1ex] $ \times  \left( \Psi^{[M; -4m-2; m_2+1; 0]}_{\frac{k_1+M}{2}, - \frac{k_1 + 2k_2-M}{2}; \epsilon'} \left( \tau, \frac{z_1}{2}, - \frac{z_2}{2}\right)  -(-1)^{m_2} \Psi^{[M; -4m-2; m_2+1; \half]}_{\frac{k_1+M}{2}, - \frac{k_1 + 2 k_2 - M}{2}; \epsilon'} \left( \tau, \frac{z_1}{2}, -\frac{z_2}{2}\right)\right), $ \\
\noindent where $ \epsilon' =0 $ if $ \half (k_1 + M) \in \ZZ, $ and $ = \half $ otherwise. 
\end{enumerate}
In case $J=III$ we have:
\begin{enumerate}
\item[III(a)] 
$ B^{[m; m_2]\mathrm{III}}_{k_1, k_2} \tzzt =
-\half e^{\frac{\pi i}{M}(2m+1)t}$ 
\\[1ex]$ \times \left( \Psi^{[M; -4m-2; m_2+1; 0]}_{\frac{k_1 +2 k_2}{2}, - \frac{k_1}{2}; \epsilon'} \left(\tau, \frac{z_1}{2}, - \frac{z_2}{2}\right) -(-1)^{m_2} \Psi^{[M; -4m-2; m_2+1; \half]}_{\frac{k_1+2k_2}{2}, - \frac{k_1}{2}; \epsilon'} \left( \tau, \frac{z_1}{2}, -\frac{z_2}{2}\right) \right), $ \\
\noindent where $ \epsilon' $ is as in I(a).
\item[III(b)] 
$\dot{ B}^{[m; m_2]\mathrm{III}}_{k_1, k_2} \tzzt =-\half e^{\frac{\pi i}{M}(2m+1)t}$ \\[1ex]
$ \times \left( \Psi^{[M;-4m-2; m_2+1; 0]}_{\frac{k_1+ 2k_2 -M}{2}, -\frac{k_1+M}{2}; \epsilon'}    \left(\tau, \frac{z_1}{2}, - \frac{z_2}{2}\right) -(-1)^{m_2} \Psi^{[M;-4m-2; m_2+1; \half]}_{\frac{k_1+ 2k_2 -M}{2}, -\frac{k_1+M}{2}; \epsilon'} \left(\tau, \frac{z_1}{2}, -\frac{z_2}{2}\right) \right)$, \\
\noindent where $ \epsilon' $ is as in I(b).
\end{enumerate}
\end{proposition}
\begin{proof}
A straightforward computation, using (\ref{3.6}), and that
$ |\beta|^2 = 2k_1 (k_1 + 2k_2) $ in both cases, and $ (z|\beta) = (k_1 + 2k_2)z_1 + k_1 z_2 $ in case I, $ (z|\beta) = k_1z_1 + (k_1 + 2k_2) z_2 $ in case III. 
\end{proof}

As explained in \cite{KW16}, Section 6.6, the formulas for the normalized character $ \ch^+_\La $ and the twisted (super)characters $ \ch^{\mathrm{tw}\pm}_\La $ are obtained from the formula for the normalized supercharacter by simple substitutions:
\begin{equation}
\label{3.7}
\ch^+_\La \tzzt = e^{-\pi i (\La|\theta)} \ch^-_\La (\tau, z_1 - \half, z_2 - \half, t), 
\end{equation}
\begin{equation}
\label{3.8}
\ch^{\pm,\tw}_\La(\tau, z_1, z_2,t) = \ch^\pm_\La \left(\tau, - z_2 + \frac{\tau}{2}, -z_1 + \frac{\tau}{2}, t + \frac{\tau}{4} - \frac{z_1 + z_2}{2} \right).
\end{equation}


A weight $ \La \in \wh^* $ is called \emph{degenerate} if $ (\La  |  \al_0 ) \in \ZZ_{\geq 0}$ (equivalently, $ (\La + \hat{\rho}  |  \al_0) \in \zp $). In the next section we will need the following description of degenerate weights. 
\begin{lemma}
\label{lem3.8}
Let $ M $ be a positive odd integer and assume that $ n:=-4m-2 $ is a positive integer, coprime to $ M.  $ Let $ J = $ I or III. Then
\begin{enumerate}
\item[(a)] The principal admissible weight $ \La^{[m;m_2]J}_{k_1,k_2} $ is degenerate iff
\[ k_1 + k_2 = \half (M-1) \mbox{ and } 
2m+m_2 \in 2\ZZ_{\geq 0}.\]
In particular, $\Lambda^{[m;m_2]}\in P^m_+$ is non-degenerate iff $2m+m_2\in 2\ZZ_{\geq 0}$.
\item[(b)] The principal admissible weight $ \dot{\La}^{[m;m_2]J}_{k_1,k_2} $ is degenerate iff
\[ k_1 + k_2 = \half (M-1) \mbox{ and } 
 2m+m_2 \in 2\ZZ_{> 0}.\]
In particular, $\dot{\Lambda}^{[m;m_2]}\in \dot{P}^m_+$ is non-degenerate iff $2m+m_2\in 2\ZZ_{> 0}$.
\end{enumerate}
\end{lemma}
\begin{proof}
Consider the simple subset $ \Pi^J_{M;k_1, k_2} $, where $ J =  $ I or III. Then for the corresponding principal admissible weight $ \La = \La^{[m;m_2]J}_{k_1,k_2} $ (resp. $ \dot{\La}^{[m;m_2]J}_{k_1,k_2} $) we have:
\[ \begin{aligned}
(\La + \wrh  |  \al_0) = & \frac{m_2 +1}{2} (\mbox{resp. } \frac{m_2-1}{2}) - \frac{n(M-k_0)}{4M} \mbox{ if } J = \mbox{I or III}. \\
\end{aligned}
 \]
 The proof follows. 
\end{proof}

\section{A new modular invariant family of $ N=3 $ representations.}

Throughout this section again $ \fg = spo_{3|2} $ and we use notation of the previous section; by $N=3$ superconformal algebras we mean the vertex algebra
$W^K(\fg, e_{-\theta})$ and its Ramond twisted analogue, cf.\cite{KW04},
\cite{KW05}. Recall that the quantum Hamiltonian reduction associates to a $ \wg $-module $ L(\La) $ of level $ K \neq - \half $ a module $ H(\La) $ over the $ N= 3 $ superconformal algebra of Neveu-Schwarz type, such that the following properties hold (see \cite{KW15} for more details and further references):
\begin{enumerate}
\item[(i)] the module $ H(\La) $ is either 0 or an irreducible positive energy module;
\item[(ii)] $ H(\La) = 0 $ iff $ (\La| \al_0) \in \ZZ_{\geq 0}$, i.e. $\Lambda$
is non-degenerate;
\item[(iii)] the irreducible module $ H(\La) $ is characterized by three numbers:
\begin{enumerate}
\item[($ \al $)] the central charge (cf. \eqref{0.1}): 
\begin{equation}
\label{4.1}
c_K = -6K-\frac{7}{2},
\end{equation}
\item[($ \beta $)] the lowest energy (the minimal eigenvalue of $L_0$)
\begin{equation}
\label{4.2}
h_\La = \frac{(\La + 2 \hat{\rho}  |  \La)}{2K+1} - (\thalf \theta + \La_0 |\La ),
\end{equation}
\item[($ \gamma $)] the spin (the value of $J_0=\half \alpha_2^\vee=-2\alpha_2$
on $v_{\Lambda}$)
\begin{equation}
\label{4.3}
s_\La = -2 (\La  |  \al_2); 
\end{equation}
\end{enumerate}  
\item[(iv)] the character $ \ch^+ $ and the supercharacter $ \ch^- $ of the module $ H(\La) $ are given by the following formula $ (t, z \in \CC, \Im \tau > 0): $
\begin{multline}
\label{4.4}
\ch^\pm_{H(\La)} (\tau, z):=  \tr^\pm_{H(\La)} q^{L_0 - \frac{c_K}{24}} e^{-4\pi i z \al_2}   
=   (\hat{R}^\pm \ch^\pm_\La) (\tau, z + \frac{\tau}{2}, -z + \frac{\tau }{2}, \frac{\tau}{4}) \overset{3}{R}\vphantom{R}^\pm (\tau, z)^{-1},
\end{multline}
\noindent where 
\begin{equation}
\label{4.5}
\overset{3}{R}\vphantom{R}^+ (\tau, z) = \frac{\eta (\frac{\tau}{2})\eta (2\tau) \vartheta_{11} (\tau,z)}{\vartheta_{00} (\tau, z)}, \quad 
\overset{3}{R}\vphantom{R}^- (\tau, z) = \frac{\eta (\tau)^3  \vartheta_{11} (\tau,z)}{\eta (\frac{\tau}{2}) \vartheta_{01} (\tau, z)}.
\end{equation}
\end{enumerate}

In order to achieve modular invariance we need also the $ N=3 $ Ramond type $ N=3 $ superconformal algebra modules. This has been discussed in \cite{KW15} as well, which we briefly recall. We should consider, associated to a $ \wg $-module $ \LLa, $ a $ \wg^{\tw} $-module $ L^{\tw} (\La), $ whose relevant features are as follows \cite{KW15}, Section 6.

We twist $ \wg $ by the element $ w = t_{\half \theta} r_\theta, $ so that the set of simple roots $ \hat{\Pi}^{\tw} $ and $ \hat{\rho}^{\tw} $ of $ \wg^{\tw}$ are:
\begin{equation}
\label{4.6}
\hat{\Pi}^{\tw} =  \left\{ -\half \delta + \al_1 + 2\al_2, \half \delta + \al_1, \half (\delta - \theta)   \right\} ; \ \hat{\rho}^{\tw} = \hat{\rho} - \frac{3}{2} (\al_1 + \al_2) - \frac{3}{8} \delta. 
\end{equation}
Furthermore, a $ \wg $-module $ \LLa $ is an irreducible $ \wg^{\tw} $-module, denoted by $ L^{\tw} (\La),$ whose highest weight is
\begin{equation}
\label{4.7}
\La^{\tw} = w(\La),
\end{equation}
\noindent and the normalized (super)character and (super)denominator are
\begin{equation}
\label{4.8}
\ch^{\pm,\tw}_\La = w(\ch^{\pm}_\La), \, \hat{R}^{\pm,\tw} = w(\hat{R}^{\pm}).
\end{equation}
In coordinates (\ref{2.1})  with $z=-z_1(\alpha_1 +2\alpha_2)-z_2\alpha_1$
we have
\begin{equation}
\label{4.9}
w(h) = \left( \tau, -z_2 + \frac{\tau}{2}, -z_1 + \frac{\tau}{2}, t - \frac{z_1 + z_2}{2} + \frac{\tau}{4} \right).
\end{equation}

The irreducible Ramond $ N =3 $ modules $ H^{\tw} (\La), $ obtained by the quantum Hamiltonian reduction from the $ \wg^{\tw} $-modules $ L^{\tw} (\La) $,
satisfy (i) and (ii) with $H$ replaced by $H^{\tw}$, and are again characterized by three numbers: the central charge $ c_K,  $ given by \eqref{4.1}, the lowest energy
\begin{equation}
\label{4.10}
h^{\tw}_\La = \frac{(\La^{\tw} + 2 \hat{\rho}^\tw | \La^\tw )}{2K+1} - \left(\half \theta + \La_0 | \La^{\tw} \right) - \frac{3}{16},
\end{equation}
\noindent and the spin
\begin{equation}
\label{4.11}
s^{\tw}_\La = s_\La - \half.
\end{equation}

The (super)character of the module $ H^\tw (\La),  $ defined as 
\[ \ch^\pm_{H^\tw (\La)} (\tau, z) := \tr^\pm_{H^\tw (\La)} q^{L^\tw_0 -\frac{c_K}{24}} e^{-4\pi i z \al_2},  \]
\noindent is as follows:
\begin{equation}
\label{4.12}
\ch^-_{H^\tw (\La)} = 0, \quad 
(\overset{3}{R}\vphantom{R}^{\tw} 
\ch^+_{H^\tw (\La)}) (\tau, z) = (\hat{R}^{+,\tw} \ch^{+,\tw}_\La) \left( \tau, z+ \frac{\tau}{2}, - z + \frac{\tau}{2}, \frac{\tau}{4}\right),
\end{equation}
\noindent where
\begin{equation}
\label{4.13}
\overset{3}{R}\vphantom{R}^{\tw} (\tau, z) = \frac{\eta (\tau)^3 \vartheta_{11} (\tau, z)}{\eta (2 \tau) \vartheta_{10} (\tau, z)}
\end{equation}
\noindent is the Ramond $ N= 3 $ superconformal algebra denominator. 

Next, we  compute the characteristic numbers $ h_\La $ and $ s_\La $ for all $ N=3  $ Neveu-Schwarz
modules $ H(\La), $ and $ N= 3 $ Ramond modules $ H^\tw (\La),  $ where $ \La $ are principal admissible weights.
\begin{lemma}
\label{lem4.1}
Let $ \La^J = \La^{[m;m_2]J}_{k_1, k_2} $ and $ \dot{\La}^J = \dot{\La}^{[m;m_2]J}_{k_1, k_2} $, where $J=I-IV$. Then
\begin{enumerate}
\item[(a)] \[ \begin{aligned}
h_{\La^\mathrm{I}} =  h_{\La^\mathrm{III}} = & \ \frac{m_2+1}{2} 
\left(k_1+\half \right) -\frac{1}{4} + \frac{2m+1}{2M} 
(k_1 (k_1 + 2k_2) + k_1+k_2), \\
h_{\La^\mathrm{II}} =  h_{\La^\mathrm{IV}} = & \ \frac{m_2+1}{2} \left(k_1-\half \right) -\frac{1}{4} + \frac{2m+1}{2M} 
(k_1( k_1 + 2k_2) - k_1-k_2), \\
h_{\dot{\La}^\mathrm{I}} =  h_{\dot{\La}^\mathrm{III}} = & \ \frac{m_2-1}{2}\left(k_1 + \half \right) + \left(m+ \frac{m_2}{2}\right) (2k_2 - M) -\frac{1}{4}\\ 
&+ \frac{2m+1}{2M} (k_1( k_1 + 2k_2) + k_1+k_2), \\
h_{\dot{\La}^\mathrm{II}} =   h_{\dot{\La}^\mathrm{IV}} = & \ \frac{m_2-1}{2}\left(k_1 - \half \right) + \left(m+ \frac{m_2}{2}\right) (2k_2 - M) -\frac{1}{4}\\ 
&+ \frac{2m+1}{2M}(k_1( k_1 + 2k_2) - k_1-k_2). \\
\end{aligned} \]
\item[(b)] \[ \begin{aligned}
s_{\La^\mathrm{I}} = s_{\La^\mathrm{IV}} &=\half m_2 + \frac{(2m+1)}{M} k_2, \\
s_{\La^\mathrm{II}} = s_{\La^\mathrm{III}} &= -\half m_2 - \frac{(2m+1)}{M} k_2 - 1,\\
s_{\dot{\La}^\mathrm{I}} = s_{\dot{\La}^\mathrm{IV}} &= -\half m_2 - 2m + \frac{(2m+1)}{M}k_2 - 1,\\
s_{\dot{\La}^\mathrm{II}} = s_{\dot{\La}^\mathrm{III}} &= \half m_2 + 2m - \frac{(2m+1)}{M} k_2.\\
\end{aligned}
 \]
\item[(c)]\[h^\tw_{\La^{[m;m_2]J}_{m_1, m_2}} =  \frac{m_2+1}{2} k_1 + \frac{2m+1}{2M} \left( k_1 (k_1+2k_2) - \frac{1}{4}\right) - \frac{1}{16},\] 
\[ h^\tw_{\dot{\La}^{[m;m_2]J}_{m_1, m_2}} =  \frac{m_2-1}{2} k_1 + \left(m + \frac{m_2}{2}\right)(2k_2 -M) + \frac{2m+1}{2M} \left( k_1 (k_1+2k_2) - \frac{1}{4}\right) - \frac{1}{16}.  \]
\item[(d)]
  $ s^\tw_{\La^J} = s_{\La^J} - \half $ and $ s^\tw_{\dot{\La}^J} = s_{\dot{\La}^J} - \half. $
\end{enumerate}
\end{lemma}

Since the $ N=3 $ irreducible positive energy modules are completely determined by their characteristic numbers, Lemma \ref{lem4.1} implies the following corollary, which explains why we need to consider only $ \wg $-modules of type I and III. 
\begin{corollary}
\label{cor4.2}
We have the following isomorphisms  of $ N=3 $ Neveu-Schwarz and Ramond 
modules:
\begin{enumerate}
\item[(a)] \[ \begin{aligned}
H \left(\La^{[m;m_2]\mathrm{I}}_{k_1-1,k_2} \right) \simeq H \left(\La^{[m;m_2]\mathrm{IV}}_{k_1,k_2} \right), & \quad H \left(\La^{[m;m_2]\mathrm{III}}_{k_1-1,k_2} \right) \simeq H \left(\La^{[m;m_2]\mathrm{II}}_{k_1,k_2} \right), \\
H \left(\dot{\La}^{[m;m_2]\mathrm{I}}_{k_1-1,k_2} \right) \simeq H \left(\dot{\La}^{[m;m_2]\mathrm{IV}}_{k_1,k_2} \right), & \quad H \left(\dot{\La}^{[m;m_2]\mathrm{III}}_{k_1-1,k_2} \right) \simeq H \left(\dot{\La}^{[m;m_2]\mathrm{II}}_{k_1,k_2} \right). \\
\end{aligned}
 \]
\item[(b)]\[ \begin{aligned}
H^\tw \left(\La^{[m;m_2]\mathrm{I}}_{k_1,k_2} \right) \simeq H^\tw \left(\La^{[m;m_2]\mathrm{IV}}_{k_1,k_2} \right), & \quad H^\tw \left(\La^{[m;m_2]\mathrm{III}}_{k_1,k_2} \right) \simeq H^\tw \left(\La^{[m;m_2]\mathrm{II}}_{k_1,k_2} \right), \\
H^\tw \left(\dot{\La}^{[m;m_2]\mathrm{I}}_{k_1,k_2} \right) \simeq H^\tw \left(\dot{\La}^{[m;m_2]\mathrm{IV}}_{k_1,k_2} \right), & \quad H^\tw \left(\dot{\La}^{[m;m_2]\mathrm{III}}_{k_1,k_2} \right) \simeq H^\tw \left(\dot{\La}^{[m;m_2]\mathrm{II}}_{k_1,k_2} \right). \\
\end{aligned}
 \]
\end{enumerate}
\end{corollary}

Using Proposition \ref{prop3.2}, formulas \eqref{4.4} and \eqref{4.11} give explicit character and supercharacter formulas for Neveu-Schwarz sector and character formulas for the Ramond sector of the $ N=3 $ superconformal algebra (recall that the supercharacters in the Ramond sector are always 0).

Recall that, by Proposition \ref{prop3.4}, the numerators of the normalized supercharacters of complementary integrable $ \wg $-modules are simple linear combinations of the functions $ B^{[m;m_2]} $ and $ \dot{B}^{[m;m_2]} $, expressed, by Proposition \ref{prop3.2}, via mock theta functions $ \Phi^{[-4m-2; m_2+1]} $ (given by \eqref{1.9}). Furthermore, the corresponding principal admissible $ \wg $-modules are the same linear combinations of the functions $ B^{[m;m_2]J}_{k_1,k_2} $ and $ \dot{B}^{[m;m_2]J}_{k_1,k_2}, $ where $ J = $ I--IV, which, by Proposition \ref{prop3.7}, are expressed via the functions $ \Psi $ (given by \eqref{1.10}).
Hence, by \eqref{4.4} and \eqref{4.12}, we obtain the following theorem.
\begin{theorem}
\label{th4.3}
Let $ \La^j = \La^{[m;m_2]J}_{k_1,k_2} $ and $ \dot{\La}^j = \La^{[m;m_2]J}_{k_1,k_2} $ for $ J =  $ I or III.
Then the numerators of $ N=3 $ (super)characters $ \ch^\pm $ and of Ramond twisted $ N = 3 $ characters $ \ch^\tw $ are the same linear combinations, as in Proposition \ref{prop3.4} of the numerators of complementary integrable $ \wg $-modules, of the 
functions $ H^{\pm[m;m_2]J}_{k_1,k_2},  H^{\tw[m;m_2]J}_{k_1,k_2},$ and the similar functions with overdots, which are given by the following formulas. 
\begin{enumerate}
\item[(a)]\[ 
\begin{aligned}
H^{-[m;m_2]\mathrm{I}}_{k_1,k_2}
\tz = & \ \half q^{\frac{2m+1}{8M}} \left( \Psi^{[M,-4m-2;m_2+1;0]}_{\frac{k_1}{2}, - \frac{k_1+2k_2}{2}; \epsilon'} ( \tau, \frac{2z+\tau}{4}, \frac{2z-\tau}{4}
) \right. \\
& \left. - (-1)^{m_2} \Psi^{[M,-4m-2; m_2+1; \half]}_{\frac{k_1}{2}, -\frac{k_1+2k_2}{2}; \epsilon'} ( \tau, \frac{2z+\tau}{4}, \frac{2z-\tau}{4}) \right), \\
\dot{H}^{-[m;m_2]\mathrm{I}}_{k_1,k_2}
\tz = & \ \half q^{\frac{2m+1}{8M}} \left( \Psi^{[M, -4m-2; m_2+1; 0]}_{\frac{k_1+M}{2}, -\frac{k_1+2k_2-M}{2}; \epsilon'}  ( \tau, \frac{2z+ \tau}{4}, \frac{2z-\tau}{4}) \right. \\  & \left. - (-1)^{m_2} \Psi^{[M, -4m-2; m_2+1; \half]}_{\frac{k_1+M}{2}, -\frac{k_1+2k_2-M}{2}; \epsilon'} ( \tau, \frac{2z+ \tau}{4}, \frac{2z-\tau}{4} ) \right), \\
\end{aligned}\] \[
\begin{aligned}
H^{-[m;m_2]\mathrm{III}}_{k_1,k_2} \tz = & \  - \half q^{\frac{2m+1}{8M}} \left( \Psi^{[M, -4m-2; m_2+1; 0]}_{\frac{k_1+2k_2}{2}, - \frac{k_1}{2}; \epsilon'} (\tau, \frac{2z+\tau}{4}, 
\frac{2z-\tau}{4}  )      \right. \\
& \left. -(-1)^{m_2} \Psi^{[M, -4m-2; m_2+1; \half]}_{\frac{k_1+2k_2}{2}, -\frac{k_1}{2}; \epsilon'} (\tau, \frac{2z+\tau}{4}, \frac{2z-\tau}{4} ) \right),  \\
\dot{H}^{-[m;m_2]\mathrm{III}}_{k_1,k_2}\tz = & - \half q^{\frac{2m+1}{8M}}  \left( \Psi^{[M, -4m-2; m_2+1; 0}_{\frac{k_1+2k_2-M}{2}, -\frac{k_1+M}{2}; \epsilon'} ( \tau, \frac{2z+\tau}
{4}, \frac{2z-\tau}{4} )\right. \\
& \left.  -(-1)^{m_2} \Psi^{[M, -4m-2; m_2+1; \half]}_{\frac{k_1+2k_2-M}{2}, -\frac{k_1+M}{2}; \epsilon'} ( \tau, \frac{2z+\tau}{4}, \frac{2z-\tau}{4}) \right) . \\
\end{aligned} \]
\item[(b)] \[ \begin{aligned}
& 
H^{+[m;m_2]\mathrm{I}}_{k_1,k_2}\tz =  \\
& \ -\half (-i)^{m_2+1} e^{-\frac{\pi i }{M} (4m+2)(k_1+k_2)} q^{\frac{2m+1}{8M}}  \left( \Psi^{[M, -4m-2; m_2+1; 0]}_{\frac{k_1}{2}, -\frac{k_1+2k_2}{2}; \epsilon'} (\tau, \frac{2z+ \tau+1}{4}, \frac{2z-\tau-1}{4} )\right. \\
& \left. -(-1)^{m_2} \Psi^{[M, -4m-2; m_2+1; \half]}_{\frac{k_1}{2}, -\frac{k_1+2k_2}{2}; \epsilon'} (\tau, \frac{2z+\tau +1}{2}, \frac{2z-\tau -1}{4} ) \right), \\
&
\dot{H}^{+[m;m_2]\mathrm{I}}_{k_1,k_2}\tz =  \  \\
&-\half (-i)^{m_2-1} e^{-\frac{\pi i }{M} (4m+2)(k_1+k_2)} q^{\frac{2m+1}{8M}} \left( \Psi^{[M, -4m-2; m_2+1; 0]}_{\frac{k_1+M}{2}, -\frac{k_1+2k_2-M}{2}; \epsilon'} (\tau, \frac{2z+\tau + 1}{4}, \frac{2z-\tau - 1}{4} )\right. \\
& \left. -(-1)^{m_2} \Psi^{[M, -4m-2; m_2+1; \half]}_{\frac{k_1+M}{2}, -\frac{k_1 + 2k_2 -M}{2}; \epsilon'} (\tau, \frac{2z+\tau + 1}{4}, \frac{2z-\tau-1}{4} 
) \right), \\[2ex]
&
H^{+[m;m_2]\mathrm{III}}_{k_1,k_2}\tz = \\
&  
\half (-i)^{m_2+1} e^{-\frac{\pi i }{M} (4m+2)(k_1+k_2)} q^{\frac{2m+1}{8M}} \left( \Psi^{[M, -4m-2; m_2+1; 0]}_{\frac{k_1+2k_2}{2}, -\frac{k_1}{2}; \epsilon'} (\tau, \frac{2z+\tau + 1}{4}, \frac{2z-\tau-1}{4} )\right. \\
&  \left. -(-1)^{m_2} \Psi^{[M, -4m-2; m_2+1; \half]}_{\frac{k_1+2k_2}{2}, -\frac{k_1}{2}; \epsilon'} (\tau, \frac{2z+\tau+1}{4}, \frac{2z-\tau-1}{4} ) \right), \\[2ex]
& 
\dot{H}^{+[m;m_2]\mathrm{III}}_{k_1,k_2} \tz = \\
&  \half (-i)^{m_2-1} e^{-\frac{\pi i }{M} (4m+2)(k_1+k_2)} q^{\frac{2m+1}{8M}}\left( \Psi^{[M,-4m-2; m_2+1;0]}_{\frac{k_1+2k_2-M}{2}, - \frac{k_1+M}{2}; \epsilon'} (\tau, \frac{2z+\tau -1}{4}, \frac{2z-\tau -1}{4} )\right. \\
& \left.  -(-1)^{m_2} \Psi^{[M,-4m -2; m_2+1; \half]}_{\frac{k_1+2k_2-M}{2}, - \frac{k_1+M}{2}; \epsilon'} (\tau, \frac{2z+\tau -1}{4}, \frac{2z-\tau -1}{4}  ) \right). \\
\end{aligned} \]
\item[(c)]\[ \begin{aligned}
  H^{\tw [m;m_2]\mathrm{I}}_{k_1,k_2} \tz = 
  & \ \half(-i)^{m_2+1}e^{-\frac{\pi i}{M}(2m+1)(k_1+k_2)}
  \left(\Psi^{[M;-4m-2;-m_2;0]}_{\frac{k_1}{2},-\frac{k_1+2k_2}{2}; \epsilon'} (\tau, \frac{2z+1}{4}, \frac{2z-1}{4} ) \right. \\
& \left.  -(-1)^{m_2}  \Psi^{[M, -4m-2(m+\half); -m_2; \half]}_{\frac{k_1}{2},-\frac{k_1 + 2k_2}{2}; \epsilon'} (\tau, \frac{2z+1}{4}, \frac{2z-1}{4}) \right), \\
\end{aligned} \]
\[ \begin{aligned}
\dot{H}^{\tw [m;m_2]\mathrm{I}}_{k_1,k_2} \tz =  & \  \half (-i)^{m_2-1}e^{-\frac{\pi i}{M}(2m+1)(k_1+k_2)} \left(\Psi^{[M,-4m-2;-m_2;0]}_{\frac{k_1+M}{2},-\frac{k_1+2k_2-M}{2}; \epsilon'}(\tau, \frac{2z+1}{4}, \frac{2z-1}{4} ) \right. \\
& \left. -(-1)^{m_2}  \Psi^{[M, -4m-2; -m_2; \half]}_{\frac{k_1+M}{2},-\frac{k_1+2k2-M}{2}; \epsilon'} (\tau, \frac{2z+1}{4}, \frac{2z-1}{4}) \right), \\
H^{\tw [m;m_2]\mathrm{III}}_{k_1,k_2} \tz = & \ -\half (-i)^{m_2-1}e^{-\frac{\pi i}{M}(2m+1)(k_1+k_2)}\left(\Psi^{[M, -4m-2; -m_2; 0]}_{ \frac{k_1+2k_2}{2},-\frac{k_1}{2}; \epsilon'} (\tau, \frac{2z+1}{4}, \frac{2z-1}{4} )\right. \\
& \left.  -(-1)^{m_2}  \Psi^{[M, -4m-2; -m_2;\half]}_{ \frac{k_1+2k_2}{2},-\frac{k_1}{2}; \epsilon'} (\tau, \frac{2z+1}{4}, \frac{2z-1}{4}) \right), \\
\dot{H}^{\tw [m;m_2]\mathrm{III}}_{k_1,k_2} \tz  = & \ -
 \half (-i)^{m_2-1}e^{-\frac{\pi i}{M}(2m+1)(k_1+k_2)}\left(\Psi^{[M, -4m-2; -m_2; 0]}_{ \frac{k_1+2k_2-M}{2},-\frac{k_1 +M}{2}; \epsilon'} (\tau, \frac{2z+1}{4}, \frac{2z-1}{4} ) \right. \\
 & \left. -(-1)^{m_2}  \Psi^{[M, -4m-2;-m_2; \half]}_{
      \frac{k_1+2k_2-M}{2},-\frac{k_1+M}{2}; \epsilon'} (\tau, \frac{2z+1}{4}, \frac{2z-1}{4} ) \right). \\
\end{aligned}
 \]
\end{enumerate}
\qed
\end{theorem}


The modified $ N=3 $ (super)characters $ \tilde{\ch}^\pm_{H(\La)} $ and the modified twisted $ N=3 $ characters $ \tilde{\ch}^\tw_{H(\La)}  $ are obtained, by definition, from their expressions, given by Theorem \ref{th4.3}, by replacing in the RHS the function $ \Psi $ by its modification $ \tilde{\Psi}. $ Note that, since $ m_2 \in \ZZ, $ we can replace
in $ \tilde{\Psi}^{[M; -4m-2;s; \epsilon]}, $ the integer $ s $ by 0, using 
\cite{KW16}, Corollary 1.6 . 

In order to compute the modular transformation properties of these modified $ N=3 $ (super)characters, we introduce the function
\begin{equation}
\label{4.14}
f^{[M;n;\epsilon](\sigma, \sigma')}_{j,k;\epsilon'} \tz : = q^{-\frac{n}{4M} {\sigma'}^2} \tilde{\Psi}^{[M,n;0;\epsilon]}_{j;k;\epsilon'} \left(\tau, \frac{z+\sigma + \sigma' \tau}{2}, \frac{z-\sigma-\sigma' \tau}{2}\right),
\end{equation}
\noindent where $ M  $ is a positive odd integer, $ n $ is a positive integer, coprime to $ M, \ \epsilon, \epsilon' = 0 $ or $ \half $ are considered $ \bmod \ZZ, \ j, k \in \epsilon' + \ZZ, $ and $ \sigma, \sigma' = 0 $ or $ \half. $
The properties of these functions that will be used are described by the following three statements.
\begin{lemma}
\label{lem4.4}
\[ f^{[M;n;\epsilon](\sigma,\sigma')}_{j,k;\epsilon'} = 
e^{-\frac{2\pi i n}{M}\sigma\sigma'} e^{\frac{2\pi i n}{M}\sigma (k-j)} 
f^{[M;n; |\sigma - \epsilon|](\sigma,\sigma')}_{k-\sigma',j+\sigma'; |\sigma' -\epsilon'|} . \] 
\end{lemma}
\begin{lemma}
\label{lem4.5}
For 
$ a,b \in \ZZ $ one has:
\[ f^{[M;n;\epsilon](\sigma, \sigma')}_{j + aM, k+bM; \epsilon'} = e^{2 \pi i n (a-b)\epsilon} f^{[M;n;\epsilon](\sigma, \sigma')}_{j,k;\epsilon'}. \]
\end{lemma}
\begin{proposition}
\label{prop4.6}
One has the following modular transformation properties: 
\begin{enumerate}
\item[(a)] \[ \begin{aligned}
& f^{[M;n;\epsilon](\sigma, \sigma')}_{j,k;\epsilon'}  \left(-\frac{1}{\tau}, \frac{z}{\tau}\right) = \\
& \frac{\tau}{M} e^{\frac{\pi i n}{2 M \tau}z^2} e^{\frac{\pi i n}{M} \sigma \sigma'} \hspace{-4ex} \sum_{(a,b) \in (\epsilon + \ZZ /M\ZZ)^2} \hspace{-4ex} e^{-\frac{2 \pi i n}{M} (ak+bj)+\frac{2 \pi i n}{M}\sigma' (a-b)} f^{[M;n;\epsilon'+\sigma'](\sigma', \sigma)}_{a,b;\epsilon} \tz .\\
\end{aligned}
 \]
\item[(b)] \[ \begin{aligned}
f^{[M;n;\epsilon](\sigma, 0)}_{j,k;\epsilon'}(\tau + 1, z) &= e^{\frac{2 \pi i n}{M}jk} f^{[M, n;\epsilon + \epsilon'](\sigma,0)}_{j,k;\epsilon} \tz,  \\
f^{[M;n;\epsilon](\sigma, \half)}_{j,k;\epsilon'} (\tau + 1, z) &= e^{\frac{2 \pi i n}{M}(jk- \frac{1}{16})} e^{\frac{2 \pi i n \sigma}{M} (k-j)} f^{[M;n;\epsilon + \epsilon'-\sigma](\half - \sigma, \half)}_{j,k;\epsilon'} \tz. \\ 
\end{aligned}
 \]
\end{enumerate}
\end{proposition}

In order to prove these properties of $ f^{[M,n;\epsilon](\sigma, \sigma')}_{j,k;\epsilon'}, $ we make use of the following properties of $ \tilde{\Psi}, $ which follow from the properties of $ \tilde{\Phi}, $ see \cite{KW15}, Theorems 1.12, and 2.8.

\begin{lemma}
\label{lem4.7}
Let $ m $ be a positive integer, $ M $ a positive odd integer coprime to m and $ s \in \half \ZZ. $ Then
\[ \tilde{\Psi}^{[M, m; s; \epsilon]}_{j + aM, k + bM; \epsilon'} = e^{2 \pi i m (a-b) \epsilon} \tilde{\Psi}^{[M, m; s; \epsilon]}_{j,k; \epsilon'} \mbox{ for } a,b \in \ZZ. \]
\end{lemma}

\begin{lemma}
\label{lem4.8}
Let $ m $ and $ M $ be as in Lemma \ref{lem4.7} and let $ s \in \ZZ. $ Then the following formulas hold for $ a, b \in \ZZ, A, B \in \half + \ZZ: $
\begin{itemize}
\item[(a)] \[ \tilde{\Psi}^{[M,m;s; \epsilon]}_{j,k;\epsilon'} (\tau, z_1 + a, z_2 + b, t) = e^{\frac{2 \pi i m}{M} (ka + jb)} \tilde{\Psi}^{[M,m;s; \epsilon]}_{j,k;\epsilon'} \tzzt \]
\item[(b)] \[ \tilde{\Psi}^{[M,m;s; \epsilon]}_{j,k;\epsilon'} (\tau, z_1 + A, z_2 + B, t) = e^{\frac{2 \pi i m}{M} (kA + jB)} \tilde{\Psi}^{[M,m;s; \thalf- \epsilon]}_{j,k;\epsilon'} \tzzt \]
\end{itemize}
\end{lemma}
\begin{lemma}
\label{lem4.9}
Let $ m \in \half \zp, $ let $ M $ be a positive odd integer coprime to $ 2m, $ and let $ s \in \half \ZZ. $ Then for $ A, B \in \half + \ZZ  $ one has:
\[ \tilde{\Psi}^{[M,m;s; \epsilon]}_{j,k;\epsilon'} (\tau, z_1 + A \tau, z_2 + B \tau, t) = q^{-\frac{m}{M} AB} e^{- \frac{2 \pi i m}{M} (Bz_1+ Az_2)} \tilde{\Psi}^{[M,m;s; \epsilon]}_{j+A,k+B;\half-\epsilon'} \tzzt. \]
\end{lemma}
\begin{lemma}
\label{lem4.10}
Let $ M $ and $ m $ be as in Lemma \ref{lem4.7}. Then:
\begin{itemize}
\item[(a)] \[ \tilde{\Psi}^{[M,m;0;\epsilon]}_{j,k;\epsilon'} \left( -\frac{1}{\tau}, \frac{z_1}{\tau}, \frac{z_2}{\tau}, t \right) = \frac{\tau}{M} e^{\frac{2 \pi i m}{M \tau}z_1z_2} \sum_{a, b \in (\epsilon + \ZZ / M\ZZ)^2} e^{-\frac{2 \pi i m}{M} (ak+bj)} \tilde{\Psi}^{[M,m;0;\epsilon']}_{a,b;\epsilon} \tzzt. \]
\item[(b)] \[ \tilde{\Psi}^{[M,m;0;\epsilon]}_{j,k;\epsilon'} (\tau + 1, z_1, z_2, t) = e^{\frac{2 \pi i m}{M}jk} \tilde{\Psi}^{[M,m;0;|\epsilon-\epsilon'|]}_{j,k; \epsilon'} \tzzt. \]
\end{itemize}
\end{lemma}
Also, we shall use the following relation, which follows from \eqref{4.14}, using Lemma \ref{lem4.8}:
\begin{equation}
\label{4.15}
q^{-\frac{n}{4M} {\sigma'}^{2}} \tilde{\Psi}^{[M,n;0;\epsilon]}_{j,k; \epsilon'} \left( \tau, \frac{z-\sigma + \sigma' \tau}{2}, \frac{z + \sigma - \sigma' \tau}{2} \right) = e^{\frac{2 \pi i n}{M} \sigma (j-k)}) f^{[M, n; |\epsilon - \sigma] (\sigma, \sigma')}_{j,k;\epsilon'} (\tau, z).
\end{equation}

Now we are going to prove our claims for the functions $ f. $
\begin{proof}[Proof of Lemma 4.4]
We have:
\[ 
\begin{aligned}
& f^{[M,n;\epsilon](\sigma, \sigma')}_{j,k,\epsilon'} (\tau, z)  = q^{-\frac{n}{4M}{\sigma'}^2} \tilde{\Psi}^{[M,n;0;\epsilon]}_{j,k;\epsilon'} \left( \tau, \frac{z + \sigma + \sigma' \tau}{2}, \frac{z-\sigma - \sigma' \tau}{2}\right)\\
& = q^{-\frac{n}{4M}{\sigma'}^2} \tilde{\Psi}^{[M,n;0;\epsilon]}_{k,j;\epsilon'} \left( \tau, \frac{z - \sigma - \sigma'\tau}{2}, \frac{z + \sigma + \sigma' \tau}{2} \right) \\
\end{aligned} \]
\[ \begin{aligned}
& = q^{-\frac{n}{4M}{\sigma'}^2} \tilde{\Psi}^{[M,n;0;\epsilon]}_{k,j;\epsilon'} \left( \tau, \frac{z+\sigma + \sigma' \tau}{2} - \sigma -\sigma' \tau, \frac{z - \sigma -\sigma' \tau}{2} + \sigma + \sigma' \tau \right) (\mbox{by Lemma}\,
  \ref{4.9}) \\
  & = q^{-\frac{n}{4M}\sigma'^2}
  q^{\frac{n}{M}(-{\sigma'}^2)}
  e^{-\frac{2 \pi i n}{M}( \sigma' (\frac{z+\sigma+\sigma' \tau}{2} -\sigma)-\sigma' ( \frac{z-\sigma-\sigma'\tau}{2} + \sigma ))} \\ & \times\tilde{\Psi}^{[M,n;0;\epsilon]}_{k-\sigma', j+\sigma'; |\epsilon' -\sigma'|} \left( \tau, \frac{z+\sigma + \sigma' \tau}{2}-\sigma, \frac{z-\sigma- \sigma' \tau}{2} + \sigma \right) \\
& = q^{-\frac{n}{4M}{\sigma'}^2} e^{\frac{2 \pi i n}{M} \sigma \sigma'} \tilde{\Psi}^{[M,n;0;\epsilon]}_{k-\sigma', j+\sigma'; |\epsilon'-\sigma'|} \left(\tau, \frac{z+ \sigma + \sigma'\tau}{2}-\sigma, \frac{z-\sigma -\sigma'\tau}{2} + \sigma \right) (\mbox{by Lemma \ref{lem4.8}}) \\
& = e^{-\frac{2 \pi i n}{M} \sigma \sigma'} e^{\frac{2 \pi i n}{M}\sigma (k-j)} q^{-\frac{n}{4M} {\sigma'}^2} \tilde{\Psi}^{[M,n;0; |\epsilon - \sigma|]}_{k-\sigma', j + \sigma'; |\epsilon' - \sigma'|} (\tau, \frac{z+\sigma + \sigma' \tau}{2}, \frac{z - \sigma - \sigma'\tau}{2}) \\
& = e^{-\frac{2 \pi i n}{M} \sigma (\sigma'+j - k)} f^{[M, n; |\epsilon-\sigma|](\sigma \sigma')}_{k-\sigma', j + \sigma'; |\epsilon'-\sigma'|} (\tau, z).
\end{aligned}
 \]
\end{proof}
\begin{proof}[Proof of Lemma \ref{lem4.5}]
We have:
\[ 
\begin{aligned}
f^{[M, n; \epsilon](\sigma, \sigma')}_{j + aM, k + bM; \epsilon'} (\tau, z) & = q^{-\frac{n}{4M}{\sigma'}^2} \tilde{\Psi}^{[M,n;0;\epsilon]}_{j+aM, k+bM; \epsilon} \left( \tau, \frac{z+ \sigma+\sigma'\tau}{2}, \frac{z - \sigma - \sigma'\tau}{2} \right)\\
& = \mbox{(by Lemma \ref{lem4.7})} \\
& = q^{-\frac{n}{4M}{\sigma'}^2} e^{2 \pi i n (a-b) \epsilon} \tilde{\Psi}^{[M,n;0;\epsilon]}_{j,k;\epsilon'} \left( \tau, \frac{z + \sigma + \sigma' \tau}{2}, \frac{z - \sigma - \sigma' \tau }{2} \right) \\
& = e^{2 \pi i n (a-b) \epsilon} f^{[M,n;\epsilon](\sigma, \sigma')}_{j,k;\epsilon'} (\tau, z).
\end{aligned}
 \]
\end{proof}

\begin{proof}[Proof of Proposition \ref{prop4.6}(a).]
We have:
\[ 
\begin{aligned}
& f^{[M,n;\epsilon](\sigma \sigma')}_{j,k;\epsilon'} \left( - \frac{1}{\tau}, \frac{z}{\tau}\right)  = e^{\frac{2 \pi i n}{4 M \tau} {\sigma'}^2 } \tilde{\Psi}^{[M,n;0;\epsilon]}_{j,k; \epsilon'} \left( - \frac{1}{\tau}, \frac{ \frac{z}{\tau}+ \sigma - \frac{\sigma'}{\tau}}{2}, \frac{ \frac{z}{\tau}- \sigma + \frac{\sigma'}{\tau}}{2} \right) \\
& = e^{\frac{\pi i n}{2M \tau}{\sigma'}^2} \tilde{\Psi}^{[M,n;0;\epsilon]}_{j,k; \epsilon'} \left( - \frac{1}{\tau}, \frac{\frac{z-\sigma'+\sigma \tau}{2}}{\tau}, \frac{\frac{z+\sigma'-\sigma \tau}{2}}{\tau} \right) \\
& = \mbox{(by Lemma \ref{lem4.10}(a))} \\
& = e^{\frac{\pi i n}{2 M \tau}{\sigma'}^2}  \frac{\tau}{M} e^{\frac{2 \pi i n}{M \tau} \frac{z-\sigma'+\sigma \tau}{2}  \frac{z + \sigma' - \sigma \tau}{2}} \hspace{-2em}\sum_{(a,b) \in (\epsilon + \ZZ / M\ZZ)^2} \hspace{-2em} e^{-\frac{2 \pi i n}{M} (ak+bj)}  \tilde{\Psi}^{[M,n;0;\epsilon']}_{a,b; \epsilon} \left( \tau, \frac{z-\sigma' + \sigma \tau}{2}, \frac{z + \sigma' - \sigma \tau}{2} \right) \\
& = \mbox{(using (\ref{4.15}))}  = \frac{\tau}{M} e^{\frac{\pi i n}{M} \sigma \sigma'} e^{\frac{\pi i n}{2 M \tau}z^2} \hspace{-2em}\sum_{(a,b) \in (\epsilon + \ZZ / M\ZZ)^2} \hspace{-2em} e^{-\frac{2 \pi i n}{M}(ak+bj)+ \frac{2 \pi i n}{M}\sigma' (a-b)} f^{[M,n; |\epsilon' - \sigma'|](\sigma, \sigma')}_{a,b; \epsilon} (\tau, z).
\end{aligned} \]
\end{proof}

\begin{proof}[Proof of 
Proposition \ref{prop4.6}(b)] 
If $ \sigma' = 0, $ we have:
\[ 
\begin{aligned}
f^{[M,n;\epsilon](\sigma,0)}_{j,k;\epsilon'} (\tau+1, z) & = \tilde{\Psi}^{[M,n;0;\epsilon]}_{j,k;\epsilon'} \left(\tau + 1, \frac{z + \sigma}{2}, \frac{z - \sigma}{2} \right) = \mbox{(by Lemma \ref{lem4.10}(a))}\\
& = e^{\frac{2 \pi i n}{M}jk} \tilde{\Psi}^{[M,n;0; |\epsilon-\epsilon'|]}_{j,k;\epsilon'} \left( \tau, \frac{z+\sigma}{2}, \frac{z-\sigma}{2}\right)\\
& = e^{\frac{2 \pi i n}{M}jk} f^{[M,n; |\epsilon-\epsilon'|]}_{j,k; \epsilon'} (\tau, z).
\end{aligned} \]
If $ \sigma' = \half, $ we have:
\[ 
\begin{aligned}
f^{[M,n;\epsilon](\sigma,\half)}_{j,k;\epsilon'} (\tau+1, z) & = e^{2 \pi i (\tau +1) (- \frac{n}{4M})\cdot \frac{1}{4}} \tilde{\Psi}^{[M,n;0;\epsilon]}_{j,k;\epsilon'} \left( \tau + 1, \frac{z + \sigma + \frac{\tau +1}{2}}{2}, \frac{z - \sigma - \frac{\tau +1}{2}}{2} \right) \\
& = \mbox{(by Lemma \ref{lem4.10}(b))}\\
& = q^{-\frac{n}{4M} \frac{1}{4}} e^{-\frac{2 \pi i n}{M}  \frac{1}{16}} e^{\frac{2 \pi i n}{M}jk} \tilde{\Psi}^{[M,n;0; |\epsilon-\epsilon'|]}_{j,k;\epsilon'} \left( \tau, \frac{z + \sigma + \frac{\tau +1}{2}}{2}, \frac{z - \sigma - \frac{\tau +1}{2}}{2} \right) \\
& = e^{\frac{2 \pi i n}{M}(jk - \frac{1}{16})} q^{-\frac{n}{4M}\frac{1}{4}} \tilde{\Psi}^{[M,n;0; |\epsilon-\epsilon'|]}_{j,k;\epsilon'} \left( \tau, \frac{z+\sigma+ \half + \frac{\tau}{2}}{2}, \frac{z-\sigma- \half -\frac{\tau}{2}}{2} \right).
\end{aligned} \]
We compute the last expression in the cases $\sigma=0$ and $\sigma=\half$
as follows.
If $ \sigma = 0, $ then this expression is equal to
\[ e^{\frac{2 \pi i n}{M}(jk - \frac{1}{16})} q^{-\frac{n}{4M} \frac{1}{4}}  \tilde{\Psi}^{[M,n;0; |\epsilon-\epsilon'|]}_{j,k;\epsilon'} \left( \tau, \frac{z + \half + \frac{\tau}{2}}{2}, \frac{z - \half - \frac{\tau}{2}}{2} \right) = e^{\frac{2 \pi i n}{M}(jk - \frac{1}{16})} f^{[M,n; |\epsilon-\epsilon'|] (\half, \half)}_{j,k;\epsilon'} (\tau, z). \]
If $ \sigma = \half, $ then this expression is equal to
\[ 
\begin{aligned}
& e^{\frac{2 \pi i n}{M}(jk - \frac{1}{16})} q^{-\frac{n}{4M} \frac{1}{4}} \tilde{\Psi}^{[M,n;0; |\epsilon-\epsilon'|]}_{j,k;\epsilon'} \left(\tau, \frac{z + 1 + \frac{\tau}{2}}{2},\frac{z - 1 - \frac{\tau}{2}}{2} \right)  = \mbox{(by Lemma \ref{lem4.8}(b))}\\
& = e^{\frac{2 \pi i n}{M}(jk - \frac{1}{16})} q^{-\frac{n}{16M}}e^{\frac{2 \pi i n}{M} (\frac{k}{2}-\half)} \tilde{\Psi}^{[M,n;0; \half- |\epsilon-\epsilon'|]}_{j,k;\epsilon'} \left( \tau, \frac{z + \frac{\tau}{2}}{2}, \frac{z - \frac{\tau}{2}}{2} \right) \\
& = e^{\frac{2 \pi i n}{M}(jk - \frac{1}{16})} e^{\frac{2 \pi i n}{M} \frac{k-j}{2}} q^{-\frac{n}{4M} \frac{1}{4}}  \tilde{\Psi}^{[M,n;0; \half- |\epsilon-\epsilon'|]}_{j,k;\epsilon'} \left(\tau, \frac{z + \frac{\tau}{2}}{2}, \frac{z - \frac{\tau}{2}}{2} \right) \\
& = e^{\frac{2 \pi i n}{M}(jk - \frac{1}{16})} e^{\frac{2 \pi i n}{M} \frac{k-j}{2}} f^{[M,n; \half - |\epsilon-\epsilon'|](0, \half)}_{j,k;\epsilon'} (\tau, z).
\end{aligned} \]
\end{proof}

We can rewrite the formulas for modified $ N =3 $ (super)characters and for twisted $ N=3 $ characters, given by Theorem \ref{th4.3}, in terms of the functions $ f^{[M;n;\epsilon](\sigma, \sigma')}_{j,k;\epsilon'} $ as follows.
\begin{theorem}
\label{th4.11}
Let  $ \La = \La^{[m;m_2]J}_{k_1, k_2}  $ or $ \La = \dot{\La}^{[m;m_2]J}_{k_1, k_2}  $, where $J$=I or III. Denote by    
$\tilde{H}^{\pm [m;m_2]J}_{k_1,k_2}$ and $\tilde{H}^{\tw[m;m_2]J}$ 
the modifications of the functions 
$H^{\pm [m;m_2]J}_{k_1,k_2}$ and $ H^{\tw [m;m_2]J}_{k_1,k_2}$,
and similarly for $\dot{H}$. 
These functions are given, up to some non-zero scalar factors, 
by the following formulas: 
\[ \begin{aligned}
&
\tilde{H}^{+ [m;m_2]J}_{k_1,k_2}, \, \tilde{\dot{H}}^{+ [m;m_2]J}_{k_1,k_2}
= f^{[M; -4m-2;0](\half, \half)}_{j,k;\epsilon'} - (-1)^{m_2} f^{[M, -4m-2; \half](\half,\half)}_{j,k;\epsilon'}, \\
&
\tilde{H}^{- [m;m_2]J}_{k_1,k_2}, \, \tilde{\dot{H}}^{- [m;m_2]J}_{k_1,k_2}
= f^{[M; -4m-2; 0](0, \half)}_{j,k;\epsilon'} - (-1)^{m_2} f^{[M; -4m-2; \half](0, \half)}_{j,k;\epsilon'}, \\
&
\tilde{H}^{\tw [m;m_2]J}_{k_1,k_2}, \, \tilde{\dot{H}}^{\tw [m;m_2]J}_{k_1,k_2}
= f^{[M; -4m-2;0](\half, 0)}_{j,k;\epsilon'} - (-1)^{m_2} f^{[M; -4m-2; \half](\half, 0)}_{j,k;\epsilon'}, \\
\end{aligned}
 \]
 \noindent where $ j,k \in \epsilon'+\ZZ $ are defined from $ k_1, k_2 $ by the following table: 
\begin{center}
 \begin{tabular}{l||c|c|c|c}
$ \La $ & $ \rule[-2ex]{0pt}{6ex} \La^{[m;m_2]\mathrm{I}}_{k_1, k_2} $ & $ \La^{[m;m_2]\mathrm{III}}_{k_1, k_2} $& $ \dot{\La}^{[m;m_2]\mathrm{I}}_{k_1, k_2} $& $ \dot{\La}^{[m;m_2]\mathrm{III}}_{k_1, k_2} $\\
\hline \hline
$ k $ & $ \rule[-2ex]{0pt}{5ex} -\frac{k_1+2k_2}{2} $ & $ -\frac{k_1}{2} $ & $ - \frac{k_1 + 2k_2 -M}{2} $& $ -\frac{k_1+M}{2} $ \\
\hline
$ j $ & $\rule[-2ex]{0pt}{5ex} \frac{k_1}{2} $ & $ \frac{k_1 +2k_2}{2} $& $ \frac{k_1+M}{2} $& $ \frac{k_1+2k_2-M}{2} $ \\
 \end{tabular}
\end{center}
\end{theorem}

By Lemmas \ref{lem4.4} and \ref{lem4.5}, the functions $ f^{[M;n;\epsilon](\sigma, \sigma')}_{j,k;\epsilon'} (\tau, z) $ are invariant, up to a non-zero constant factor, with respect to the following transformation of indices:
\begin{equation}
\label{4.16}
(j,k) \mapsto (k,j), \ (j,k) \mapsto (j + \ZZ M, k + \ZZ M), \ \mbox{ and } (j,k) \mapsto (k-\thalf, j + \thalf).
\end{equation}
Hence we can choose for the fundamental domain of parameters $ (j,k) $ the points in 
$ (\half \ZZ) \times (\half \ZZ) $ 
from the triangle
\begin{equation}
\label{4.17}
0 \leq j \leq k < M.
\end{equation}

Now consider for each $ \La $ from the Table of Theorem \ref{th4.11} the domain of indices $ (j,k)\in (\half \ZZ) \times (\half \ZZ) $, denoting, for simplicity, for $ J= $ I or III, by $ J $ (resp. $ \dot{J} $) the domain of $ (j,k) $ for $ \La = \La^{[m;m_2]J}_{k_1,k_2} $ (resp. $ \dot{\La}^{[m;m_2]J}_{k_1,k_2} $). It is easy to see that these domains are as follows:
\begin{itemize}
\item[I:] $ j + k \leq 0, \ j-k \leq \frac{M-1}{2}, \ j \geq 0 $;
\item[III:] $ j + k \geq 1, \ j-k \leq \frac{M-1}{2}, \ j \geq 0 $;
\item[$ \dot{\mathrm{I}}$:]$ j + k \leq M, \ j-k \leq \frac{M-1}{2}, \ j \geq \frac{M}{2} $;
\item[$ \dot{\mathrm{III}}$:]$ j + k \geq 1-M, \ j-k \leq \frac{M-1}{2}, \ k \leq -\frac{M}{2}. $
\end{itemize} 
Next, using translation $ (j,k) \mapsto (j,k+M) $ the domain I $ \cup  $ III can be brought to:
\begin{equation}
\label{4.18}
\mathrm{I} \cup \mathrm{III}: k-j \geq \tfrac{M+1}{2}, \ k \leq M, \ j \geq 0.
\end{equation}
Furthermore the domain $ \dot{\mathrm{III}}, $ translated by $ (j,k) \mapsto (j+M, k+M), $ together with $ \dot{\mathrm{I}} $ give:
\[ \dot{\mathrm{I}} \cup \dot{\mathrm{III}}: j-k \leq \tfrac{M-1}{2}, \ k \leq \tfrac{M}{2}, \ j \geq \tfrac{M}{2}. \]
Applying to this transformation $ (j,k) \mapsto (k,j), $ we obtain:
\begin{equation}
\label{4.19}
\dot{\mathrm{I}} \cup \dot{\mathrm{III}}: j-k \geq \tfrac{M-1}{2}, \ k \geq \tfrac{M}{2}, \ j \leq \tfrac{M}{2}.
\end{equation}
From \eqref{4.18} and \eqref{4.19} we obtain:
\begin{equation}
\label{4.20}
\mathrm{I} \cup \mathrm{III} \cup \dot{\mathrm{I}} \cup \dot{\mathrm{III}}: \tfrac{M}{2} \leq k \leq M, \ 0 \leq j \leq \tfrac{M}{2}.
\end{equation}

Comparing \eqref{4.20} with \eqref{4.17}, we see that the parameters $ (j,k) $ from the triangles 
\begin{equation}
\label{4.21}
 0 \leq j \leq k \leq \frac{M}{2}\,\, \hbox{ and}\,\, \frac{M}{2} \leq j \leq k \leq M 
 \end{equation}
are missing. In order to get these missing representations, consider the set 
of simple roots
\[ \hat{\Pi}' = r_{\al_1} \hat{\Pi} = \{ \al_0 + \al_1, -\al_1, \al_1 + \al_2 \}, \]
and let $ \hat{\Delta}'_+ $ denotes the corresponding subset of positive 
roots in $ \hat{\Delta} $. Denote by $ \Pi^{{J}'}_{M;k_1,k_2} $, 
where $ J' = \mathrm{I}' $ or $ \mathrm{III}', $ the simple subsets $ S $ of 
$ \hat{\Delta}, $ given by (\ref{3.6a}), but such that 
$ r_{k_1 \delta + \al_1}S \subset \hat{\Delta}'_+ $. Then the range of possible pairs $ (k_1, k_2) $ will be different from (\ref{3.6b}), namely it becomes
\begin{equation}
\label{4.22}
\mathrm{I}' \ (\mbox{resp. } \mathrm{III}'): k_1 \leq 0 \ (\mbox{resp. } \leq -1),\, k_1 + k_2 \geq 0,\, k_1 + 2k_2 \leq M-1 \ (\mbox{resp. } \leq M ),
\end{equation}
Indeed $ r_{k_1\delta + \al_1}\Pi^{\mathrm{I}'}_{k_1, k_2} = \{ (k_0 + k_1)\delta + \al_0 + \al_1, -k_1 \delta - \al_1, (k_1 + k_2) \delta + \al_1 + \al_2 \} \subset \hat{\Delta}'_+ = 
r_{\al_1} \hat{\Delta}_+ $ iff $ k_0 + k_1 \geq 0, k_1 \leq 0, k_1 + k_2 \geq 0, $ which implies \eqref{4.22} for $ \mathrm{I}' $, and similarly for $ \mathrm{III}' $.

Similarly to Proposition \ref{prop3.6}, denote by $ \La^{[m;m_2] J'}_{k_1, k_2} $ the weight $ \La $ of level $ K, $ such that $ \La^0 = \La^{[m;m_2]} $ and $ \Pi_{\La} = \Pi_{M; k_1, k_2}^{J'}, $ where $ J' =  \mathrm{I}'$ or $ \mathrm{III}' $. We call such $ \La $ \emph{non-pricipal admissible weights}. All claims stated above for principal admissible weights hold for the non-principal admissible ones, except that in Theorem \ref{th4.11} the indices $ j,k \in \epsilon' + \ZZ $ are defined from $ k_1, k_2 $ by the following table:
\begin{center}
 \begin{tabular}{l||c|c}
$ \La $ & $ \rule[-2ex]{0pt}{6ex} \La^{[m;m_2]\mathrm{I}'}_{k_1, k_2} $ & $ \La^{[m;m_2]\mathrm{III}'}_{k_1, k_2} $\\
\hline \hline
$ j $ & $ \rule[-2ex]{0pt}{6ex} \frac{k_1}{2} $ & $ \frac{k_1+2k_2}{2} $ \\
\hline
$ k $ & $ \rule[-2ex]{0pt}{6ex}  -\frac{k_1+2k_2}{2} $ & $ -\frac{k_1}{2} $ \\
 \end{tabular}
\end{center}
Let $ J' = \mathrm{I}' $ or $ \mathrm{III}' $ denote the domain of parameters $ (j,k) $ for $ \La = \La^{[m;m_2]J'}_{k_1, k_2}, $ such that $ \Pi_{\La} = \Pi^{J'}_{M;k_1, k_2}$. Then these domains are the following triangles:
\[ 
\begin{aligned}
\mathrm{I}': \quad  & k \geq -\frac{M-1}{2}, \ j \leq -\half, \ k-j \leq 0,\\
\mathrm{III}': \quad  & k \geq \half, \ j \leq \frac{M-1}{2}, \ k-j \leq 0. \\
\end{aligned} \]
Hence, applying transformations $ (j,k) \mapsto (k,j) $ and $ (j,k) \mapsto (j+M, k+M), $ if necessary, we obtain the triangles from \eqref{4.21}. We thus arrive at the following theorem.

\begin{theorem}
\label{th4.12}
Let $ M $ be a positive odd integer and let $ m \in \frac{1}{4} \ZZ $ be such that $ n = -4m -2 $ is a positive integer, coprime to $ M. $ Let $ V $ denote the span of the following functions:
\begin{enumerate}
\item[(i)] modified characters and supercharacters of representations of $ N=3 $ superconformal algebra of central charge $ C = \frac{3n}{2M}-\half, $ obtained by QHR from the principal admissible $ \wg $-modules $ \LLa $ with $ \La = \La^{[m;m_2]J}_{k_1,k_2} $ and  $ \dot{\La}^{[m;m_2]J}_{k_1,k_2}, $ where $ J = $ I or III, $ m_2 = 0 $ or 1, and $ k_1, k_2 $ are integers satisfying 
(\ref{3.6b}),
\item[(ii)] the same as in (i), except that $ \La = \La^{[m;m_2]J'}_{k_1,k_2}, $ where 
$ J' = \mathrm{I}'$ or $ \mathrm{III}' $ and $ k_1, k_2  $ are integers, satisfying \eqref{4.22}.
\item[(iii)] modified characters of representations of the Ramond twisted 
$ N=3 $ superconformal algebra with central charge $ C = \frac{3n}{2M}-\half, $ obtained by QHR from twisted $ \wg $-modules with the same highest weights as in (i) and (ii). 

Then
\begin{enumerate}
\item[(a)] The space $ V $ coincides with the span of the functions $ f^{[M;n;\epsilon](\sigma, \sigma')}_{j,k;\epsilon'}  (\tau, z)$ with 
$(\sigma,\sigma')\neq(0,0)$,
defined by \eqref{4.14}. 
\item[(b)] The space $ V $ is invariant with respect to the following action of $ SL_2(\ZZ): $
\[ f\bigg{|}_{\left(\begin{smallmatrix}
a & b \\
c & d \\
\end{smallmatrix}\right)}  (\tau, z) = ( c \tau + d)^{-1} e
^{-\frac{\pi i (C+\half) c z^2}{3 (c \tau + d)}} f \left( \frac{a \tau + b}{c \tau + d}, \frac{z}{c \tau + d }\right). \]
\end{enumerate}
\end{enumerate}
\end{theorem}
\begin{proof}
Claim (a) follows from Theorem \ref{th4.11} and its analogue for non-principal admissible weights, due to the discussion preceding the statement of Theorem \ref{th4.12}. Claim (b) follows from claim (a) due to Proposition \ref{prop4.6} and the fact that $ SL_2(\ZZ) $ permutes the denominators $ \overset{3}{R}\vphantom{R}^\pm $ and $ \overset{3}{R}\vphantom{R}^\tw, $ given by \eqref{4.5} and \eqref{4.13} (see \cite{KW15}, Proposition 6.1).
\end{proof}

\begin{remark}
\label{rem4.13}
For $m=-\frac{3}{4}$, i.e. $n=1$, Theorem \ref{th4.12} holds without modifications of (super)characters, i.e., they are meromorphic modular functions. In fact, by Remark \ref{rem1.3}, formula \eqref{4.14} for $n=1$ turns into the following equation:
\[ f^{[M,1;\epsilon](\sigma,\sigma')}_{j,k;\epsilon'}(\tau,z)
=-ie^{\frac{\pi i}{M}(k-j)\sigma}  
q^{\frac{1}{M}
(j+\frac{\sigma'}{2})(k-\frac{\sigma'}{2})}
e^{\frac{\pi i}{M}(j+k)z}\]
\[\times \frac{\eta(M\tau)^3\vartheta_{11}(M\tau,z+(j+k)\tau)}
{\vartheta_{11}(M\tau,\frac{z+\sigma+\sigma'\tau}{2}+j\tau+\epsilon)
\vartheta_{11}(M\tau,\frac{z-\sigma-\sigma'\tau}{2}+k\tau-\epsilon)}.
\]
\end{remark}
\begin{remark}
\label{rem4.14}
Let $ M =1. $ Then Lemmas \ref{lem4.4} and \ref{lem4.5} give:
\[ 
f^{[1,n; \epsilon](\half, \half)}_{\half, \half; \half}  = 
e^{2 \pi i n (\epsilon+\frac{1}{4})} f^{[1,n; \half -\epsilon](\half, \half)}_{0, 0; 0}; 
\,
f^{[1,n; \epsilon](0, \half)}_{\half, \half; \half}  = e^{2 \pi i n \epsilon} f^{[1,n; \epsilon](0, \half)}_{0, 0; 0};\,
f^{[1,n; \epsilon](\half, 0)}_{\epsilon', \epsilon'; \epsilon'}  = f^{[1,n; \half -\epsilon](\half, 0)}_{\epsilon', \epsilon'; \epsilon'}.
\]
Hence in this case the following functions span $ V $ in Theorem \ref{th4.12}:
\[ f^{[1,n; \epsilon](\half, \half)}_{0, 0; 0}, \quad f^{[1,n; \epsilon](0, \half)}_{0, 0; 0}, \quad f^{[1,n; 0](\half, 0)}_{\epsilon', \epsilon'; \epsilon'}, \mbox{ where } \epsilon, \epsilon' = 0 \mbox{ or } \half. \]
\end{remark}

\begin{remark}
\label{rem4.15}
There are more relations between the functions \eqref{4.14}, than the ones, provided by Lemmas \ref{lem4.4} and \ref{lem4.5}. For example, one can prove the following relations, using these lemmas: 
\begin{equation}
\label{4.23}
f^{[M,n;\epsilon](\sigma, \sigma')}_{k + \frac{2M-1}{2}, k + \frac{M}{2}; \epsilon'} = e^{2 \pi i n \epsilon} e^{-\frac{2 \pi i n}{M} \sigma \sigma'} 
e^{\frac{2 \pi i n}{M} \frac{M+1}{2} \sigma} f^{[M,n; |\epsilon-\sigma|](\sigma, \sigma')}_{\left( k +\frac{M}{2}\right) -\sigma', k - \half + \sigma'; |\epsilon' - \sigma'|};
\end{equation}
\begin{equation}
\label{4.24}
f^{[M,n;\epsilon](\sigma, \sigma')}_{k+ \frac{2M-1}{2}, k + \frac{M}{2}, ; \epsilon'} = e^{2 \pi i n \epsilon} e^{-\frac{2 \pi i n}{M} \sigma \sigma'} e^{\frac{2 \pi i n}{M} \frac{M+1}{2} \sigma} f^{[M,n; |\epsilon-\sigma|](\sigma, \sigma')}_{\left( k +\frac{M}{2}\right) -\sigma', k - \half + \sigma'; |\epsilon' - \sigma'|};
\end{equation}
\begin{equation}
\label{4.25}
f^{[M,n;\epsilon] (\half, \half)}_{k + \frac{2M-1}{2}, k + \frac{M}{2}; \epsilon'} = e^{2 \pi in n \epsilon} e^{\frac{\pi i n }{2}} f^{[M,n; \half - \epsilon](\half, \half)}_{k + \frac{M-1}{2}, k ; \half - \sigma'}.
\end{equation}
\end{remark}

\section{Example: $ \fg = spo_{2|3}, m = - \frac{3}{4} $ and $ N=3,\, c =1$} 

In this section we shall need the following signed mock theta functions, studied in \cite{KW16}, cf. \eqref{1.8}:
\begin{equation}
\label{5.01}
\Phi^{\pm[m;s]}_1 (\tau, z_1, z_2) = \sum_{j \in \ZZ } (\pm 1)^j \frac{e^{2 \pi i m j (z_1 + z_2) + 2 \pi is z_1} q^{j^2m + js}}{1 - e^{2 \pi i z_1}q^j}, \mbox{ where } m \in \half \ZZ_{>0},\, s \in \half \ZZ.
\end{equation}
As in \eqref{1.9}, we let
\begin{equation}
\label{5.02}
\Phi^{\pm[m;s]} \tzzt = e^{2\pi i m t} \left(  \Phi^{\pm[m;s]}_1 (\tau, z_1, z_2) -  \Phi^{\pm[m;s]}_1 (\tau, -z_2, -z_1)  \right)
\end{equation}
We have:
\begin{equation}
\label{5.03}
\Phi^{\pm[\half; s]} = \Phi^{\pm[\half; s']} \mbox{ if } s-s' \in \ZZ.
\end{equation}
This can be checked directly (or it follows from \cite{KW16}, Corollary 1.6(a) for $ m = \half $, since $ \Phi^{\pm[\half; s]} = \tilde{\Phi}^{\pm[\half; s']}  $).

Using \eqref{3.1}, the superdenominator identity for the affine Lie superalgebra $ \hat{spo}_{2|3} $ (see \cite{KW94}, and \eqref{3.7}, \eqref{3.8}), we obtain for $ \epsilon = 0 $ or $ \half $, corresponding to non-twisted and twisted sectors respectively:
\begin{equation}
\label{5.04}
\Phi^{\pm[\half; \epsilon]} \tzzt =- ie^{\pi i t} \frac{\eta (\tau)^3 \vartheta_{11} (\tau,z_1 + z_2)}{\vartheta_{11} (\tau, z_1) \vartheta_{11} (\tau, z_2)} \frac{\vartheta_{2 \epsilon, \delta} (\tau, \frac{z_1 - z_2}{2})}
    {\vartheta_{2 \epsilon, \delta} (\tau, \frac{z_1 + z_2}{2})}\, , 
\end{equation}
\noindent where $ \delta = 0 $ (resp. =1) for + (resp. $ - $).

Next, we can rewrite the supercharacter formulas for $ L(\La^{[m;m_2]})$ and
$ L(\dot{\La}^{[m;m_2]}) $, where $ m = -\frac{3}{4} $ and $ m_2 = 0 $ or 1, given by Proposition \ref{prop3.4}, using \eqref{3.2} and \eqref{3.3}:
\begin{equation}
\label{5.05}
(\hat{R}^-\ch^-_{\La^{[m;m_2]}}) \tzzt = i e^{-\frac{\pi i t}{2}} \frac{\eta (2 \tau)^3 \vartheta_{11} (2 \tau, z_1 - z_2) \vartheta_{1-m_2,0} (2 \tau, \frac{z_1 +z_2}{2})}{\vartheta_{11} (2 \tau, z_1) \vartheta_{11} (2 \tau, z_2) \vartheta_{1-m_2,0} (2 \tau, \frac{z_1-z_2}{2})},
\end{equation}
\begin{equation}
\label{5.06}
(\hat{R}^- \ch^-_{\dot{\La}^{[m;m_2]}}) \tzzt = -i e^{-\frac{\pi i t}{2}} \frac{\eta (2 \tau)^3 \vartheta_{11} (2 \tau, z_1 - z_2) \vartheta_{1-m_2 , 0} (2 \tau, \frac{z_1 +z_2}{2})}{\vartheta_{01} (2 \tau, z_1) \vartheta_{01} (2 \tau, z_2) \vartheta_{m_2 , 0} (2 \tau, \frac{z_1-z_2}{2})}.
\end{equation}
Using \eqref{3.7}, we derive from \eqref{5.05} and \eqref{5.06}:
\begin{equation}
\label{5.07}
(\hat{R}^+ \ch^+_{\La^{[m;m_2]}}) \tzzt = i e^{-\frac{\pi i t}{2}} \frac{\eta (2 \tau)^3 \vartheta_{11} (2 \tau, z_1 - z_2) \vartheta_{1-m_2 , 1} (2 \tau, \frac{z_1 + z_2}{2})}{\vartheta_{10} (2 \tau, z_1) \vartheta_{10} (2 \tau, z_2) \vartheta_{1-m_2 , 0} (2 \tau, \frac{z_1 -z_2}{2})},
\end{equation}
\begin{equation}
\label{5.08}
(\hat{R}^+ \ch^+_{\dot{\La}^{[m;m_2]}}) \tzzt = i e^{-\frac{\pi i t}{2}} \frac{\eta (2 \tau)^3 \vartheta_{11} (2 \tau, z_1 - z_2) \vartheta_{1-m_2 , 1} (2 \tau, \frac{z_1 + z_2}{2})}{\vartheta_{00} (2 \tau, z_1) \vartheta_{00} (2 \tau, z_2) \vartheta_{m_2 , 0} (2 \tau, \frac{z_1 -z_2}{2})}.
\end{equation}
\begin{equation}
\label{5.09}
(\hat{R}^\tw \ch^\tw_{\La^{[m;m_2]}}) \tzzt = \\
ie^{-\frac{\pi i t}{2}} q^{-\frac{1}{16}} e^{\frac{\pi i}{4}(z_1 + z_2)}
\frac{\eta (2 \tau)^3 \vartheta_{11} (2 \tau, z_1 - z_2)\vartheta_{1-m_2 , 1} (2 \tau, \frac{z_1 + z_2 - \tau}{2})}
     {\vartheta_{10} (2 \tau, z_1 - \frac{\tau}{2}) \vartheta_{10} (2 \tau, z_2 - \frac{\tau}{2})  
          \vartheta_{1-m_2 , 0} (2 \tau, \frac{z_1 - z_2}{2})}, 
\end{equation}
\begin{equation}
\label{5.10}
(\hat{R}^\tw \ch^\tw_{\dot{\La}^{[m;m_2]}}) \tzzt =\\- ie^{-\frac{\pi i t}{2}} q^{-\frac{1}{16}} e^{\frac{\pi i}{4}(z_1 + z_2)}
\frac{\eta (2 \tau)^3 \vartheta_{11} (2 \tau, z_1 - z_2)  \vartheta_{1-m_2 , 1} (2 \tau, \frac{z_1 + z_2 - \tau}{2})}
     {\vartheta_{00} (2 \tau, z_1 - \frac{\tau}{2}) \vartheta_{00} (2 \tau, z_2 - \frac{\tau}{2})   
          \vartheta_{m_2 , 0} (2 \tau, \frac{z_1 - z_2}{2})}.
\end{equation}

It is shown in \cite{AKMPP} that $ m = -\frac{3}{4} $ is a collapsing level of the simple minimal $ W $-algebra $ W_m (spo_{2|3}, \theta). $ Namely, this vertex algebra is isomorphic to its even part, which is the affine vertex algebra $ V_1 (s\ell_2), $ associated to the affine Lie algebra $ \hat{s\ell}_2 $ of level 1. 

Hence all the irreducible $ W_{-3/4} (spo_{2|3}, e_{-\theta}) $-modules are just the two fundamental irreducible $ V_1 (s\ell_2) $-modules. Their characters are \cite{K90}:
\[ \frac{\Theta_{s,1} (\tau, z)}{\eta (\tau)}, \mbox{ where } s= 0 \mbox{ or }1. \]
It is not difficult to deduce from this the following character formulas for the $ N=3 $ modules with $C=1$:
\[ 
\begin{aligned}
\ch^\pm_{H (\La^{[m;0]})} (\tau, z) = & \ \ch^\pm_{H(\dot{\La}^{[m;1]})} (\tau, z) = \ch_{H^{\tw} (\La^{[m;1]})} (\tau, z) = \frac{\Theta_{0,1} (\tau, z)}{\eta (\tau)}; \\
\ch^\pm_{H (\La^{[m;1]})} (\tau,z) = & \ \ch^\pm_{H(\dot{\La}^{[m;0]})} (\tau, z) = \ch_{H^{\tw}(\dot{\La}^{[m;1]})} (\tau,z) = \frac{\Theta_{1,1} (\tau, z)}{\eta (\tau)};\\
\ch_{H^{\tw} (\La^{[m;0]})}   = & \ \ch_{H^{\tw} (\dot{\La}^{[m;0]})} = 0. \\
\end{aligned} \]
Comparing with \eqref{4.4}, \eqref{4.5}, \eqref{5.05}--\eqref{5.08} in the non-twisted case, and with \eqref{4.12}, \eqref{4.13}, \eqref{5.09}, \eqref{5.10} in the twisted case, we obtain some theta function identities. 

\section{Integrable $ \wg $-modules, where $ \fg = ps\ell_{2|2} $}

As in \cite{KW14}, Section 8, choose the set of simple roots of $ \fg = ps\ell_{2|2}  $ to be  $ \Pi = \{ \al_1, \al_2, \al_3 \} $, the highest root $ \theta $ being $ \al_1 + \al_2 + \al_3, $ where:
\[ (\al_1 | \al_1) = (\al_3 | \al_3) = (\al_1 | \al_3) = 0, \quad (\al_2 | \al_2) = -2, \quad (\al_1| \al_2) = (\al_2|\al_3) = 1, \quad (\theta | \theta ) = 2; \]
\[ \sdim \fg = -2, \quad \ell : = \mbox{rank } \fg = 2, \quad \mbox{defect } \fg = 2, \quad h^\vee = 0, \quad 2 \rho = 2 \wrh = -\al_1-\al_3. \]

We shall use coordinates \eqref{2.1} on $ \wh = \wh^* $ with 
\begin{equation}
\label{6.01}
z = -\half (z_1 - z_2) (\al_1 + \al_3) - z_1 \al_2,
\end{equation}
\noindent so that
\[ e^{-\al_1(h)} = e^{-\al_3 (h)} = e^{2 \pi i z_1}, \quad e^{-\al_2(h)} = 
e^{-2 \pi i (z_1 + z_2)}, \mbox{ and } (z|z) = -2z_1 z_2. \]

The $ \wg $-modules $ L(m\La_0) $ with $  m \in \ZZ $ have been discussed in \cite{KW14}, Section 8. They are integrable with respect to $ \theta $ and 
$ \delta - \theta $ if $ m \geq 0 $ (resp. $ \al_2 $ and $ \delta - \al_2 $) if $ m \leq 0,  $ and atypical with respect to $ \al_1 $ and $ \al_3. $ 
We showed in \cite{KW14}, Section 9, that the modified (super)characters of (non)twisted $ N=4 $ modules, obtained by (twisted) QHR from the $ \wg $-modules $ L(m\La_0) $ with $ m \in \ZZ_{<0} $ form a modular invariant family. In the present paper the process of modification is simplified and the QHR is applied to a larger set of integrable $ \wg $-modules of level $ m \in \ZZ_{<0}. $

Given $ m \in \CC, $ we denote by $ P^m_+ $ the set $ \bmod \,\CC \delta  $ of $ \La \in \wh^* $ of level $ m, $ such that $ L(\La) $ is integrable with respect to $ \al_2 $ and $ \delta - \al_2, $ and atypical with respect to $ \al_1 $ and $ \al_3, $ i.e. $ (\la + \wrh|\al_i) = 0 $ for $ i = 1,3. $ The following proposition is proved in the same way as Proposition \ref{prop3.1}.
\begin{proposition}
\label{prop6.1}
\begin{enumerate}
\item[(a)] The set $ P^m_+ $ is non-empty iff $ m \in \ZZ_{\leq 0} $
\item[(b)] Let $ m \in \ZZ_{\leq 0} $ and let $ \La^{[m;m_2]} = m\La_0- \half m_2 (\al_1 + \al_3)=(m+m_2)\Lambda_0+m_2\Lambda_2. $ Then $ P^m_+ $ consists of weights $ \La^{[m;m_2]}  $, where $ m_2 \in \ZZ$ and $ 0 \leq m_2 \leq - m $.
\end{enumerate}
\qed 
\end{proposition}

We proceed as in Section 3 to write down the characters of the $ \wg $-modules $ \LLa, $ where $ \La \in P^m_+, m \neq 0. $ Introduce the functions
\[ F^{[m;m_2]} = \sum_{j \in \ZZ} t_{-j \al_2}
\frac{e^{\La^{[m;m_2]}+ \wrh}}{(1-e^{-\al_1})(1-e^{-\al_3})}, \]
and the differential operator
\[D_0=\frac{1}{2 \pi i } \left( \frac{\partial}{\partial z_1} - \frac{\partial}{\partial z_2}\right).  
\]
It is straightforward to show, using \eqref{2.2}, that in coordinates 
\eqref{2.1}, \eqref{6.01} we have,
cf. \cite{KW14}, Section 8,
\begin{equation}
\label{6.02}
F^{[m;m_2]} \tzzt = ( -m_2 +D_0)\Phi^{[-m;m_2]}_1 (\tau, z_1, z_2, -t), 
\end{equation}
\begin{equation}
\label{6.03}
r_{\al_2} F^{[m;m_2]} \tzzt = ( -m_2 + D_0)
(\Phi^{[-m;m_2]}_1 (\tau, -z_2, -z_1, -t)).
\end{equation}

Recall from \cite{KW14}, (8.4), the normalized affine non-twisted and twisted $ ps\ell_{2|2} $ denominators and superdenominators, where $ \epsilon, \epsilon' = 0 $ or $ \half: $
\begin{equation}
\label{6.04}
\hat{R}^{(\epsilon)}_{\epsilon'} (\tau, z_1, z_2) = (-1)^{2 \epsilon'} \eta (\tau)^4 \frac{\vartheta_{11} (\tau, z_1-z_2) \vartheta_{11} (\tau, z_1 +z_2)}{\vartheta_{1-2\epsilon', 1-2 \epsilon} (\tau, z_1)^2 \vartheta_{1-2 \epsilon', 1 - 2 \epsilon} (\tau, z_2)^2}.
\end{equation}
The non-twisted affine superdenominator (resp. denominator) is $\hat{R}^-=\hat{R}^{(0)}_0 $ (resp. $\hat{R}^+ = \hat{R}^{(\half)}_0 $).

One can show that for $\Lambda=\Lambda^{[m;m_2]}$ formulas 
(\ref{2.4}) and(\ref{2.6}) hold with $W^\#=\{1, r_{\alpha_2}\}$ and $j_\Lambda
=1$. Hence,
using \eqref{6.02} and \eqref{6.03} we can rewrite formula \eqref{2.4} as follows: 
\begin{equation}
\label{6.05}
\left(\hat{R}^- \ch^-_{\La^{[m;m_2]}}\right) \tzzt = ( -m_2 + D_0)
\Phi^{[m;m_2]} (\tau, z_1, z_2, -t),
\end{equation}
\begin{equation}
\label{6.06}
\left(\hat{R}^+ \ch^+_{\La^{[m;m_2]}}\right) \tzzt = -(-1)^{m_2} ( -m_2  + D_0)
\Phi^{[m;m_2]} \left(\tau, z_1 + \thalf, z_2 + \thalf, -t \right).
\end{equation}

As before, we define the modified (super)characters $ \tilde{\ch}^\pm_{\La^{[m;m_2]}} $ by replacing the functions $ \Phi^{[-m;m_2]} $ in the RHS of \eqref{6.05} and \eqref{6.06} by $ \tilde{\Phi}^{[-m;m_2]}. $

In order to study modular transformation properties of the modified characters of the quantum Hamiltonian reduction of $ \wg $-modules $ \LLa, \La \in P^m_+ $ we need the following lemma. 
\begin{lemma}
\label{lem6.2}
Let $ f\tzzt = e^{2 \pi i m t} f(\tau, z_1, z_2, 0) $ be a differentiable in $ z_1 $ and $ z_2 $ function, and let
\begin{equation}
\label{6.07}
g \tzzt =  (D_0 + \frac{m(z_1-z_2)}{2 \tau})  f \tzzt .
\end{equation}
\begin{enumerate}
\item[(a)] If the function $ f $ satisfies the modular transformation property
\begin{equation}
\label{6.08} f \left(-\frac{1}{\tau}, \frac{z_1}{\tau}, \frac{z_2}{\tau}, t + \frac{z_1z_2}{\tau}\right) = \tau f \tzzt, 
\end{equation}
then the function $ g $ satisfies the following modular transformation propery
\begin{equation}
\label{6.09}
g \left( - \frac{1}{\tau}, \frac{z_1}{\tau}, \frac{z_2}{\tau}, t + \frac{z_1z_2}{\tau}\right) = \tau^2 g \tzzt. 
\end{equation}
\item[(b)] If $ f (\tau, z_2, z_1, t) = f \tzzt,$ and it satisfies the elliptic transformation property for all $ a, b \in \ZZ: $
\begin{equation}
\label{6.10}
f(\tau, z_1 + a \tau, z_2 + b \tau, t - bz_1 -az_2) = q^{mab} f\tzzt.
\end{equation}
then the function $ g $ satisifies the following relations for all 
$ a,b \in \ZZ  $:
\begin{equation}
\label{6.11}
g (\tau, z+ a \tau, z + b \tau, t - (a+b) z) = \frac{m}{2} (b-a) q^{mab} f (\tau, z,z, t),
\end{equation}
\begin{equation}
\label{6.12}
g \left(\tau, z +\! \thalf \!+ a \tau, z + \!\thalf \!+ b \tau, t (a+b) z \right) = \frac{m}{2} (b-a) e^{\pi i m (a+b)} q^{mab} f \left(\tau, z + \!\thalf, z + \!\thalf, t\right).
\end{equation}
\end{enumerate}
\end{lemma}
\begin{proof}
It is a straightforward verification. 
\end{proof}

Let $ g^{[-m]} \tzzt $ be the function, defined by (\ref{6.07}) for $ f \tzzt = \tilde{\Phi}^{[-m;0]} (\tau, z_1, z_2, -t), $ where $ m $ is a negative integer. 
\begin{lemma}
\label{lem6.3}
The function $ g^{[-m]} \tzzt $ satisfies \eqref{6.09}, \eqref{6.11} and \eqref{6.12}.
\end{lemma}
\begin{proof}
It follows from Lemma \ref{lem6.2} and the modular and elliptic transformation properties of the function $ \tilde{\Phi}^{[-m;0]} \tzzt, $ see \cite{KW14}, Corollary 5.11.
\end{proof}
\begin{lemma}
\label{lem6.4}
Let $ m $ be a negative integer. Then
\[\begin{aligned}
g^{[-m]} \left( \tau, z + \frac{\tau}{2}, z - \frac{\tau}{2}, \frac{\tau}{4}\right) = & -\frac{m}{2} e^{-2 \pi i m z } q^{-\frac{m}{4}} \tilde{\Phi}^{[-m;0]} \left( \tau, z + \frac{\tau}{2}, z + \frac{\tau}{2}, 0 \right) \\
 = &  -\frac{m}{2} e^{2 \pi i m z } q^{-\frac{m}{4}} \tilde{\Phi}^{[-m;0]} \left( \tau, z - \frac{\tau}{2}, z - \frac{\tau}{2}, 0 \right).
\end{aligned} 
\]
\end{lemma}

\begin{proof}
  Replacing
  $z-\frac{\tau}{2}$ by $(z+\frac{\tau}{2})-\tau$ in  
  $ g^{[-m]} ( \tau, z+ \tot,  z - \tot , \frac{\tau}{4} )$
  and using Lemma \ref{lem6.3}, formula \eqref{6.11} gives the first equality, while, replacing $ z - \tot $ by $ (z - \tot) + \tau,  $ and using this lemma,
  gives the second equality. 
\end{proof}

\begin{remark}
\label{rem6.5}
Let $ m $ be a negative integer. Then
\[ 
\begin{aligned}
\tilde{\Phi}^{[-m;0]} \left( \tau, z + \frac{\tau}{2}, z - \frac{\tau}{2}, -\frac{\tau}{4}  \right)   = & e^{-2 \pi i m z} q^{-\frac{m}{4}} \tilde{\Phi}^{[-m;0]} \left( \tau, z+ \frac{\tau}{2}, z+ \frac{\tau}{2}, 0 \right)\\
 = & e^{2 \pi i m z} q^{-\frac{m}{4}} \tilde{\Phi}^{[-m;0]} \left( \tau, z- \frac{\tau}{2}, z- \frac{\tau}{2}, 0 \right).
\end{aligned} 
\]
\end{remark}
\begin{proof}
Recalling the elliptic transformation property for $ m \in \zp $ and $ a,b \in \ZZ $
\[ \tilde{\Phi}^{[m;0]} (\tau, z_1 + a \tau, z_2 + b \tau, t) = e^{2 \pi i m (bz_1 + az_2)} q^{-mab} \tilde{\Phi}^{[m;0]} \tzzt, \]
we obtain the first (resp. second) equality by replacing
$ z - \tot $ (resp. $ z + \tot $)
by $\left(  z + \tot\right) - \tau $ (resp. by $ \left( z - \tot \right) + \tau $) in $ \tilde{\Phi}^{[-m;0]} \left( \tau, z + \tot, z - \tot, - \tof \right) . $
\end{proof}

In order to perform the twisted quantum Hamiltonian reduction \cite{KW05}, as in \cite{KW14}, Section 8, we define the twist by $ w_0 :=  r_{\al_2} t_{\half \al_2}, $ so that
$ G^{[-\al_1] \tw}_0, G_0^{[-\al_3] \tw} $ are the annihilation operators and $ G_0^{[-(\al_1 + \al_2)] \tw}, G_0^{[-(\al_2 + \al_3)] \tw} $ are the creation operators. 
Then the modified (non-)twisted (super)characters are given by the following proposition. 
\begin{proposition}
\label{prop6.6}
Let $ \La = \La^{[m;m_2]} \in P^m_+;$ then
\begin{enumerate}
\item[(a)] \[ 
\begin{aligned}
& \left( \hat{R}^+ \tilde{\ch}^+_\La \right) \tzzt = \\
& \ g^{[-m]} \left( \tau, z_1 + \half, z_2 + \half, t\right) - \left( \frac{m (z_1-z_2)}{2 \tau} + m_2\right) \tilde{\Phi}^{[-m;0]} \left( \tau, z_1 + \half, z_2 + \half, - t \right); \\
& \left( \hat{R}^- \tilde{\ch}^-_\La \right) \tzzt =  \ g^{[-m]} \tzzt - \left( \frac{m (z_1-z_2)}{2 \tau} + m_2\right) \tilde{\Phi}^{[-m;0]} \left( \tau, z_1, z_2, - t \right).
\end{aligned} \]
\item[(b)] \[ 
\begin{aligned}
& \left( \hat{R}^{+,\tw } \tilde{\ch}^{+,\tw}_\La \right) \tzzt =  \ g^{[-m]} \left( \tau, -z_2 - \half - \frac{\tau}{2}, -z_1 - \half - \frac{\tau}{2}, t - \frac{z_1+z_2}{2} - \frac{\tau}{4} \right) \\
& - \left( \frac{m (z_1-z_2)}{2 \tau} + m_2\right) \tilde{\Phi}^{[-m;0]} \left( \tau, -z_2 -\half - \frac{\tau}{2}, -z_1 -\half - \frac{\tau}{2}, -t + \frac{z_1+z_2}{2} + \frac{\tau}{4} \right) ; \\
&\left( \hat{R}^{-,\tw} \tilde{\ch}_\Lambda^{-,\tw}\right)  \tzzt = g^{[-m]} \left( \tau, -z_2 - \frac{\tau}{2},- z_1 - \frac{\tau}{2}, t - \frac{z_1+z_2}{2} -\frac{\tau}{4} \right) \\
& - \left( \frac{m(z_1-z_2)}{2} + m_2
\right) \tilde{\Phi}^{[-m;0]} \left( \tau, -z_2 - \frac{\tau}{2}, -z_1 -\frac{\tau}{2}, - t + \frac{z_1+z_2}{2} + \frac{\tau}{4} \right),  \end{aligned} \]
\end{enumerate}
where the twisted affine superdenominator (resp. denominator) is 
$ \hat{R}^{-,\tw} = \hat{R}^{(0)}_\half $ (resp.
$ \hat{R}^{+,\tw} = \hat{R}^{(\half)}_\half$), see \eqref{6.04}.
\end{proposition}

\begin{proof}
(a) follows from the character formulas \eqref{6.05}, \eqref{6.06} and the definition of $ g^{[-m]}. $

In order to prove (b) we use the transformation formula of coordinates under the twist $ w_0: $
\[ r_{\al_2} t_{\half \al_2} \tzzt = \left( \tau, -z_2 - \tot, - z_1 - \tot, t - \frac{z_1+z_2}{2} - \tof \right) . \]
Substituting in (a) gives the result.
\end{proof}

\section{QHR of integrable $\hat{ps\ell}_{2|2}$-modules and modular invariant family of $ N=4 $ representations}

Throughout this section $ \fg = ps\ell_{2|2} $ and we use notation of the previous section. Recall that the QHR associates to a $ \wg $-module $ \LLa $ of level $ K $ and the highest weight
\begin{equation}
\label{7.01}
\La = (K-2m_1 + m_2)\La_0 + m_1 (\La_1 + \La_3) + m_2 \La_2,
\end{equation}
a module $ H(\La) $ over the $ N=4 $ superconformal algebra of Neveu-Schwarz type, such that the following properties hold (see \cite{KW14} for more details and further references):
\begin{enumerate}
\item[(i)] the module $ H(\La) $  is either 0 or an irreducible positive energy module;
\item[(ii)] $ H(\La)=0 $ iff $ K-2m_1 + m_2 \in \ZZ_{\geq 0}; $
\item[(iii)] the irreducible module $ H(\La) $  is characterized by three numbers:
\begin{enumerate}
\item[$ (\al) $] the central charge
\begin{equation}
\label{7.02}
c_K = -6 (K+1),
\end{equation} 
\item[$ (\beta) $] the lowest energy
\begin{equation}
\label{7.03}
h_\La = \frac{m_1(m_1-m_2-1)}{K} - m_1 + \frac{m_2}{2},
\end{equation}
\item[$ (\gamma) $] the spin
\begin{equation}
\label{7.04}
s_\La = \La (J_0) = m_2, \mbox{ where } J_0 = -\al_2.
\end{equation}
\end{enumerate}
\item[(iv)] the (super)character of the module $ H(\La) $ is given by the following formula:
\begin{equation}
\label{7.05}
\left( \overset{4}{R}\vphantom{R}^\pm \ch^\pm_{H(\La)} \right) \tz = \left( \hat{R}^\pm \ch^\pm_\La \right)  \left( \tau, z + \frac{\tau}{2}, z - \frac{\tau}{2}, \frac{\tau}{4} \right), 
\end{equation}
where $  \overset{4}{R}\vphantom{R}^+ = \overset{4}{R}\vphantom{R}^{(\half)}_\half, \overset{4}{R}\vphantom{R}^- =  \overset{4}{R}\vphantom{R}^{(0)}_\half,$ and 
\begin{equation}
\label{7.06}
\overset{4}{R}\vphantom{R}^{(\epsilon)}_{\epsilon'} \tz = \eta (\tau)^3 \frac{\vartheta_{11} (\tau, 2z)}{\vartheta_{1-2\epsilon', 1 - 2\epsilon} (\tau, z)^2}.
 \end{equation}
\end{enumerate}

Similar results hold in the Ramond twisted case. Twisting the $ \wg $-module $ \LLa $ by $ w_0 = r_{\al_2} t_{\half \al_2}, $ we obtain, by the twisted QHR \cite{KW05}, a module $ H^{\tw} (\La) $ over the $ N=4 $ superconformal algebra of Ramond type, such that properties (i)-(iv) hold with the following changes 
(see \cite{KW14}):
\begin{equation}
\label{7.07}
h^\tw_\La = h_\La - \frac{m_2}{2} - \frac{K+1}{4},
\end{equation}
\begin{equation}
\label{7.08}
s^\tw_\La = -m_2 -K -1,
\end{equation}
and the supercharacter of $ H^\tw (\La) $ is given by the following formula:
\begin{equation}
\label{7.09}
\left( \overset{4}{R}\vphantom{R}^{\pm,\tw} \ch^\pm_{H^\tw (\La)} \right) \tz = \left( \hat{R}^{\pm,\tw} \ch^{\pm,\tw}_\La \right) \left( \tau, z + \frac{\tau}{2}, z - \frac{\tau}{2}, \frac{\tau}{4} \right) ,
\end{equation}
where $ \overset{4}{R}\vphantom{R}^{+,\tw} = \overset{4}{R}\vphantom{R}^{(\half)}_0, \overset{4}{R}\vphantom{R}^{-,\tw} = \overset{4}{R}\vphantom{R}^{(0)}_0. $

\begin{proposition}
\label{prop7.1}
Let $ \La = \La^{[m;m_2]} = m\La_0 - \frac{m_2}{2} (\al_1 + \al_3) = m_0 \La_0 + m_2 \La_2 \in P^m_+ $, where $ m_0 = m + m_2. $ Then the modified $N=4$
(non)twisted (super)characters are given by the following formulas:
\begin{enumerate}
\item[(a)] \[ 
\begin{aligned}
\left( \overset{4}{R}\vphantom{R}^+ \tilde{\ch}^+_{H(\La)} \right) \tz & = -m_0 q^{\frac{m}{4}} \tilde{\Phi}^{[-m;0]} \left( \tau, z + \half + \tot, z + \half - \tot, 0\right) \\
& = -(-1)^m m_0 e^{2 \pi i m z} q^{-\frac{m}{4}} \tilde{\Phi}^{[-m;0]} \left( \tau, z + \half - \tot, z + \half - \tot, 0 \right). 
\end{aligned} \]
\item[(b)]  

$  \left( \overset{4}{R}\vphantom{R}^- \tilde{\ch}^-_{H(\La)} \right) \tz = -m_0 q^{\frac{m}{4}} \tilde{\Phi}^{[-m;0]} \left( \tau, z + \tot, z - \tot, 0 \right).  $

\item[(c)] $ \left( \overset{4}{R}\vphantom{R}^{+,\tw} \tilde{\ch}^+_{H^\tw (\La)}\right) \tz = m_0 \tilde{\Phi}^{[-m;0]} \left( \tau, z + \half, z + \half, 0\right).  $
\item[(d)] $ \left( \overset{4}{R}\vphantom{R}^{-,\tw} \tilde{\ch}^-_{H^\tw (\La)} \right) \tz = m_0 \tilde{\Phi}^{[-m;0]}  (\tau, z, z, 0).$
\end{enumerate}
\end{proposition}
\begin{proof}
Using \eqref{7.05}, Proposition \ref{prop6.6}, and Lemma \ref{lem6.4}, we have:
\[ 
\begin{aligned}
\left( \overset{4}{R}\vphantom{R}^+ \tilde{\ch}^+_{H(\La)} \right) \tz  = &   -\frac{m}{2} e^{-2 \pi i m (z + \half)} q^{- \frac{m}{4}} \tilde{\Phi}^{[-m;0]} \left( \tau, z + \half + \tot, z + \half + \tot, 0\right) \\
& - \left( \frac{m}{2} + m_2 \right) \tilde{\Phi}^{[-m;0]}  \left( \tau, z + \half + \tot, z + \half - \tot, - \tof \right). 
\end{aligned} \]
Using Remark \ref{rem6.5}, we conclude that the first formula in (a) holds. The proof of the remaining formulas is the same. 
\end{proof}

\begin{proposition}
\label{prop7.2}
Let $ \La = \La^{[m;m_2]} \in P^m_+. $ Then
\begin{enumerate}
\item[(a)] \[ 
\begin{aligned}
\left( \overset{4}{R}\vphantom{R}^+ \tilde{\ch}^+_{H(\La)} \right) \left( -\frac{1}{\tau}, \frac{z}{\tau} \right) & = (-1)^m \tau e^{-\frac{2 \pi i m }{\tau} z^2} \left( \overset{4}{R}\vphantom{R}^+ \tilde{\ch}^+_{H(\La)} \right) \tz, \\
\left( \overset{4}{R}\vphantom{R}^- \tilde{\ch}^-_{H(\La)} \right) \left( -\frac{1}{\tau}, \frac{z}{\tau} \right)& =  - \tau e^{-\frac{2 \pi i m }{\tau}z^2} \left( \overset{4}{R}\vphantom{R}^{+,\tw} \tilde{\ch}^+_{H^\tw(\La)} \right) \tz,  \\
\left( \overset{4}{R}\vphantom{R}^{+,\tw}\tilde{\ch}^+_{H^\tw(\La)} \right) \left( -\frac{1}{\tau}, \frac{z}{\tau} \right) & = - \tau e^{-\frac{2 \pi i m }{\tau}z^2}  \left( \overset{4}{R}\vphantom{R}^- \tilde{\ch}^-_{H(\La)} \right) \tz, \\ 
\left( \overset{4}{R}\vphantom{R}^{-,\tw}\tilde{\ch}^-_{H^\tw(\La)} \right) \left( -\frac{1}{\tau}, \frac{z}{\tau} \right) & = \tau e^{-\frac{2 \pi i m }{\tau}z^2} \left( \overset{4}{R}\vphantom{R}^{-,\tw}\tilde{\ch}^-_{H^\tw(\La)} \right) \tz. \\
\end{aligned} \]
\item[(b)]\[ 
\begin{aligned}
\left( \overset{4}{R}\vphantom{R}^\pm \tilde{\ch}^\pm_{H(\La)} \right) (\tau + 1, z) & = e^{\frac{\pi i m}{2}}  \left( \overset{4}{R}\vphantom{R}^\mp \tilde{\ch}^\mp_{H(\La)} \right) \tz, \\
\left( \overset{4}{R}\vphantom{R}^{\pm,\tw} \tilde{\ch}^\pm_{H^\tw(\La)} \right) (\tau + 1, z) & = \left( \overset{4}{R}\vphantom{R}^{\pm,\tw} \tilde{\ch}^\pm_{H^\tw(\La)} \right) (\tau, z).
\end{aligned} \]
\end{enumerate}
\end{proposition}
\begin{proof}
Using the second equality in Proposition \ref{prop7.1}(a), we obtain:
\[ \left( \overset{4}{R}\vphantom{R}^+ \tilde{\ch}^+_{H(\La)} \right) \left( -\frac{1}{\tau}, \frac{z}{\tau} \right) = -m_0 (-1)^m e^{2 \pi i m \frac{z}{\tau}} e^{\frac{2 \pi i m}{4 \tau}} \tilde{\Phi}^{[-m;0]} \left( -\frac{1}{\tau}, \frac{z}{\tau} + \half + \frac{1}{2 \tau}, \frac{z}{\tau} + \half + \frac{1}{2 \tau} \right) .  \]
Applying to the RHS the transformation formula \eqref{1.15} of $ \tilde{\Phi}^{[m;0]},  $ we obtain:
\[ \begin{aligned}
& \left(\overset{4}{R}\vphantom{R}^+ \tilde{\ch}^+_{H(\La)} \right) \left( -\frac{1}{\tau}, \frac{z}{\tau} \right) \\
& = (-1)^m \tau e^{-\frac{2 \pi i m}{\tau} z^2} (-m_0) e^{-2 \pi i m (z + \half)} q^{-\frac{m}{4}} \tilde{\Phi}^{[-m;0]} \left( \tau, z + \tot + \half, z + \tot + \half \right) .
\end{aligned}  \]
This gives the first formula in (a), due to the second equality in Proposition \ref{prop7.1}(a).
The proof of all other equations in (a) is similar. 

In order to prove (b) we use the transformation formulas \eqref{1.13} and \eqref{1.16} of $ \tilde{\Phi}^{[m;0]}. $ Applying the first equation in Proposition \ref{prop7.1}(a) and using \eqref{1.13}, \eqref{1.16}, we obtain:
\[ 
\begin{aligned}
\left(\overset{4}{R}\vphantom{R}^+ \tilde{\ch}^+_{H(\La)} \right) (\tau + 1, z) & = -m_0 e^{2 \pi i (\tau+1) \frac{M}{4}} \tilde{\Phi}^{[-m;0]} \left( \tau +1, z + \tot +1, z - \tot \right) \\
& = e^{\frac{\pi i m }{2}} \left( -m_0 q^{\frac{M}{4}} \tilde{\Phi}^{[-m;0]} \left( \tau, z + \tot, z - \tot \right)\right) ,
\end{aligned} \]
which gives the result due to the equation in Proposition \ref{prop7.1}(b). The proof of all the other equations in (b) is similar. 
\end{proof}

In order to write down the modular transformation formulas for the modified $ N =4 $ untwisted and twisted (super)characters, it is convenient to use the ``$ \epsilon $-notation'': $ \tilde{\ch}^{(\epsilon)}_{H(\La), \half}  = \tilde{\ch}^\pm_{H(\La)}$ and $\tilde{\ch}^{(\epsilon)}_{H(\La), 0}  =
\tilde{\ch}^\pm_{H^\tw(\La)}  $, where $ \epsilon = \half $ for + and $ \epsilon = 0 $ for $ - $.
\begin{theorem}
\label{th7.3}
Let $ m  $ be a negative integer and let $ \La = \La^{[m;m_2]} \in P^m_+. $ Then
\begin{enumerate}
\item[(a)] The $ N = 4 $ modules $ H(\La) $ (resp. $ H^\tw (\La) $) are zero iff $ m+m_2 \in \ZZ_{\geq 0}. $ Otherwise, they are irreducible with characteristic numbers
\[ c_m = -6(m+1), \ h_\La = \frac{m_2}{2} (\mbox{resp. } h^\tw_\La = -\frac{m+1}{4}), \ s_\La = m_2\, (\mbox{resp. } s^\tw_\La = -m_2 - m - 1). \]
\item[(b)] One has the following modular transformation formulas:
\begin{equation}
\label{7.10}
\tilde{\ch}^{(\epsilon)}_{H(\La), \epsilon'} 
\left( -\frac{1}{\tau}, \frac{z}{\tau} \right) = -(-1)^{2\epsilon+2\epsilon(m+1)(1-2\epsilon')} 
e^{\frac{\pi i c_m}{3\tau}z^2} 
\tilde{\ch}^{(\epsilon')}_{H(\La), \epsilon} \tz;
\end{equation}
\begin{equation}
\label{7.11}
\tilde{\ch}^{(\epsilon)}_{H(\La), \epsilon'} (\tau +1, z) = e^{\frac{\pi i m}{2} (1-2\epsilon')-\pi i \epsilon'} \tilde{\ch}^{(|\epsilon-\epsilon'|)}_{H(\La), \epsilon'} \tz.
\end{equation}
\end{enumerate}
\end{theorem}
\begin{proof}
(a) follows from the above description of the $ N=4 $ QHR.
(b) follows from Proposition \ref{prop7.2} and the modular transformation formulas of the functions $ \overset{4}{R}\vphantom{R}^{(\epsilon)}_{\epsilon'}, $ given by \cite{KW14}, Lemma 10.1.
\end{proof}
\begin{remark}
\label{rem7.4} 
If in place of the twist $ w_0 $ we take a naturally looking twist $ w'_0 = t_{\frac{\theta}{2}} r_\theta, $ then the modular invariance as in Theorem \ref{th7.3}(b) does not hold any longer. 
\end{remark}

\section{Principal admissible modules over $ \wg $, where $ \fg = ps\ell_{2|2}. $}

Throughout this section $ \fg = ps\ell_{2|2} $ and we use notation of the previous two sections.

First we list the principal simple subsets $ {\Pi}^{\mathrm{I-IV}}_{k_1, k_2} = t_\beta y S_{(M)} \subset \hat{\Delta}_+, $ where $ M $ is a positive integer (cf. Section 2) and $ k_3 = k_1: $

\[ \begin{aligned}
\Pi^{\mathrm{I}}_{k_1,k_2} & = \lbrace  k_0 \delta + \al_0, k_1 \delta + \al_1, k_2 \delta + \al_2, k_1 \delta + \al_3 \rbrace,\,   M = \sum_{i = 0}^{3} k_i + 1,\\
 y & = 1,\, \beta = - \frac{2k_1 + k_2}{2} (\al_1+\al_3) -k_1 \al_2;
\end{aligned}   \] 
\[ \begin{aligned}
\Pi^{\mathrm{II}}_{k_1,k_2} & = \left\lbrace  k_0 \delta - \al_0, k_1 \delta-\al_1, k_2 \delta - \al_2, k_1 \delta - \al_3 \right\rbrace,\, M = \sum_{i = 0}^{3} k_i - 1, \\
y & = r_\theta r_{\al_2},\, \beta = \frac{2k_1 + k_2}{2} (\al_1 + \al_3) + k_1 \al_2 ; 
\end{aligned}  \] 
\[ \begin{aligned}
\Pi^{\mathrm{III}}_{k_1,k_2} & = \left\lbrace k_0 \delta + \al_0, k_1 \delta + \al_1 + \al_2, k_2 \delta - \al_2, k_1 \delta + \al_2 + \al_3 \right\rbrace,\, M = \sum_{i = 0}^{3} k_i + 1, \\
y & = r_{\al_2},\, \beta = - \frac{2k_1 + k_2}{2} (\al_1 + \al_3) - (k_1 + k_2) \al_2;  
\end{aligned} \] 
\[ \begin{aligned}
\Pi^{\mathrm{IV}}_{k_1,k_2} & = \left\lbrace k_0 \delta - \al_0, k_1 \delta - (\al_1 + \al_2), k_2 \delta + \al_2, k_1 \delta - (\al_2 + \al_3) \right\rbrace ,\, M = \sum_{i = 0}^{3} k_i - 1, \\
y & = r_\theta,\, \beta = \frac{2k_1 +k_2}{2} (\al_1 + \al_3) + (k_1 +k_2) \al_2.
\end{aligned}  \] 
Note that in all cases we have:
\begin{equation}
\label{8.01}
y^{-1} \beta = - \frac{2k_1 + k_2}{2} (\al_1 + \al_3) - k_1 \al_2 \mbox{ and } |\beta|^2 = 2k_1 (k_1 + k_2).
\end{equation}

The explicit formulas for the corresponding principal admissible level $ K = \frac{m}{M}  $ weights
\begin{equation}
\label{8.02}
 \La^{[m;m_2] \mathrm{I-IV}}_{k_1,k_2} = t_\beta y \left( \La^{[m;m_2]} - (M-1) K \La_0 + \rho \right) - \rho, 
\end{equation} 
associated to the integrable weight $ \La^{[m;m_2]} \in P^m_+ $ are as follows:
\[ 
\begin{aligned}
\La^{[m;m_2]\mathrm{I}}_{k_1, k_2}  = \ &K \La_0 - \frac{K(2k_1+k_2)+m_2}{2} \theta + \frac{Kk_2+m_2}{2} \al_2 - \left( Kk_1 (k_1 +k_2) + (m_2+1)k_1 \right) \delta; \\
\La^{[m;m_2]\mathrm{II}}_{k_1, k_2}  = \ &K \La_0 + \frac{K(2k_1+k_2)+m_2 + 2}{2} \theta - \frac{Kk_2+m_2+2}{2} \al_2 \\
 &- \left( Kk_1 (k_1+k_2) + (m_2+1)k_1 \right) \delta; \\
 \La^{[m;m_2]\mathrm{III}}_{k_1, k_2}  = \ &K \La_0 - \frac{K(2k_1+k_2)+m_2}{2} \theta - \frac{Kk_2+m_2+2}{2} \al_2 \\
 & - \left( Kk_1 (k_1 + k_2) + (m_2+1)k_1 \right) \delta; \\
 \La^{[m;m_2]\mathrm{IV}}_{k_1, k_2}  = \ &K \La_0 + \frac{K(2k_1 + k_2)+m_2+2}{2} \theta + \frac{Kk_2+m_2}{2} \al_2 \\
 &- \left( Kk_1 (k_1+k_2) + (m_2 +1)k_1 \right) \delta. 
\end{aligned}
 \]

For $\La= \La^{[m;m_2]\mathrm{I-IV}}_{k_1, k_2}, $ the weights 
$ \La^\tw = r_{\al_2} t_{\half \al_2} (\La)  $ are as follows:

\[ 
\begin{aligned}
\La^{[m;m_2]\mathrm{I}\tw}_{k_1, k_2} = \ & K\La_0 - \frac{K(2k_1+k_1)+m_2}{2} \theta + \frac{K(k_2 +1)+m_2}{2} \al_2 \\
&- \left( Kk_1 (k_1 + k_2) + (m_2 +1)k_1 - \frac{K}{4} - \frac{Kk_2+m_2}{2}\right) \delta; \\
\La^{[m;m_2]\mathrm{II}\tw}_{k_1, k_2} = \ & K\La_0 + \frac{K(2k_1 + k_2)+m_2 +2}{2} \theta - \frac{K(k_2-1)+m_2 +2}{2} \al_2  \\
&- \left( Kk_1 (k_1 + k_2) + (m_2 +1)k_1 -\frac{K}{4} + \frac{Kk_2 + m_2 + 2}{2}\right) \delta; \\
\La^{[m;m_2]\mathrm{III}\tw}_{k_1, k_2} = \ & K\La_0 - \frac{K(2k_1 + k_2)+m_2}{2} \theta - \frac{K(k_2-1)+m_2 +2}{2} \al_2\\
&- \left( Kk_1 (k_1+k_2) +(m_2+1)k_1 -\frac{K}{4} + \frac{Kk_2+m_2+2}{2}\right) \delta ; \\
\La^{[m;m_2]\mathrm{IV}\tw}_{k_1, k_2} = \ & K\La_0 + \frac{K(2k_1+k_2) +m_2 +2}{2} \theta + \frac{K(k_2+1)+m_2}{2} \al_2 \\
&- \left( Kk_1 (k_1+k_2) + (m_2 +1)k_1 -\frac{K}{4} - \frac{Kk_2+m_2}{2}\right) \delta. \\
\end{aligned} \]

Recall that all the principal admissible weights $ \La = \La^{[m;m_2]\mathrm{I-IV}}_{k_1, k_2} $ are associated to the integrable weight $ \La^0 = \La^{[m;m_2]}. $ Hence we can apply formula \eqref{2.10} and Proposition \ref{prop6.6} to obtain formulas for modified normalized supercharacters of principal admissible $ \wg $-modules.
  \begin{proposition}
 \label{prop8.1}
 \begin{enumerate}
\item[(a)] \[ 
\begin{aligned}
&\left(\hat{R}^- \tilde{\ch}^-_{\La^{[m;m_2]\mathrm{I}}_{k_1, k_2}} \right) \tzzt  = e^{\frac{2 \pi i m t}{M}} e^{\frac{2 \pi i m}{M} ((k_1+k_2)z_1 - k_1 z_2)} q^{\frac{m}{M}k_1 (k_1+k_2)} \\
& \times \left( g^{[-m]} (M\tau, z_1 + k_1 \tau, z_2 -(k_1+k_2)\tau, 0 ) \right. \\
&- \left. \left( \frac{m(z_1-z_2+(2k_1 + k_2) \tau)}{2 M \tau} + m_2 \right) \tilde{\Phi}^{[-m;0]} (M\tau, z_1 + k_1 \tau, z_2 -(k_1+k_2)\tau, 0 ) \right) .
\end{aligned} \]
\item[(b)] \[ 
\begin{aligned}
&\left(\hat{R}^- \tilde{\ch}^-_{\La^{[m;m_2]\mathrm{II}}_{k_1, k_2}} \right) \tzzt =  e^{\frac{2 \pi i m t}{M}} e^{\frac{2 \pi i m }{M} (-(k_1+k_2)z_1 + k_1 z_2)} q^{\frac{m}{M}k_1 (k_1+k_2)} \\
& \times \left( g^{[-m]} (M\tau, -z_1 + k_1\tau, -z_2 - (k_1+k_2)\tau, 0) \right.\\
& -\left. \left( \frac{m(z_2 - z_1 + (2k_1+k_2) \tau)}{2 M \tau} + m_2\right) \tilde{\Phi}^{[-m_2;0]} (M\tau, -z_1 + k_1\tau, -z_2 - (k_1+k_2)\tau, 0) \right) .
\end{aligned} \]
\item[(c)] \[ 
\begin{aligned}
& \left(\hat{R}^- \tilde{\ch}^-_{\La^{[m;m_2]\mathrm{III}}_{k_1, k_2}} \right) \tzzt = e^{\frac{2 \pi i m t}{M}} e^{\frac{2 \pi i m}{M}(-(k_1+k_2)z_2 + k_1 z_1)} q^{\frac{m}{M}k_1 (k_1+k_2)} \\
& \times \left( g^{[-m]} (M\tau, -z_2 + k_1\tau, -z_1 - (k_1+k_2)\tau, 0) \right.\\
& - \left. \left( \frac{m (z_1 - z_2 + (2k_1 + k_2) \tau)}{2 M \tau} + m_2 \right) \tilde{\Phi}^{[-m;0]} (M\tau, -z_2 + k_1\tau, -z_1 - (k_1+k_2)\tau, 0) \right).
\end{aligned} \]
\item[(d)] \[ 
\begin{aligned}
& \left(\hat{R}^- \tilde{\ch}^-_{\La^{[m;m_2]\mathrm{IV}}_{k_1, k_2}} \right) \tzzt = e^{\frac{2 \pi i m t}{M}} e^{\frac{2 \pi i m}{M}((k_1+k_2)z_2 - k_1 z_1)} q^{\frac{m}{M}k_1 (k_1+k_2)} \\
& \times \left( g^{[-m]} (M\tau, z_2 + k_1\tau, z_1 - (k_1+k_2)\tau, 0) \right.\\
& -\left. \left( \frac{m (z_2-z_1+ (2k_1+k_2)\tau)}{2M \tau} + m_2 \right) \tilde{\Phi}^{[-m;0]} (M\tau, z_2 + k_1\tau, z_1 - (k_1+k_2)\tau, 0) \right).
\end{aligned} \]
 \end{enumerate}
 \end{proposition}
 Formulas for the modified normalized characters $ \tilde{\ch}^+_\La $ and the twisted (super)characters $ \tilde{\ch}^{\pm,\tw} $ are obtained from the formulas for the modified normalized supercharacters $ \tilde{\ch}^-_\La, $ given by Proposition \ref{prop8.1}, by the following substitutions (cf. \cite{KW16}, Section 6.6):
 \begin{equation}
 \label{8.03}
 \tilde{\ch}^+_\La \tzzt = e^{\pi i (\La|\al_2)} \tilde{\ch}^-_\La \left(\tau, z_1 + \half, z_2 + \half, t \right); 
 \end{equation}
 \begin{equation}
 \label{8.04}
 \begin{aligned}
  \tilde{\ch}^{\pm,\tw} \tzzt & = \tilde{\ch}^\pm_\La (w_0 \tzzt) \\
  & =\tilde{\ch}^\pm_\La \left( \tau, -z_2 - \tot, -z_1 - \tot, t - \frac{z_1 +z_2}{2} - \tof \right).
 \end{aligned}
  \end{equation}
    Note also that the $ (\La | \al_2), $ occurring in \eqref{8.03}, are given by the following formulas:
  \begin{equation}
  \label{8.05}
  ( \La^{[m;m_2] I\, \mbox{or}\, IV}_{k_1, k_2} |  \al_2 ) = -Kk_2 - m_2, \quad 
( \La^{[m;m_2] II\, \mbox{or}\, III}| \al_2 ) = Kk_2 + m_2 +2. 
  \end{equation}
  
 In order to study the modular transformation properties of the QHR of admissible $ \wg $-modules, we introduce the following functions, where $ \epsilon, \epsilon' = 0 $ or $ \half, $ and $ j,k \in \epsilon' + \ZZ: $
  \begin{equation}
 \label{8.06}
 P^{[M;-m;\epsilon]}_{j,k; \epsilon'} \tzzt = \left( D_0 + \frac{m (z_1 - z_2)}{2 M \tau}\right)  \tilde{\Psi}^{[M;-m;0;\epsilon]}_{j,k; \epsilon'} (\tau, z_1, z_2, -t).
 \end{equation}
  \begin{lemma}
 \label{lem8.2}
 \begin{enumerate}
 \item[(a)] \[ P^{[M;-m;\epsilon]}_{j,k; \epsilon'}(\tau, z_2,z_1,t) = - P^{[M;-m;\epsilon]}_{k,j; \epsilon'} \tzzt .\] 
 \item[(b)] \[ P^{[M;-m;\epsilon]}_{j,k; \epsilon'} (\tau, -z_1, -z_2, t) = P^{[M;-m;\epsilon]}_{-j,-k;\epsilon'} \tzzt .\]
 \end{enumerate}
 \end{lemma}
  \begin{proof}
 Relation (a) follows from the identity
 \begin{equation}
 \label{8.07} \tilde{\Psi}^{[M,-m;0;\epsilon]}_{j,k; \epsilon'} (\tau, z_2, z_1, t) = \tilde{\Psi}^{[M, -m; 0; \epsilon]}_{k,j;\epsilon'} \tzzt, 
 \end{equation}
 while (b) follows from the identity
 \begin{equation}
 \label{8.08}
 \tilde{\Psi}^{[M,-m;0;\epsilon]}_{j,k;\epsilon'} (\tau, -z_1, -z_2, t) =- \tps^{[M;-m;0;\epsilon]}_{-j,-k; \epsilon'} \tzzt. 
  \end{equation}
 \end{proof}
 \begin{lemma}
\label{lem8.3}
 For $ a \in \ZZ, $ we have
\begin{enumerate}
\item[(a)] \[ g^{[-m]} (\tau, z_1 + a, z_2 +a, t) = g^{[-m]} \tzzt.  \]
\item[(b)] \[ P^{[M, -m; \epsilon]}_{j,k;\epsilon} (\tau, z_1 +a, z_2 +a, t) = e^{-\frac{2 \pi i m}{M} (j+k)a} P^{[M,-m;\epsilon]}_{j,k;\epsilon'}\tzzt .\]
\end{enumerate}
\end{lemma}
\begin{proof}
Claim (a) follows from the definition of $ g^{[-m]}: $
\begin{equation}
\label{8.09}
g^{[-m]} \tzzt = \left( D_0 + \frac{m(z_1-z_2)}{2 \tau} \right) \tph^{[-m;0]} (\tau, z_1, z_2, -t),
\end{equation}
and the elliptic transformation property of $ \tilde{\Phi}: $
\[ \tilde{\Phi}^{[-m;0]} (\tau, z_1 + a, z_2 + a, t ) = \tph^{[-m;0]} \tzzt. \]
Claim (b) follows from the definition \eqref{8.06} of  $ P^{[M, -m;\epsilon]}_{j,k; \epsilon'} $ and the elliptic transformation property of $ \tps: $
\[ \tps^{[M,-m;0;\epsilon]}_{j,k;\epsilon'} (\tau, z_1 + a, z_2 + a, t) = e^{-\frac{2 \pi i m }{M}(j+k)a} \tps^{[M, -m;0; \epsilon]}_{j,k;\epsilon'} \tzzt.  \]
\end{proof}
 \begin{lemma}
\label{lem8.4}
\[ \begin{aligned}
&D_0\tilde{\Phi}^{[-m;0]} ( M\tau, z_1 + j \tau + \epsilon, z_2 + k \tau + \epsilon, -\frac{t}{M}) = g^{[-m]} ( M \tau, z_1 + j \tau + \epsilon, z_2 + k\tau + \epsilon, \frac{t}{M} ) \\
&- \left( \frac{m(z_1-z_2)}{2M\tau} + \frac{m(j-k)}{2M} \right)  
\tph^{[-m;0]} ( M\tau, z_1+ j\tau + \epsilon, z_2 + k \tau + \epsilon. - \frac{t}{M}) . 
\end{aligned}
 \]
\end{lemma} 
\begin{proof}
Substitute $ \tau \mapsto M \tau, z_1 \mapsto z_1 + j \tau + \epsilon, z_2 \mapsto z_2 + k \tau + \epsilon $ in formula \eqref{8.09}.
\end{proof}
\begin{lemma}
\label{lem8.5}
\[ 
\begin{aligned}
P^{[M,-m;\epsilon]}_{j,k;\epsilon} \tzzt &= q^{-\frac{m}{M}jk} e^{-\frac{2 \pi i m}{M}(kz_1 + jz_2)} \left( g^{[-m]} ( M \tau, z_1 j\tau + \epsilon, z_2 + k \tau + \epsilon, \frac{t}{M} ) \right. \\
& + \left. \frac{m(j-k)}{2M} \tph^{[-m;0]} ( M\tau, z_1+ j \tau + \epsilon, z_2 + k \tau - \epsilon,- \frac{t}{M} ) \right) .
\end{aligned} \]
\end{lemma}
\begin{proof}
We have:
\[ 
\begin{aligned}
& D_0 \tps^{[M;-m;0;\epsilon]}_{j,k;\epsilon'} \tzzt\\
 & = - \frac{m}{M} (k-j) q^{-\frac{m}{M}jk} e^{- \frac{2 \pi i m}{M} (kz_1 + jz_2)} \tph^{[-m;0]} ( M \tau, z_1 + j\tau + \epsilon, z_2 + k \tau - \epsilon, \frac{t}{M} )  \\
&+ q^{-\frac{m}{M}jk} e^{- \frac{2 \pi i m }{M}(kz_1 + jz_1)} D_0 \tph^{[-m;0]} ( M \tau, z_1 + j\tau + \epsilon, z_2 + k \tau - \epsilon, \frac{t}{M}).
\end{aligned} \]
Plugging this in \eqref{8.06} and using Lemma \ref{lem8.4} and \eqref{1.10}, we obtain the result. 
\end{proof} 
\begin{remark}
\label{rem8.6}
\[ P^{[1, -m;\epsilon]}_{j,k;\epsilon'} (\tau,z,z,t) = 0. \]
This follows by Lemma \ref{lem8.5} for $ z_1 = z_2 = z $ and using the elliptic transformation property of $ \tph^{[-m;0]}. $
\end{remark}
\begin{proposition}
\label{prop8.7}
Let either $ M $ be a positive integer and $ m = -1 $, or  $ M $ be a positive odd integer and $ m \in \ZZ_{\leq -2} $ be coprime to $ M. $ Then
\begin{enumerate}
\item[(a)]  \[ 
P^{[M, -m; \epsilon]}_{j,k;\epsilon'} \left( -\frac{1}{\tau}, \frac{z_1}{\tau}, \frac{z_2}{\tau}, t \right) = \frac{\tau^2}{M} e^{-\frac{2 \pi i m }{M\tau} z_1 z_2} \hspace{-2em}\sum_{(a,b) \in (\epsilon + \ZZ / M\ZZ)^2} \hspace{-2em} e^{\frac{2 \pi i m }{M} (ak + bj)} P^{[M, -m; \epsilon']}_{a,b; \epsilon} \tzzt. 
 \]
\item[(b)] \[ P^{[M, -m; \epsilon]}_{j,k;\epsilon'} (\tau + 1, z_1, z_2, t) = e^{-\frac{2 \pi i m }{M}jk} P^{[M, -m; | \epsilon-\epsilon'|]}_{j,k; \epsilon'} \tzzt. \]
\end{enumerate}
\end{proposition}
\begin{proof}
Applying $ \tau D_0 $ to both sides of \eqref{1.17}, we obtain:
\[ 
\begin{aligned}
& \left( D_0 \tps^{[M,-m;0;\epsilon]}_{j,k;\epsilon'} \right) \left( -\frac{1}{\tau}, \frac{z_1}{\tau}, \frac{z_2}{\tau}, t  \right)  = \\
& \frac{\tau m}{M^2} (z_1 - z_2) e^{-\frac{2 \pi i m}{M \tau} z_1 z_2} \hspace{-2em}\sum_{(a,b) \in (\epsilon + \ZZ / M\ZZ)^2} \hspace{-2em} e^{\frac{2 \pi i m }{M} (ak+bj)} \tps^{[M,-m;0;\epsilon']}_{a,b; \epsilon} \tzzt \\
&+ \frac{\tau^2}{M} e^{-\frac{2 \pi i m}{M\tau}z_1z_2} \hspace{-2em}\sum_{(a,b) \in (\epsilon + \ZZ / M\ZZ)^2} \hspace{-2em} e^{\frac{2 \pi i m}{M} (ak+bj)} D_0 \tps^{[M,-m;0;\epsilon']}_{a,b;\epsilon} \tzzt. 
\end{aligned} \]
Using this and \eqref{1.17} we compute the $ S $-transformation of $ P^{[M,-m; \epsilon]}_{j,k;\epsilon'} $ as follows:
\[
\begin{aligned}
  P^{[M, -m; \epsilon]}_{j,k;\epsilon'} \left( -\frac{1}{\tau}, \frac{z_1}{\tau}, \frac{z_2}{\tau}, t \right) &=
  (D_0 \tps^{[M, -m;0;\epsilon]}_{j,k;\epsilon'}) \left( -\frac{1}{\tau}, \frac{z_1}{\tau}, \frac{z_2}{\tau}, -t \right) \\
&-\frac{m(z_1-z_2)}{2M} \tps^{[M,-m;0;\epsilon]}_{j,k;\epsilon'} \left( -\frac{1}{\tau}, \frac{z_1}{\tau}, \frac{z_2}{\tau}, -t \right) = A+B+C, \end{aligned} \]
where \[ 
\begin{aligned}
A & : = \frac{\tau m}{M^2} (z_1-z_2) e^{-\frac{2 \pi i m}{M \tau} z_1z_2} \hspace{-2em}\sum_{(a,b) \in (\epsilon + \ZZ / M\ZZ)^2} \hspace{-2em} e^{\frac{2 \pi i m }{M} (ak+bj)} \tps^{[M;-m;0;\epsilon']}_{a,b; \epsilon} (\tau, z_1, z_2, -t), \\
B & : =  \frac{\tau^2}{M} e^{-\frac{2 \pi i m}{M \tau}z_1z_2} \hspace{-2em}\sum_{(a,b) \in (\epsilon + \ZZ / M\ZZ)^2} \hspace{-2em} e^{\frac{2 \pi i  m }{M}(ak+bj)} D_0 
\tps^{[M, -m;0;\epsilon']}_{a,b;\epsilon} (\tau, z_1, z_2, -t), \\
C & : = -\frac{m\tau}{2M^2} (z_1-z_2) e^{-\frac{2 \pi i m }{M\tau} z_1z_2} \hspace{-2em}\sum_{(a,b) \in (\epsilon + \ZZ / M\ZZ)^2} \hspace{-2em} e^{\frac{2 \pi i m }{M}(ak+bj)} \tps^{[M;-m;0;\epsilon']}_{a,b; \epsilon} (\tau, z_1, z_2, -t). \\
\end{aligned} \]
We have:
\[ A+C = \frac{\tau^2}{ M}\frac{m}{2 M}\frac{z_1-z_2}{\tau} 
e^{-\frac{2 \pi i m}{M \tau}z_1z_2}  \hspace{-2em}\sum_{(a,b) \in (\epsilon + \ZZ / M\ZZ)^2} \hspace{-2em} e^{\frac{2 \pi i m}{M}(ak+bj)} \tps^{[M, -m;0;\epsilon']}_{a,b;\epsilon} (\tau, z_1, z_2, -t). \]
Hence
\[ 
\begin{aligned}
P^{[M;-m;\epsilon]}_{j,k;\epsilon'} \left( -\frac{1}{\tau}, \frac{z_1}{\tau}, \frac{z_2}{\tau}, t \right) & = \frac{\tau^2}{M} e^{-\frac{2 \pi i m }{M \tau}z_1 z_2} \hspace{-2em}\sum_{(a,b) \in (\epsilon + \ZZ / M\ZZ)^2} \hspace{-2em} e^{\frac{2 \pi i m }{M} (ak+bj)} \\
& \times \left( D_0 \tps^{[M,-m;0;\epsilon']}_{a,b;\epsilon} (\tau, z_1, z_2, -t) + \frac{m(z_1 -z_1)}{2M\tau} \tps^{[M;-m;0;\epsilon']}_{a,b;\epsilon} (\tau, z_1, z_2, -t) \right) ,
\end{aligned} \]
proving (a). Claim (b) follows from \eqref{1.18}.
\end{proof}
\begin{remark}
\label{rem8.8}
For $ a,b \in \ZZ, $ the following formulas hold: 
\[ 
\tps^{[M,-m;s;\epsilon]}_{j+aM, k+bM; \epsilon'}  = e^{2 \pi i m (a-b)\epsilon} \tps^{[M,-m;s;\epsilon]}_{j,k;\epsilon'}; \,P^{[M,-m;\epsilon]}_{j+aM, k+bM; \epsilon'} = e^{2 \pi i m (a-b) \epsilon} P^{[M,-m; \epsilon]}_{j,k;\epsilon'}.
\]
Indeed, the first formula follows from the elliptic transformation property of $\tilde{\Phi}^{[-m;s]}$, while the second formula follows from the first.
\end{remark}

 \section{QHR of admissible $ \hat{ps\ell}_{2|2} $-modules and a new family 
of modular invariant $ N=4 $ representations.}
 
 In this section, $ \fg = ps\ell_{2|2}, m $ is a negative integer, $ M $ is a positive integer, and we apply the QHR to the $ \wg $-modules $ \LLa,  $ where $ \La $ are principal admissible weights of level $ K = \frac{m}{M}, $ associated to the integrable weights $ \La^{[m;m_2]} \in P^m_+ $ (given by Proposition \ref{prop6.1}). The former are given by \eqref{8.02} and listed after that formula. We keep notation of the previous three sections.
 
 \begin{lemma}
 \label{lem9.1}
 Let $ \La = \La^{[m;m_2]J}_{k_1,k_2}, $ where $ J =  $ I - IV, is a principal admissible weight, such that $ \La^0 = \La^{[m;m_2]} \in P^m_+$ and assume that $ H(\La) \neq 0. $ Then
 \begin{enumerate}
\item[(a)] $ h_\La = K(k_1 \pm \half) (k_1 + k_2 \pm \half) +(k_1\pm\half)( m_2 + 1) - \frac{K+2}{4}, $ where + (resp. $ - $) occurs if $ J =  $ I or III (resp. II or IV);

  $ s_\La = Kk_2 + m_2 $ (resp. $ =- Kk_2-m_2 -2 $) if $ J =  $ I or IV (resp. = II or III).
\item[(b)] $ h^\tw_\La = Kk_1 (k_1 + k_2 \pm 1) + k_1(m_2 +1) - \frac{K+1}{4}, $ where + (resp. $ - $) occurs if $ J =  $ I (resp. II); 

$ h^\tw_\La = K(k_1 \pm 1) (k_1 + k_2) +k_1( m_2 + 1) - \frac{K+1}{4},$ where + (resp. $ - $) occurs if $ J =  $ III (resp. IV); 

$ s_\La^\tw = K(k_2+1) +m_2 -1 $ (resp. $ = -K(k_2-1) - m_2 -3 $) if $ J =  $ I or IV (resp. II or III).
\item[(c)] One has the following isomorphisms of the $ N=4 $ superconformal algebra modules:
\[H \left( \La^{[m;m_2]\mathrm{IV}}_{k_1,k_2}\right)\simeq   
H \left( \La^{[m;m_2]\mathrm{I}}_{k_1-1,k_2}\right), 
\quad H \left( \La^{[m;m_2]\mathrm{II}}_{k_1,k_2}\right)\simeq    
H \left( \La^{[m;m_2]\mathrm{III}}_{k_1-1,k_2}\right), 
\]
and the same isomorphisms hold for the Ramond twisted $ N =4  $ modules $ H^\tw (\La). $
 \end{enumerate}
 \end{lemma}
  \begin{proof}
 (a) and (b) are easily computed by formulas \eqref{7.03}, \eqref{7.04}, \eqref{7.07}, \eqref{7.08}. Claim (c) follows from (a) and (b).
 \end{proof}
 
 Thus, we need to consider the $ N=4 $ modules $ H\left( \La^{[m;m_2]J}_{k_1,k_2} \right) $ and $ H^\tw\left( \La^{[m;m_2]J}_{k_1,k_2} \right)  $ only for $ J =  $ I and III. 
 
Formulas for the supercharacters $ \ch^-_{H(\La)} $ and characters $ \ch^+_{H(\La)} $ for the principal admissible weights $ \La = \La^{[m;m_2]J}_{k_1, k_2}, J =  $ I or III, are obtained from that for $ L(\La), $ given by Proposition \ref{prop8.1} and formula \eqref{8.03}, by making use of formula \ref{7.05}. Likewise, formulas for the supercharacters $ \ch^-_{H^\tw(\La)} $ and characters $ \ch^+_{H^\tw(\La)} $ for these weights are obtained, using Proposition \ref{prop8.1} and \eqref{8.03}, by making use of \eqref{7.09}.
 
\begin{proposition}
\label{prop9.2}
Let $ A = \frac{m}{M} (2k_1+k_2 +1)+m_2. $
\begin{enumerate}
\item[(a)] \[ 
\begin{aligned}
\left( \overset{4}{R}\vphantom{R}^- \tch^-_{H(\La^{[m;m_2]\mathrm{I}}_{k_1,k_2})}\right)  \tz & = P^{[M,-m;0]}_{k_1 + \half, -(k_1 + k_2 + \half); \half} (\tau, z, z, 0)  - A \tps^{[M,-m;0;0]}_{k_1 + \half, -(k_1+k_2+\half); \half} (\tau, z, z, 0); \\
\left( \overset{4}{R}\vphantom{R}^+ \tch^+_{H(\La^{[m;m_2]\mathrm{I}}_{k_1,k_2})}\right) \tz & = -
(-1)^{m_2}\left(  P^{[M,-m; \half]}_{k_1 + \half, -(k_1+k_2+\half); \half} \right.
(\tau, z, z, 0) \\
& -\left.  A \tps^{[M,-m;0;\half]}_{k_1+\half, -(k_1+k_2+\half); \half} (\tau, z, z, 0)\right); \\
\left( \overset{4}{R}\vphantom{R}^- \tch^-_{H(\La^{[m;m_2]\mathrm{III}}_{k_1,k_2})}\right)  \tz & = -P^{[M,-m;0]}_{k_1+k_2+\half, -k_1-\half; \half} (\tau, z, z, 0) + A \tps^{[M,-m;0;0]}_{k_1+k_2 + \half, -k_1-\half; \half} (\tau, z, z, 0); \\
\left( \overset{4}{R}\vphantom{R}^- \tch^-_{H(\La^{[m;m_2]\mathrm{III}}_{k_1,k_2})} \right)  \tz & = (-1)^{m_2}\left( P^{[M, -m; \half]}_{k_1+ k_2 + \half, -k_1 -\half; \half} (\tau, z, z, 0) \right. \\
&- \left. A \tps^{[M,-m;0;\half]}_{k_1+k_2 + \half, -k_1-\half; \half} 
(\tau, z, z, 0)\right).
\end{aligned} \]
\item[(b)] \[ 
\begin{aligned}
\left( \overset{4}{R}\vphantom{R}^{-,\tw} \tch^-_{H^\tw(\La^{[m;m_2]\mathrm{I}}_{k_1,k_2})}\right)  \tz & = -P^{[M;-m;0]}_{k_1+k_2+1,-k_1; 0} (\tau, z,z,0) + A \tps^{[M, -m; 0; 0]}_{k_1+k_2+1, -k_1;0} (\tau, z,z,0); \\
\left( \overset{4}{R}\vphantom{R}^{+,\tw} \tch^+_{H^\tw(\La^{[m;m_2]\mathrm{I}}_{k_1,k_2})}\right)  \tz & = (-1)^{m_2} \left( P^{[M,-m; \half]}_{k_1+k_2+1. -k_1;0} (\tau, z,z,0) \right.  \\
&- \left. A \tps^{[M,-m;0;\half]}_{k_1+k_2+1, -k_1;0} (\tau, z,z,0)\right);\\
\left( \overset{4}{R}\vphantom{R}^{-,\tw} \tch^-_{H^\tw(\La^{[m;m_2]\mathrm{III}}_{k_1,k_2})}\right)  \tz & = P^{[M, -m;0]}_{k_1+1, -(k_1+k_2); 0} (\tau, z,z,0) - A \tps^{[M, -m; 0;0]}_{k_1+1, -(k_1+k_2); 0} (\tau, z,z, 0); \\
\left( \overset{4}{R}\vphantom{R}^{+,\tw} \tch^+_{H^\tw(\La^{[m;m_2]\mathrm{III}}_{k_1,k_2})}\right)  \tz & = -(-1)^{m_2} \left(
P^{[M;-m; \half]}_{k_1+1, -(k_1+k_2); 0} (\tau, z,z,0) \right. \\
&- \left. A \tps^{[M,-m;0;\half]}_{k_1+1, -(k_1+k_2); 0} (\tau, z, z, 0)\right).
\end{aligned} \]
\end{enumerate}
\end{proposition}
\begin{remark}
\label{rem9.3}
The proof of Proposition \ref{prop9.2} is facilitated by the following relations (where we assume that the RHSs in Proposition \ref{prop8.1} are extended to arbitrary integer values of $k_1, k_2$):
\[ 
\begin{aligned}
\left(\hat{R}^\pm \tch^\pm_{\La^{[m;m_2]\mathrm{III}}_{k_1,k_2}}\right) 
\tzzt = \hat{R}^\pm \tch^\pm_{\La^{[m;m_2]\mathrm{I}}_{k_1,k_2}} (\tau, -z_2, -z_1, t); \\
\left(\hat{R}^{\pm,\tw } \tch^{\pm,\tw}_{\La^{[m;m_2]\mathrm{III}}_{k_1,k_2}} 
\right)\tzzt = \hat{R}^{\pm,\tw} \tch^{\pm,\tw }_{\La^{[m;m_2]\mathrm{I}}_{k_1+1,k_2-2}} (\tau, -z_2, -z_1, t),
\end{aligned} \]
which imply similar relations for the $ N=4 $ (non)twisted (super)characters, obtained by (twisted) QHR. The proof of these relations is straightforward. 
\end{remark}
\begin{remark}
\label{rem9.4}
Using Lemma \ref{lem8.2} and \eqref{8.07}, some formulas in Proposition \ref{prop9.2} can be rewritten as follows:
\[ 
\begin{aligned}
& \left( \overset{4}{R}\vphantom{R}^+ \tch^+_{H(\La^{[m;m_2]\mathrm{III}}_{k_1,k_2})}\right)  \tz  = - P^{[M,-m;\half]}_{k_1+k_2+\half, -k_1-\half; \half} (\tau, z,z,0) +A \tps^{[M,-m;0;\half]}_{k_1+k_2+\half, -k_1-\half; \half} (\tau, z, z, 0); \\
& \left( \overset{4}{R}\vphantom{R}^- \tch^-_{H(\La^{[m;m_2]\mathrm{III}}_{k_1,k_2})}\right)  \tz  =- P^{[M,-m;0]}_{k_1+k_2+\half, -k_1-\half; \half} (\tau, z,z,0) + A \tps^{[M,-m;0;0]}_{k_1+k_2+\half, -k_1-\half; \half} (\tau, z,z,0); \\
& \left( \overset{4}{R}\vphantom{R}^{+,\tw} \tch^+_{H^\tw(\La^{[m;m_2]\mathrm{I}}_{k_1,k_2})}\right)  \tz  = -(-1)^m P^{[M,-m;\half]}_{k_1+\half+\frac{M}{2}, -(k_1+k_2+\half)+\frac{M}{2}; 0} (\tau, z,z,0)\\
& +(-1)^m A \tps^{[M,-m;0; \half]}_{k_1+\half + \frac{M}{2}, -(k_1+k_2+\half)+\frac{M}{2}; 0} (\tau, z,z,0); \\
& \left( \overset{4}{R}\vphantom{R}^{-,\tw} \tch^-_{H^\tw(\La^{[m;m_2]\mathrm{I}}_{k_1,k_2})}\right)   \tz  =-(-1)^m P^{[M,-m;0]}_{k_1+\half+\frac{M}{2}, -(k_1+k_2+\half)+\frac{M}{2}; 0} (\tau, z,z,0) \\
& +(-1)^m A \tps^{[M. -m;0;0]}_{k_1+\half+ \frac{M}{2}, -(k_1+k_2+\half)+ \frac{M}{2}; 0} (\tau, z, z, 0).
\end{aligned} \]
\end{remark}

Introduce the following functions:
\[ 
\chi^{(1)[M,-m;\epsilon]}_{j,k; \epsilon'} \tz : = \frac{P^{[M,-m;\epsilon]}_{j,k;\epsilon'} (\tau, z, z, 0)}{\overset{4}{R}\vphantom{R}^{(\epsilon)}_{\epsilon'} \tz}, \quad 
\chi^{(0)[M,-m;\epsilon]}_{j,k; \epsilon'} \tz : = \frac{\tps^{[M,-m;\epsilon]}_{j,k;\epsilon'} (\tau, z, z, 0)}{\overset{4}{R}\vphantom{R}^{(\epsilon)}_{\epsilon'} \tz}, 
 \]
where $ \overset{4}{R} \vphantom{R}^{(\epsilon)}_{\epsilon'} \tz $ are given by \eqref{7.06}.
\begin{remark}
\label{rem9.5}

By Lemma \ref{lem8.2} and \eqref{8.07} we have:
\[ \chi^{(1)[M,-m;\epsilon]}_{j,k;\epsilon'}
= - \chi^{(1)[M,-m;\epsilon]}_{k,j;\epsilon'}, \quad  \chi^{(0)[M,-m;\epsilon]}_{j,k;\epsilon'} =  \chi^{(0)[M,-m;\epsilon]}_{k,j;\epsilon'}. \]
It follows that $\chi^{(1)[M,-m;\epsilon]}_{j,j;\epsilon'}=0$.
Note also that, due to Remark \ref{rem8.6}, we have $ \chi^{(1)[1,-m;\epsilon]}_{j,k;\epsilon'} \tz = 0. $ 
\end{remark}
For cases I and III define the pairs $ (j,k) \in (\half + \ZZ)^2 $ by the following table:
\begin{center}
\begin{tabular}{l||c|c}
 & I & III \\
\hline
\hline
$ j $ \rule[-1pt]{0pt}{14pt} & $ k_1 + \half $ & $ k_1 + k_2 + \half $ \\
$ k $  \rule[-1pt]{0pt}{14pt} & $ -k_1 - k_2 - \half $ & $ -k_1 - \half $ \\
\end{tabular}
\end{center}
and for cases $ \mathrm{I}^\tw $ and $ \mathrm{III}^\tw $ define $ (j,k) \in \ZZ^2 $ by the following table:
\begin{center}
\begin{tabular}{l||c|c}
 & $ \mathrm{I}^\tw $ & $ \mathrm{III}^\tw $ \\
\hline
\hline
$ j $ \rule[-1pt]{0pt}{14pt} & $ k_1 + k_2 +1 $ & $ k_1 + 1 $ \\
$ k $  \rule[-1pt]{0pt}{14pt} & $ -k_1 $ & $ -k_1-k_2 $ \\
\end{tabular}
\end{center}
The following theorem is immediate by Proposition \ref{prop9.2} and Remark \ref{rem9.3}.
\begin{theorem}
\label{th9.6}
For $  \La = \La^{[m;m_2]J}_{k_1, k_2}, $ where $ J= $  I or III, the modified (super)characters of the QHR of $ \LLa $ and $ L^\tw (\La) $ are given by the following formulas up to a sign, where $ (j,k) $ are defined by the above tables, and $ \epsilon=0 $ (resp. $ \half $) refers to the supercharacter (resp. character):
\[ \begin{aligned}
\tch\vphantom{h}^{(\epsilon)}_{H(\La)} \tz &= \chi^{(1)[M,-m;\epsilon]}_{j,k; \half} \tz  - \left( \frac{m(j-k)}{M} + m_2\right)  \chi^{(0)[M,-m;\epsilon]}_{j,k; \half} \tz; \\
\tch\vphantom{h}^{(\epsilon)}_{H^\tw(\La)} \tz &= \chi^{(1)[M,-m;\epsilon]}_{j,k; 0} \tz  - \left( \frac{m(j-k)}{M} + m_2\right)  \chi^{(0)[M,-m;\epsilon]}_{j,k; 0} \tz. \\
\end{aligned}  \]
\end{theorem}

Next, we establish modular transformation formulas for functions $ \chi \tz  $.
\begin{theorem}
\label{th9.7}
Let either $ M $ be a positive integer and $ m = -1 $, or  $ M $ be a positive odd integer and $ m \in \ZZ_{\leq -2} $ be coprime to $ M. $ Then
for $ \al = 0 $ or 1 one has:
\[ 
\begin{aligned}
\chi^{(\al) [M,-m;\epsilon]}_{j,k;\epsilon'}  \left( -\frac{1}{\tau}, \frac{z}{\tau} \right) &= -(-1)^{(1-2 \epsilon)(1-2\epsilon')} \frac{\tau^\al}{M} e^{-\frac{2 \pi i }{\tau}(\frac{m}{M}+1)z^2} \hspace{-2em}\sum_{(a,b) \in (\epsilon + \ZZ / M\ZZ)^2} \hspace{-2em} e^{\frac{2 \pi i m}{M}(ak+bj)} \chi^{(\al) [M, -m; \epsilon']}_{a,b; \epsilon} \tz; \\
\chi^{(\al) [M,-m;\epsilon]}_{j,k;\epsilon'} (\tau +1, z) &= e^{-\frac{2 \pi i m}{M} jk-\pi i \epsilon'} \chi^{(\al) [M, -m; |\epsilon-\epsilon'|]}_{j,k;\epsilon} \tz. \\
\end{aligned} \]
\end{theorem}
\begin{proof}
Use \cite{KW14}, Lemma 10.1, for modular trasnformations of the denominators of $ \chi (\tau, z), $ and Proposition \ref{prop8.7} and \eqref{1.17} for modular transformation of its numerators. 
\end{proof}

Denote by $ \mathbb{V}^{(\al)}_C, \ \al = 0 $ or 1, $ C \in \CC, $ the space of complex valued  function in $ \tau, z \in \CC, $ such that $ \Im \tau > 0, $ with the following action of $ SL_2 (\ZZ): $
\begin{equation}
\label{9.01}
f \big{|}_{ \left(\begin{smallmatrix}
a & b \\
c & d \\
\end{smallmatrix} \right)} \tz = (c \tau + d)^{-\al} e^{- \frac{\pi i C c z^2}{3 (c \tau + d)}} f \left( \frac{a \tau  +b}{c \tau + d}, \frac{z}{c \tau +d} \right) .
\end{equation}
This induces an action of $ SL_2 (\ZZ) $ on $ \mathcal{V}^{(0)}_C \oplus \mathcal{V}^{(1)}_C $.

Theorems \ref{th9.6} and \ref{th9.7} imply modular invariance of a family of $ N = 4  $ representations as stated in the following theorem.

\begin{theorem}
\label{th9.8}Let $ M $ and $ m $ be as in Theorem \ref{th9.7}, and let $ C = -6 \left( \frac{m}{M}+1 \right). $ Let $\epsilon, \epsilon' = 0 $ or $ \half $, and let  
\[ \Omega^\epsilon_M = \{ (j,k) \in \left( \epsilon + \ZZ\right)^2 | \ 0 < j \leq k \leq M  \}.  \] 
Let $ V $ be the span of the functions $ \chi^{(\al)[M,-m;\epsilon]}_{j,k;\epsilon'}  \tz$
in $ \mathcal{V}^{(0)}_c \oplus \mathcal{V}^{(1)}_c, $  where $ \al = 0 $
or 1, and $ (j,k) \in \Omega^\epsilon_M$. 
Then
\begin{enumerate}
\item[(a)] The space $ V $ is $ SL_2 (\ZZ) $-invariant with respect to the action (\ref{9.01}). Explicitly, we have
the second formula of Theorem \ref{th9.7}, and the first formula can be rewritten as the following two formulas:
\[ 
\begin{aligned}
\chi^{(1)[M;-m;\epsilon]}_{j,k;\epsilon'} \left( -\frac{1}{\tau}, \frac{z}{\tau}\right) &= (-1)^{(1-2\epsilon)(1-2\epsilon')} \frac{2 i \tau}{M} e^{\frac{\pi i }{3 \tau}Cz^2} \\
&\times \sum_{(a,b) \in \Omega^\epsilon_M} e^{\frac{m}{M} (j+k)(a+b)} \sin \frac{\pi m}{M} (j-k)(a-b) \chi^{(1)[M;-m;\epsilon']}_{a,b;\epsilon} \tz, \\
\chi^{(0)[M;-m;\epsilon]}_{j,k;\epsilon'} \left( -\frac{1}{\tau}, \frac{z}{\tau}\right) &= -(-1)^{(1-2\epsilon)(1-2\epsilon')} \frac{1}{M} e^{\frac{\pi i }{3 \tau}Cz^2} \\
& \times \sum_{(a,b) \in \Omega^\epsilon_M} (2-\delta_{a,b}) 
e^{\frac{m}{M} (j+k)(a+b)} \cos \frac{\pi m}{M} (j-k) (a-b) \chi^{(0)[M;-m;\epsilon']}_{a,b;\epsilon} \tz.
\end{aligned} \]
\item[(b)] The space $ V $ coincides with the span of all modified (super)characters of the representations $ H(\La) $ and $ H^\tw (\La) $ of the $ N=4 $ 
Neveu-Schwarz and Ramond algebras, where the $ \La $'s are as in Theorem \ref{th9.6}.
\end{enumerate}
\end{theorem}
\begin{proof}
Due to Theorems \ref{th9.6} and \ref{th9.7}, formula \eqref{8.07} and Remark \ref{rem8.8}, it suffices to show that, using transformations
\begin{equation}
\label{9.02} 
(j,k) \mapsto (j+aM, k+bM), \mbox{ where } a, b \in \ZZ, \mbox{ and } (j,k) \mapsto (k,j),
\end{equation}
one can obtain the set $ \Omega^\epsilon_M $ from the pairs $ (j,k) $ from 
the above first (resp. second) table for $\epsilon=\half$ (resp. $=0$). 

First, consider the principal admissible $ \wg $-modules $ \LLa $ with $ \La = \La^{[m;m_2]J}_{k_1,k_2} $, where $ J = $ I or III. In the case I we have 
\[ k_1, k_2 \geq 0, \quad 2k_1 + k_2 \leq M-1, \]
hence we can see from the above table that 
\begin{equation}
\label{9.03}
j,k \in \half + \ZZ, \quad j \geq \half, \quad j + k \leq 0, \quad k \geq j - M, 
\end{equation}
Similarly, in case III we have
\begin{equation}
\label{9.04}
j, k \in \half + \ZZ, \quad k \leq - \half, \quad j + k \geq 1, \quad k \geq j - M.
\end{equation}
The union of domains \eqref{9.03} and \eqref{9.04} can be transformed by $ (j,k) \mapsto (j, k+M) $ and $ (j,k) \mapsto (k,j) $ to the domain
\[j,k \in \half + \ZZ, \quad k \geq \half, \quad j \leq M-\half, \quad j-k \leq 0,\]
which is the domain $\Omega^\half_M$.

Likewise in the case $ \mathrm{I}^\tw $ we have:
\begin{equation}
\label{9.06}
j,k \in \ZZ, \quad k \leq 0, \quad j + k \geq 1, \quad k \geq j-M,
\end{equation}
and in the case $ \mathrm{III}^\tw $ we have:
\begin{equation}
\label{9.07}
j,k \in \ZZ, \quad j \geq 1, \quad j + k \leq 0, \quad k \geq j - M.
\end{equation}
Applying the same transformations as in the non-twisted case to the union of domains \eqref{9.06} and \eqref{9.07}, we obtain the domain
\[j,k \in \ZZ, \, k \geq 1, \quad j \leq M, \quad j-k \leq 0, 
\]
which is the domain $\Omega^0_M$.
\end{proof}

\section{Integrable $ \wg $-modules, where $ \fg = D(2,1;a) $}

Recall that the Lie superalgebra $ D(2,1;a) $ is a family of simple 17-dimensional Lie superalgebras with the even part isomorphic to $ \sl (2) \oplus \sl (2) \oplus \sl (2), $ depending on a parameter $ a \neq 0, -1. $ Two members of this family are isomorphic iff the corresponding values of parameters lie on the orbit of the group, generated by transformations $ a \mapsto -1 -a, \ a \mapsto \frac{1}{a} $ \cite{K77}. Hence, provided that $ a \in \QQ, $ we may assume that $ -1 < a \leq -\half, $ and write it uniquely in the form:
\begin{equation}
\label{10.01}
 a = -\frac{p}{p + q}, \mbox{ where } p, q \in \zp \mbox{ are coprime and}\,\, p\geq q.
\end{equation}
Recall that the condition $ a \in \QQ $ is necessary for existence of integrable $ \wg $-modules, see \cite{KW16}, Corollary 6.4. 

We consider the set of simple roots $ \Pi = \{ \al_1, \al_2, \al_3 \} $ of $ \fg, $ where $ \al_1 $ is odd and $ \al_2, \al_3 $ are even, with the following non-zero scalar products:
\[ (\al_2| \al_2) = 2a, \quad (\al_3|\al_3) = -2(a+1), \quad (\al_1|\al_2) = -a, \quad (\al_1|\al_3) = a+ 1. \]
The highest root is $ \theta = 2 \al_1 + \al_2 + \al_3, $ and $ (\theta | \theta) = 2; $ we also have $ \rho = -\al_1, h^\vee = 0. $

A $ \wg $-module $ \LLa $ is called {\it integrable} if the roots $ \al_2, \delta - \al_2, \al_3, \delta - \al_3 $ are integrable for $ L(\La) $ (according to the terminology of Section 2, this should be called complementary integrable). Let $ \La_i, \,i = 0,1,2,3 $, be the fundamental weights, i.e. $ (\La_i | \al_j^\vee) = \delta_{ij}, $ where $ \al^\vee_1 = \al_1 $ and $ \al^\vee_j = \frac{2 \al_j}{(\al_j | \al_j)} $ for $ j \neq 1 $ (these conditions define the $ \La_j $ up to adding a multiple of $ \delta $). Then any $ \La \in \wh^* $ can be written, up to a multiple of $ \delta,  $ in the form $ \La = \sum_{i = 0}^{3} m_i \La_i$. Since $ \delta = \al_0 + 2 \al_1 + \al_2 + \al_3 = \al_0^\vee + 2 \al_1^\vee + a \al_2^\vee -(a+1)\al_3^\vee $, the level $ K $ of this $ \La $ is given by
\begin{equation}
\label{10.02}
K = m_0 + 2m_1 + am_2 -(a+1)m_3.
\end{equation}
Recall that an integrable $ \wg $-module $ \LLa $ of non-critical (= non-zero) level $ K $ exists iff (\cite{KW16}, Corollary 6.4)
\begin{equation}
\label{10.03}
 K = -\frac{pqn}{p+q} \mbox{ for some positive integer }n.
\end{equation}

Consider the following four odd (isotropic) roots:
\[ \beta_0 = \al_0 + \al_1, \quad \beta_1 = \al_1, \quad \beta_2 = \al_1 + \al_2, \quad \beta_3 = \al_1 + \al_3,  \]
and the corresponding four sets of weights ($j=0,1,2,3$):
\[ P^{'K}_{+,j} = \{ \La \in \wh^* | \LLa \mbox{ is integrable of level \eqref{10.03}}, \mbox{and}\,(\La + \rho | \beta _j) = 0 \}. \]
The following proposition describes the sets $ P^{'K}_{+,j}, $ up to adding multiples of $ \delta.  $
\begin{proposition}
\label{prop10.1}
Let 
\[ 
\begin{aligned}
\La^{[K;m_2, m_3](0)} & = \sum_{i=0}^{3} m_i \La_i, \mbox{ where }  (p+q)m_1 = p(m_2+1) + q(m_3+1) - npq,\, m_0 + m_1 =-1 ;  \\
\La^{[K;m_2, m_3](1)} & = \sum_{i=0}^{3} m_i \La_i, \mbox{ where } (p+q)m_0 = pm_2 + qm_3 - npq,\, m_1 = 0  ;  \\
\La^{[K;m_2, m_3](2)} & = \sum_{i=0}^{3} m_i \La_i, \mbox{ where } (p+q) (m_0+1) = -p (m_2+1) + q(m_3+1)-npq, \\
&(p+q)m_1 = p(m_2+1)  ;  \\
\La^{[K;m_2, m_3](3)} & = \sum_{i=0}^{3} m_i \La_i, \mbox{ where }  (p+q)(m_0+1) = p(m_2+1)-q(m_3+1)-npq,\\
& (p+q)m_1 = q(m_3+1).   \\
\end{aligned} \]
Then, up to adding multiples of $\delta$, we have:
\[ 
\begin{aligned}
P^{'K}_{+,0} & = \lbrace \La^{[K;m_2, m_3](0)} |\, m_2, m_3 \in \ZZ_{\geq 0}, 0 \leq m_2 < nq, 0 \leq m_3 < np \rbrace; \\ 
P^{'K}_{+,1} & = \lbrace \La^{[K;m_2, m_3](1)} |\, m_2, m_3 \in \ZZ_{\geq 0}, 0 \leq m_2 < nq, 0 \leq m_3 < np \rbrace \cup \{nq\La_2\} \cup \{ np\La_3\}; \\
P^{'K}_{+,2} & = \lbrace \La^{[K;m_2, m_3](2)} |\, m_2, m_3 \in \ZZ_{\geq 0}, 0 \leq m_2 \leq nq-2, 0 \leq m_3 \leq np-1 \rbrace; \\
P^{'K}_{+,3} & = \lbrace \La^{[K;m_2, m_3](3)} |\, m_2, m_3 \in \ZZ_{\geq 0}, 0 \leq m_2 \leq nq-1, 0 \leq m_3 \leq np-2 \rbrace. \\
\end{aligned} \]
\end{proposition}
\begin{proof}
It is standard, using odd reflections, see \cite{KW16}, proof of Theorem 5.4.
\end{proof}
\begin{remark}
\label{rem10.2} 
If $ m_3 = np-1 $  (resp. $ m_2 = nq-1 $), then $ \La^{[K;m_2, m_3](2)} $ (resp. $ \La^{[K;m_2, m_3](3)} $) $ \in  P^{'K}_{+,0}.   $
\end{remark}
In view of this remark, we define $ P^K_{+,j},\, j = 0,1,2,3, $ as follows:
\[ 
\begin{aligned}
P^K_{+, j} & = P^{'K}_{+,j}\, \mbox{ if }\, j = 0,1,\\
P^K_{+,j} & = \lbrace \La^{[K;m_2,m_3](j)} |\, m_2, m_3 \in \ZZ_{\geq 0}, 0 
\leq m_2 \leq nq-2, 0 \leq m_3 \leq np-2 \rbrace \, \mbox{ if }\,  j = 2,3. \\
\end{aligned} \]

Next we consider the integrability condition with respect to $ \theta; $ its proof is standard.
\begin{lemma}
\label{lem10.3}
The root $ \theta $ is integrable for $ \LLa, $ where 
$\La = \sum_{i=0}^{3} m_i \La_i$,
iff
\begin{enumerate}
\item[(a)] $ m_1 =0 $ and 
\begin{enumerate}
\item[(i)] $ am_2 - (a+1)m_3 \in \ZZ_{\geq 0} $,
\item[(ii)] $ am_2 - (a+1) m_3 = 0  $ implies that $ m_2 = m_3 =0. $
\end{enumerate}
\item[(b)] $ m_1 \neq 0 $ and \begin{enumerate}
\item[(i)] $ 2m_1 + a(m_2+1) -(a+1) (m_3+1) \in \ZZ_{\geq 0}, $ 
\item[(ii)] $ 2m_1 + a(m_2+1) -(a+1) (m_3+1) = 0  $ implies that $ m_1+a(m_2+1) =0 $  and $ m_1 -(a+1)(m_3+1) = 0. $
\end{enumerate}
\end{enumerate}
\end{lemma}
\begin{remark}
\label{rem10.4}
If $ a $ is given by \eqref{10.01}, then 
\[ 
\begin{aligned}
am_2 - (a+1)m_3 & = - \frac{pm_2 + qm_3}{p+q}, \\
2m_1 + a(m_2+1) - (a+1) (m_3+1) & = 2m_1 - \frac{p(m_2+1)+q(m_3+1)}{p+q}, \\
m_1 + a(m_2+1) & = m_1 - \frac{p(m_2+1)}{p+q}, \\
m_1 - (a+1) (m_3+1) & = m_1 - \frac{q(m_3+1)}{p+q}.
\end{aligned} \]
\end{remark}
By Lemma \ref{lem10.3} and Remark \ref{rem10.4}, we have the following:
\begin{lemma}
\label{lem10.5}
The root $ \theta $ is integrable for $ L(\La^{[K;m_2, m_3](j)}) $ iff:
\begin{enumerate}
\item[(0)] For $ j = 0, \,m_2 = nq-1 $ and $ m_3 = np-1. $
\item[(1)] For $ j=1, \,m_2 = m_3 = 0.  $
\item[(2)] For $ j = 2, \,p(m_2+1)-q(m_3+1) \in (p+q) \ZZ_{\geq 0}. $
\item[(3)]  For $ j = 3, \,q(m_3+1)- p(m_2+1) \in (p+q) \ZZ_{\geq 0}. $
\end{enumerate}
\end{lemma}
\begin{proof}
Since $ m_1 = 0 $ for $ \La^{[K;m_2, m_3](1)}, $ we have by Lemma \ref{lem10.3} and Remark \ref{rem10.4} that $L( \La^{[K;m_2, m_3](1)}) $ is integrable with respect to $ \theta  $ iff 
\[ am_2 -(a+1)m_3 = - \frac{pm_2 + qm_3}{p+q} \in \ZZ_{\geq 0}, \]
so $ m_2 = m_3 = 0, $ proving (1).

Since $ m_1 \neq 0 $ for  $ \La^{[K;m_2, m_3](0)},  $ we have by Lemma \ref{lem10.3} and Remark \ref{rem10.4} that  
$L( \La^{[K;m_2, m_3](0)}) $ is integrable with respect to $ \theta $ iff
\[ 2m_1 + a(m_2+1) - (a+1)(m_3+1) = 2m_1 - \frac{p(m_2+1) + q (m_3 +1)}{p+q} \in \ZZ_{\geq 0}. \]
But 
\[ 2m_1 - \frac{p(m_2+1)+q(m_3+1)}{p+q} = \frac{1}{p+q} \left( p(m_2+1 -nq) + q(m_3+1-np) \right) \leq 0. \]
Hence $ m_2+1-nq =0= m_3+1-np, $ proving (0). 

Since $ m_1 \neq 0 $ for $ \La^{[K;m_2, m_3](2)}, $ we have by Lemma \ref{lem10.3} and Remark \ref{rem10.4} that integrability of 
$L( \La^{[m;m_2, m_3](2)}) $ with respect to $ \theta $ implies that 
\[ p(m_2+1)-q(m_3+1) \in (p+q) \ZZ_{\geq 0}. \]
Since also
\[ m_1 + a(m_2+1) = \frac{1}{p+q} \left( (p+q)m_1 -p(m_2+1)\right) =0, \]
we obtain (2).
The proof of (3) is the same as of (2).
\end{proof}


We will also need the following integrability conditions with respect to $ \al_0. $

\begin{lemma}
\label{lem10.6}
Let $ \La = \La^{[K;m_2, m_3](j)} \in P^K_{+, j} $. Then

\begin{enumerate}
\item[(a)] If $ m_2 $ (resp. $ m_3 $) $ = 0 $  and $ j = 1, $ then $ \al_0 $ is integrable for $ \LLa  $ iff $ m_3 = np  $ (resp. $ m_2 = nq $).
\item[(b)] If $ m_2 = nq-1 $ (resp. $ m_3 = np-1 $) and $ j=0, $ then $ \al_0 $ is not integrable for $ \LLa. $
\item[(c)] If $ j = 2 $ or 3, then $ \al_0 $ is not integrable for $ \LLa. $
\end{enumerate}
\end{lemma}

\begin{proof}
It is straightforward, using that $ \al_0 $ is integrable for $ \LLa $ iff $ m_0 \in \ZZ_{\geq 0} $ and the description of the sets $ P^K_{+, j}, $ see Proposition \ref{prop10.1} and Remark \ref{rem10.2}.
\end{proof}

Let $ \La = \La^{[K; m_2, m_3](j)} $ be the highest weight of one of the integrable $ \wg $-modules $ \LLa. $ Then in the supercharacter formula \eqref{2.4}-\eqref{2.7} we have: $ W^{\#} $ contains the subgroup $ W^{\#}_0 = 
\{1,  r_{\al_2}, r_{\al_3}, r_{\al_2} r_{\al_3} \}, \, L^{\#} = \frac{1}{a} \ZZ \al_2 + 
\frac{1}{a+1} \ZZ \al_3,\, T = \{ \beta_j \}. $ In order to apply this formula, we need to introduce the following functions:
\[ 
\begin{aligned}
F^{[K; m_2, m_3](j)} & = q^{\frac{|\La + \rho|^2}{2K}} \sum_{\al \in L^{\#}} t_\al \frac{e^{\La + \rho}}{1 - e^{-\beta_j}}, \\
B^{[K; m_2, m_3](j)} & = \sum_{w \in W^{\#}_0} \epsilon(w) w 
(F^{[K; m_2, m_3](j)}).
\end{aligned} \]
\begin{remark}
\label{rem10.7}
We have for $ j = 0,2,3: $
\[ r_{\al_j} (\La^{[K;m_2, m_3](1)} + \rho) = \La^{[K; {m_2}', {m_3}'](j)}+ \rho ,\]
where $ ({m_2}', {m_3}') = (m_2, m_3) $ if $ j = 0, \ =(-m_2-2, m_3)$  if $ j = 2, \ = (m_2, -m_3-2) $ if $ j = 3. $
\end{remark}
By Remark \ref{rem10.7}, we need to compute $ F^{[K; m_2, m_3](j)} $ and $B^{[K; m_2, m_3](j)}  $ only for $ j = 1. $
\begin{lemma}
\label{lem10.8} We have
\[\begin{aligned}
F^{[K; m_2, m_3](1)} &= e^{K\La_0 + \frac{K-1-m_0}{2}\theta} q^{\frac{a(a+1)}{4K}(m_2 -m_3)^2} \\
& \times \sum_{j,k\in \ZZ} \frac{e^{\left( \frac{Kj}{a} + \frac{m_2+1}{2} \right) \al_2 + \left( \frac{-Kk}{a+1} + \frac{m_3+1}{2}  \right) \al_3  }  q^{K \left( \frac{j^2}{a} - \frac{k^2}{a+1}\right) +j(m_2+1) + k(m_3+1) }}{1 - e^{-\al_1} q^{j+k}}.
\end{aligned} \]
 \end{lemma}
\begin{proof}
Let $ \La = \La^{[K; m_2, m_3](1)}. $ Then
\begin{equation}
\label{10.04}
\frac{|\La + \rho |^2}{2K} = \frac{(m_2-m_3)^2}{4n(p+q)}.
\end{equation}
Since $ \La_1 = 2 \La_0 + \theta, \ \La_2 = a\La_0 + \frac{\al_2}{2} + \frac{a}{2} \theta, \ \La_3 = -(a+1) \La_0 + \frac{\al_3}{2} - \frac{a+1}{2} \theta, $ we have, using \eqref{10.01} and \eqref{10.02}:
\begin{equation}
\label{10.05}
\La + \rho = m_0 \La_0 + m_2\La_2 + m_3 \La_3 - \al_1 = K\La_0 + \frac{m_2+1}{2}\al_2 + \frac{m_3+1}{2} \al_3 + \frac{K-1-m_0}{2} \theta. 
\end{equation}
The action of $ t_{j,k} :=t_{j \al_2^\vee + k \al_3^\vee} $ is as follows: $ t_{j,k} (\La_0)  = \La_0 + \frac{j}{a} \al_2 - \frac{k}{a+1} \al_3 + \left( \frac{k^2}{a+1} - \frac{j^2}{a} \right) \delta, \
t_{j,k} (\al_1)  = \al_1 + (j+k) \delta, \ t_{j,k} (\al_2)  = \al_2 - 2 j \delta, \ t_{j,k} (\al_3)  = \al_3 - 2 k \delta, t_{j,k} (\theta) = \theta. $
Hence, using \eqref{10.05}, we obtain:
\[ \begin{aligned}
t_{j,k} (\La+\rho) &= K\La_0 + \frac{K-1-m_0}{2} \theta + \left( \frac{Kj}{a} + \frac{m_2+1}{2} \right) \al_2 + \left( \frac{-Kk}{a+1} + \frac{m_3+1}{2} \right) \al_3 \\
& - \left( K \left( \frac{j^2}{a} - \frac{k^2}{a+1}\right) + j(m_2+1) + k(m_3+1) \right) \delta.
\end{aligned}  \]
The lemma follows from this and \eqref{10.04}.
\end{proof}

We introduce the following coordinates in $ \wh^*: $
\begin{equation}
\label{10.06}
h = 2 \pi i( - \tau \La_0 +z + t \delta) = (\tau, z_1, z_2, z_3, t),
\end{equation}
where 
\begin{equation}
\label{10.07}
z = -\frac{z_2+z_3}{2} \theta + \frac{z_1 -z_2}{2a} \al_2 - \frac{z_1 - z_3}{2(a+1)} \al_3,
\end{equation}
In these coordinates we have:
\begin{equation}
\label{10.08}
e^{-\al_1} = e^{2 \pi i z_1}, e^{-\al_2} = e^{2 \pi i (-z_1 + z_2)}, e^{-\al_3} = e^{2 \pi i (-z_1+z_3)}, e^{-\theta} = e^{2 \pi i (z_2 + z_3)}. 
\end{equation}
Now we rewrite the formula, given by Lemma \ref{lem10.8}, in these coordinates in two different ways. 
\begin{lemma}
\label{lem10.9}
\begin{enumerate}
\item[(a)] \[ 
\begin{aligned}
F^{[K;m_2, m_3](1)} &= e^{2 \pi i K t} e^{-\pi i (K-1-m_0)(z_2+z_3)} e^{-2 \pi i (m_3+1)z_1} \\
& \times \sum_{j \in \ZZ} e^{2 \pi i \left( \frac{Kj}{a} + \frac{m_2+1}{2}\right)  (z_1-z_2) + 2 \pi i \left( \frac{Kj}{a+1} + \frac{m_3+1}{2}\right) (z_1-z_3)} \\
& \times q^{\frac{K}{a(a+1)} \left( j + \frac{a(a+1)}{2K} (m_2-m_3)\right)^2} \Phi_1^{[\frac{-K}{a+1}; m_3+1]} (\tau, z_1, -z_3-2j\tau, 0).\\
\end{aligned} \]
\item[(b)]\[ 
\begin{aligned}
F^{[K;m_2, m_3](1)} &= e^{2 \pi i K t} e^{-\pi i (K-1-m_0)(z_2+z_3)} e^{-2 \pi i (m_2+1)z_1} \\
& \times \sum_{k \in \ZZ} e^{2 \pi i \left(- \frac{Kk}{a} + \frac{m_2+1}{2}\right)  (z_1-z_2) + 2 \pi i \left(- \frac{Kk}{a+1} + \frac{m_3+1}{2}\right) (z_1-z_3)} \\
& \times q^{\frac{K}{a(a+1)} \left( k + \frac{a(a+1)}{2K} (m_3-m_2)\right)^2} \Phi_1^{[\frac{K}{a}; m_2+1]} (\tau, z_1, -z_2-2k\tau, 0).\\
\end{aligned} \]
\end{enumerate}
\end{lemma}
\begin{proof}
First, we rewrite the formula for $ F^{[K;m_2, m_3](1)}, $ given by 
Lemma \ref{lem10.8}, by letting $ k = r-j: $
\[ 
\begin{aligned}
F^{K;m_2,m_3](1)} &= e^{K\La_0 + \frac{K-1-m_0}{2} \theta} q^{\frac{a(a+1)}{4K} (m_2-m_3)^2} \\
&\times \sum_{j \in \ZZ} e^{\left( \frac{Kj}{a} + \frac{m_2+1}{2}\right)  \al_2 + \left( \frac{Kj}{a+1} + \frac{m_3+1}{2}\right) \al_3} q^{\frac{K}{a(a+1)}j^2 + (m_2+m_3)j} \\
& \times \sum_{r \in \ZZ} \frac{e^{-\frac{K}{a+1} r \al_3} q^{\frac{2K}{a+1}jr} q^{-\frac{K}{a+1}r^2 + (m_3+1)r}}{1 - e^{-\al_1} q^r}.
\end{aligned} \]
Using \eqref{10.08}, we obtain:
\[ 
\begin{aligned}
F^{[K;m_2, m_3](1)} &= e^{2 \pi i K z} e^{-\pi i (K-1-m_0) (z_2+z_3)} e^{-2 \pi i (m_0+1)z_1}\\
& \times \sum_{j \in\ZZ} e^{2 \pi i \left( \frac{Kj}{a} + \frac{m_2+1}{2}\right) (z_1-z_2) + 2 \pi i \left( \frac{Kj}{a+1} + \frac{m_3+1}{2}\right) (z_1-z_3)} q^{\frac{K}{a(a+1)} \left( j + \frac{a(a+1)}{2K} (m_2-m_3)\right)^2 } \\
& \times \sum_{r \in\ZZ} \frac{e^{2 \pi i \frac{-K}{a+1}\left( z_1 -(z_3+2j \tau)\right) } e^{2 \pi i (m_3+1)z_1} q^{-\frac{K}{a+1}r^2 + (m_3+1)r} }{1 - e^{2 \pi i z_1} q^r},
\end{aligned} \]
proving (a). The proof of (b) is similar, using substitution $ j = r-k. $
\end{proof}

We define the two modifications $ \tef^{[K;m_2, m_3](1)}_P $ and $ \tef^{[K;m_2, m_3](1)}_Q $ of the function $ F^{[K;m_2, m_3](1)} $, given by (a) and (b) of Lemma \ref{lem10.9} respectively, by replacing the function $ \Phi_1 $ by its modification $ \tph_1, $ see \eqref{1.8} and Remark \ref{rem1.2}. We thus get two different modifications, given by the following lemma. 
\begin{lemma}
\label{lem10.10}
\begin{enumerate}
\item[(a)] 
\[\begin{aligned}
&\tef^{[K;m_2, m_3](1)}_P (\tau, z_1, z_2, z_3, t) \\
&= e^{2 \pi i K t} \Theta_{m_2-m_3, n(p+q)} (\tau, z_1-(a+1)z_2 - az_3) \tph_1^{[np;m_3+1]} (\tau, z_1, -z_3, 0).
\end{aligned} 
 \]
\item[(b)]\[\begin{aligned}
&\tef^{[K;m_2, m_3](1)}_Q (\tau, z_1, z_2, z_3, t) \\
&= e^{2 \pi i K t} \Theta_{m_3-m_2, n(p+q)} (\tau, z_1+(a+1)z_2 + az_3) \tph_1^{[nq;m_2+1]} (\tau, z_1, -z_2, 0).
\end{aligned} 
 \]
\end{enumerate}
\end{lemma}
\begin{proof}
We have by Lemma \ref{lem10.9} and \eqref{10.03}:
\[\begin{aligned}
\tef^{[K;m_2, m_3](1)}_P (\tau, z_1, z_2, z_3, t) &= e^{2 \pi i K t} e^{-\pi i (K-1-m_0) (z_2 + z_3)} e^{-2 \pi i (m_3+1)z_1} \\
& \times \sum_{j \in \ZZ} e^{2 \pi i \left( \frac{Kj}{a} + \frac{m_2+1}{2}\right) (z_1-z_2) + 2 \pi i \left( \frac{Kj}{a+1} + \frac{m_3+1}{2}\right) (z_1-z_3)} \\
&\times g^{\frac{K}{a(a+1)}\left( j + \frac{a(a+1)}{2K} (m_2-m_3)\right)^2} e^{2 \pi i np 2jz_1} \tph_1^{[np; m_3+1]} (\tau, z_1, -z_3, 0).
\end{aligned} 
 \]
 We have used here the elliptic transformation property of the function $ \tph_1, $ given by \eqref{1.14}, see Remark \ref{rem1.2}. But from \eqref{1.1} it is clear that the factor in front of $ \tph^{[np;m_3+1]}_1 $ is $ \Theta_{m_2-m_3, n(p+q)} (\tau, z_1 -(a+1)z_2 -az_3),$ proving (a). The proof of (b) is the same.
\end{proof}
\begin{remark}
\label{rem10.11}
In coordinates \eqref{10.06}, \eqref{10.07} the formulas for the action of the reflections of the Weyl group are as follows (the coordinate $ \tau $ is fixed):
\[ 
\begin{aligned}
r_{\al_2} (z_1, z_2, z_3, t) & = (z_2, z_1, -z_1+z_2+z_3, t), \\
r_{\al_3} (z_1, z_2, z_3, t) & = (z_3, -z_1+z_2+z_3, z_1, t),\\
r_{\al_2}r_{\al_3}  (z_1, z_2, z_3, t) & = (-z_1+z_2+z_3, z_3, z_2, t),\\
r_{\theta} (z_1, z_2, z_3, t) & = (z_1-z_2-z_3, -z_3, -z_2, t), \\
r_{\al_0} (z_1, z_2, z_3, t) & = (z_1-z_2-z_3+\tau, -z_3+\tau, -z_2+\tau, t-z_2-z_3+\tau). \\
\end{aligned} \]
\end{remark}

Next, we define the two modification of the function $ B^{[K;m_2, m_3](1)} $ which we denote by $ \tbe_P^{[K;m_2, m_3](1)}  $ and $ \tbe_Q^{[K;m_2, m_3](1)} $, by replacing the function $ F^{[K;m_2, m_3](1)} $ in its expression by its respective modification $ \tef_P $ and $ \tef_Q. $ Using Remark \ref{rem10.11} and \eqref{1.9}, we obtain the following explicit expressions:
\begin{equation}
\label{10.09}
\begin{aligned}
&\tbe_P^{[K;m_2, m_3](1)} (\tau, z_1,z_2,z_3,t) \\
& = e^{2 \pi i Kt} \left( \Theta_{m_2-m_3, n(p+q)} (\tau, z_1 - (a+1) z_2 - az_3) \tph^{[np;m_3+1]} (\tau, z_1, -z_3, 0) \right.\\
&\left. + \Theta_{m_3-m_2, n(p+q)} (\tau, z_1+(a-1)z_2+az_3) \tph^{[np; m_3+1]} (\tau, -z_1+z_2+z_3, -z_2, 0) \right), 
\end{aligned}
\end{equation}
\begin{equation}
\label{10.10}
\begin{aligned}
&\tbe_Q^{[K;m_2, m_3](1)} (\tau, z_1,z_2,z_3,t)  = \\
& = e^{2 \pi i Kt} \left( \Theta_{m_3-m_2, n(p+q)} (\tau, z_1 + (a+1) z_2 + az_3) \tph^{[nq;m_2+1]} (\tau, z_1, -z_2, 0) \right.\\
&\left. + \Theta_{m_2-m_3, n(p+q)} (\tau, z_1-(a+1)z_2-(a+2)z_3) 
\tph^{[nq; m_2+1]} (\tau, -z_1+z_2+z_3, -z_3, 0) \right).
\end{aligned}
\end{equation}
These formulas motivate us to introduce the following functions $ P_j $ and $ Q_j \ (j \in \ZZ): $
\[ 
\begin{aligned}
P_j (\tau, z_1, z_2, z_3, t) & = e^{2 \pi i Kt} \left( \Theta_{j, n(p+q)} (\tau, z_1-(a+1)z_2-az_3) \tph^{[np;0]} (\tau, z_1, -z_3, 0) \right. \\
& \left. + \Theta_{-j, n(p+q)} (\tau, z_1 + (a-1) z_2 + az_3) \tph^{[np;0]} (\tau, -z_1+z_2+z_3, -z_2, 0) \right), \\
Q_j (\tau, z_1, z_2, z_3, t) & = e^{2 \pi i Kt} \left( \Theta_{j, n(p+q)} (\tau, z_1+(a+1)z_2+az_3) \tph^{[nq;0]} (\tau, z_1, -z_2, 0) \right. \\
& \left. + \Theta_{-j, n(p+q)} (\tau, z_1 - (a+1) z_2 - (a+2)z_3) \tph^{[nq;0]} (\tau, -z_1+z_2+z_3, -z_3, 0) \right). \\
\end{aligned} \]
\begin{lemma}
\label{lem10.12}
We have for $j\in \ZZ$:
\begin{enumerate}
 \item[(a)]\[
 r_{\al_0} (P_j) 
= -P_{-j-2np} 
, \,
 r_{\al_i} (P_j )
= -P_j 
\mbox{ for }  i = 2,3, \,
  r_\theta (P_j) 
= -P_{-j}    .
\]
\item[(b)]  \[ 
r_{\al_0} (Q_j)    
= -Q_{-j-2nq} 
, \, 
r_{\al_i} (Q_j)     
= -Q_j 
 \mbox{ for }  i = 2,3, \, 
r_\theta (Q_j)  
= -Q_{-j}.    
\]
\end{enumerate}
\end{lemma}
\begin{proof}
It is straightforward, using Remark \ref{rem10.11}.
\end{proof}

For $ j = 0,2,3 $ we let
\[ \tef^{[K; m_2, m_3](j)}_{P(\mathrm{resp.} \ Q)} = r_{\al_j} \left( \tef^{[K; m'_2, m'_3](1)}_{P(\mathrm{resp.} \ Q)}\right) , \]
where $ (m'_2, m'_3) $ are expressed via $ m_2, m_3 $ as in Remark \ref{rem10.7}, and we let
\[ \begin{aligned}
\tbe^{[K; m_2, m_3](j)}_{P(\mathrm{resp.} \ Q)} &= \sum_{w \in W^{\#}_0} \epsilon(w) w \left( \tef^{[K; m_2, m_3](j)}_{P(\mathrm{resp.} \ Q)} \right) \\
& \left(  =  r_{\al_j} \left( \tbe^{[K; m_2, m_3](1)}_{P(\mathrm{resp.} \ Q)} \right) \right).
\end{aligned}
\]
Formulas \eqref{10.09}, \eqref{10.10} and Lemma \ref{lem10.12} give the following expressions for  the functions $ \tbe^{[K; m_2, m_3](j)}_{P(\mbox{resp. }Q)}  $ for $ j = 0,1,2,3. $
\begin{lemma}
\label{lem10.13}
\begin{enumerate}
\item[(a)]
\[ 
\begin{aligned}
\tbe^{[K;m_2,m_3](1)}_P & = P_{m_2-m_3} , \quad  \tbe^{[K;m_2,m_3](0)}_P = -P_{m_3-m_2-2np} , \, \\
\tbe^{[K;m_2,m_3](2)}_P  & = -P_{-m_2-m_3-2} , \quad \tbe^{[K;m_2,m_3](3)}_P  = -P_{m_2+m_3+2} . 
\end{aligned} 
\] 
\item[(b)]\[ 
\begin{aligned}
\tbe^{[K;m_2,m_3](1)}_Q & = Q_{m_3-m_2}  , \quad \tbe^{[K;m_2,m_3](0)}_Q   = -Q_{m_2-m_3-2nq} , \, \\
\tbe^{[K;m_2,m_3](2)}_Q & = -Q_{m_2+m_3+2} , \quad \tbe^{[K;m_2,m_3](3)}_Q   = -Q_{-m_2-m_3-2}. 
\end{aligned} 
\]
\end{enumerate}
\end{lemma}


Next, for $ \La = \La^{[K; m_2, m_3](j)} \in P^K_{+,j} $ we consider the two modifications of supercharacters pf integrable $ \wg $-modules $ \LLa $ given by \eqref{2.4}, by replacing in the RHS the functions $ F^{[K; m_2, m_3](j)} $ by their modifications $ \tef^{[K; m_2, m_3](j)}_{P(\mathrm{resp.} \ Q)} $. We denote the respective modified supercharacters by $ \tch^{-[P]}_\La $ and $ \tch^{-[Q]}_\La. $ By \eqref{2.5} the normalized superdenominator $\hat{R}^-$ is given by \eqref{2.5} (see \eqref{12.6} in Section 12 for an explicit formula).
\begin{lemma}
\label{lem10.14}
Let $ \La = \La^{[K;m_2,m_3](1)} \in P^K_{+,1}. $ Then
\begin{enumerate}
\item[(a)] $ \hat{R}^- \tch^{-[P]}_\La = P_{m_2 - m_3} $ if $ \theta $ and $ \al_0 $ are not integrable for $ \LLa. $
\item[(b)] $ \hat{R}^- \tch^{-[P]}_\La = \frac{2}{j_\La} P_0 $ if $ \theta $ is integrable for $ \LLa. $ 
\item[(c)] $ \hat{R}^- \tch^{-[P]}_\La = \frac{1}{j_\La} \left( P_{m_2-m_3} + P_{m_3-m_2-2np}\right) $ if $ \al_0 $ is integrable for $ \LLa.  $ In particular, $ \hat{R}\vphantom{R}^- \tch^{-[P]}_\La = \frac{2}{j_\La} P_{nq}$ if $ m_2 = nq, m_3 = 0 $ (resp. $ = \frac{2}{j_\La}P_{-np} $ if $ m_2 = 0, m_3 = np $).
\end{enumerate}
\end{lemma}
\begin{proof}
  By the character formula \eqref{2.4}-\eqref{2.7} and Proposition \ref{prop2.1}
  we have $ \hat{R}\vphantom{R}^- \tch^{[P]}_\La  = \tbe^{[K;m_2,m_3](1)}_P$ in case (a). Then claim (a) follows from Lemma \ref{lem10.13}. 

  Similarly, in case (b) we have: $ \hat{R}^- \tch^{[P]}_\La = \frac{1}{j_\La} \left( \tbe^{[K; m_2, m_3](1)} - r_\theta \tbe^{[K;m_2,m_3](1)} \right) . $ Applying Lemma \ref{lem10.12} and Lemma \ref{lem10.5} (1) to the second term,
  we get claim (b), using Proposition \ref{prop2.1}. The proof of claim (c) is similar. 
\end{proof}
The proof of the next two lemmas is similar. 
\begin{lemma}
\label{lem10.15}
Let $ \La = \La^{[K;m_2,m_3](0)} \in P^K_{+,0} $. Then 
\begin{enumerate}
\item[(a)]  $ \hat{R}\vphantom{R}^- \tch^{-[P]}_\La =- P_{m_3-m_2-2np} $ if $ \theta  $ and $ \al_0 $ are not integrable for $ \LLa. $
\item[(b)] $ \hat{R}\vphantom{R}^- \tch^{-[P]}_\La = 
-\frac{2}{j_\Lambda}
P_{n(p+q)} $ if $ \theta  $ is integrable for $ \LLa. $
\item[(c)] $ \hat{R}\vphantom{R}^- \tch^{-[P]}_\La = -\frac{1}{j_\La} \left( P_{m_2-m_3} +P_{m_3-m_2-2np} \right)  $ if $ \al_0 $ is integrable for $ \LLa.  $
\end{enumerate}
\end{lemma}
\begin{lemma}
\label{lem10.16}
\begin{enumerate}
\item[(a)] Let $ \La = \La^{[K;m_2,m_3](2)} \in P^K_{+,2}. $ Then
\begin{enumerate}
\item[(i)] $ \hat{R}\vphantom{R}^- \tch^{-[P]}_\La = -P_{-m_2-m_3-2} $ if $ \theta $ is not integrable for $ \LLa, $
\item[(ii)] $ \hat{R}\vphantom{R}^- \tch^{-[P]}_\La = -\frac{1}{j_\La} \left( P_{m_2+m_3+2} + P_{-m_2-m_3-2} \right)  $ if $ \theta $ is integrable for $ \LLa. $
\end{enumerate} 
\item[(b)] Let $ \La = \La^{[K;m_2,m_3](3)} \in P^K_{+,3}. $ Then
\begin{enumerate}
\item[(i)] $ \hat{R}\vphantom{R}^- \tch^{-[P]}_\La = -P_{m_2+m_3+2} $ if $ \theta $ is not integrable for $ \LLa, $
\item[(ii)] $ \hat{R}\vphantom{R}^- \tch^{-[P]}_\La = = -\frac{1}{j_\La} \left( P_{m_2+m_3+2} + P_{-m_2-m_3-2} \right)  $ if $ \theta $ is integrable for $ \LLa. $
\end{enumerate}
\end{enumerate}
\end{lemma}
\begin{lemma}
\label{lem10.17}
The same formulas as in Lemmas \ref{lem10.14}-\ref{lem10.16} hold (with the same proofs) if we replace $ [P] $ by $ [Q] $ in the LHS, and $ p $ with $ q $ and $ m_2 $ with $ m_3 $ in the RHS.
\end{lemma}
\begin{lemma}
\label{lem10.18}
\begin{enumerate}
\item[(a)] If $ \La = \La^{[K; m_2, m_3](1)}  \in P^K_{+,1}$ is such that $ m_2 = 0 $ or $ m_3 = 0, $ then
\[ \hat{R}\vphantom{R}^- \tch^{-[P]}_\La = P_{m_2 -m_3}; \quad \hat{R}\vphantom{R}^- \tch^{-[Q]}_\La = Q_{m_3 -m_2}. \]
\item[(b)] If $ \La = \La^{[K; m_2, m_3](0)}  \in P^K_{+,0}$ is such that $ m_2 = nq-1 $ or $ m_3 = np-1, $ then
\[ \hat{R}\vphantom{R}^- \tch^{-[P]}_\La = -P_{m_3 -m_2 - 2np}; \quad \hat{R} \tch^{-[Q]}_\La =- Q_{m_2 -m_3-2nq}. \]
\end{enumerate}
\end{lemma}

\begin{proof}
It follows from Lemmas \ref{lem10.14}-\ref{lem10.17}, along with the computation of $ j_\La, $ using Proposition \ref{prop2.1}.
\end{proof}

\begin{proposition}
\label{prop10.19} 
Denote by  $ V^{[P]} $ the span
of the set of functions $ \{ P_j |\, j \in \ZZ / 2n(p+q)\ZZ \} $ (over $ \CC  $).
Then  
\[\span \{ \hat{R}\vphantom{R}^- \tch^{-[P]}_\La |\, \La \in P^K_{+,0} \cup P^K_{+,1}  \}  = V^{[P]} = \span \{ \hat{R}\vphantom{R}^- \tch^{-[P]}_\La  |\, \La \in 
\bigcup\limits_{j=0}^{3} P^K_{+,j} \}.
\]
\end{proposition}
\begin{proof}
By the above discussion it follows that the space $ V^{[P]} $ contains all functions $ \hat{R}\vphantom{R}^- \tch^{-[P]}_\La $ for $ \La \in P^K_{+,j}, \ j = 0, \ldots, 3. $ Hence we need to check only that $ V^{[P]} $ is contained in the span of the set $ \{ \hat{R}\vphantom{R}^- \tch^{-[P]}_\La | \La \in P^K_{+,0} \cup P^K_{+,1} \}. $ This is implied by the following four equalities of sets:
\begin{equation}
\label{10.11}
\lbrace \hat{R}\vphantom{R}^- \tch^{-[P]}_{\La^{[K;0,m_3](1)}} \ | \ 0 \leq m_3\leq np
\rbrace = 
\lbrace P_j \ | \ -np \leq j \leq 0 
\rbrace ,
\end{equation}
\begin{equation}
\label{10.12}
\lbrace \hat{R}\vphantom{R}^- \tch^{-[P]}_{\La^{[K;m_2,0](1)}}  \ | \ 0 \leq m_2 \leq nq \rbrace =\lbrace P_j \ | \ 0 \leq j \leq nq \rbrace ,
\end{equation}
\begin{equation}
\label{10.13}
\lbrace  \hat{R}\vphantom{R}^- \tch^{-[P]}_{\La^{[K;nq-1,m_3](0)}} \ | \ 0 \leq m_3 \leq np-1\rbrace = 
\lbrace -P_j \ | \ nq+1\leq j \leq n(p+q)\rbrace ,  
\end{equation}
\begin{equation}
\label{10.14}
\lbrace \hat{R}\vphantom{R}^- \tch^{-[P]}_{\La^{[K;m_2,np-1](0)}} \ | \ 0 \leq m_2 \leq nq-1\rbrace =
\lbrace 
-P_j \ | \ -n(p+q) \leq j \leq -np-1\rbrace .
\end{equation}
These equalities follow from Lemma \ref{lem10.18}.
\end{proof}
One proves the next statement along the same lines.
\begin{proposition}
\label{prop10.20}
 The same result as Proposition \ref{prop10.19} holds if we replace $ V^{[P]} $ by $ V^{[Q]} = \span \{ Q_j \ | \ j \in \ZZ / 2n(p+q) \ZZ \} $ and $ \tch^{-[P]}_\La $ by $ \tch^{-[Q]}_\La $.
\end{proposition}

The modular transformations of the functions $ P_j $'s and $ Q_j $'s are given by the following proposition.
\begin{proposition}
\label{prop10.21}
For $ j \in \ZZ / 2n(p+q) \ZZ $ we have:
\[ 
\begin{aligned}
P_j \left( -\frac{1}{\tau}, \frac{z_1}{\tau}, \frac{z_2}{\tau}, \frac{z_3}{\tau}, t \right) = \frac{(-i\tau)^\half \tau}{\sqrt{2n(p+q)}} e^{\frac{\pi i K}{\tau}(z|z)} 
\hspace{-4ex} \sum_{k\in\ZZ /2n(p+q)\ZZ} \hspace{-4ex} e^{-\frac{\pi i jk}{n(p+q)}} P_k \tzzzt,
\end{aligned} \]
where $ z $ is given by \eqref{10.07};
\[ P_j (\tau+1, z_1,z_2,z_3, t) = e^{\frac{\pi i j^2}{2n(p+q)}} P_j \tzzzt. \]
The same formulas hold with $ P_j $'s replaced by $ Q_j $'s.
\end{proposition}
\begin{proof}
First note the following four expressions for $ a(z|z) = -\frac{pq}{p+q} (z|z), $ which follow easily from pairwise orthogonality of roots $ \al_2, \al_3 $ and $ \theta: $
\begin{equation}
\label{10.15} 
\begin{aligned}
&\frac{p+q}{2} (z_1 -(a+1)z_2-az_3)^2 -2pz_1z_3 \\
= & \frac{p+q}{2} (z_1 +(a-1)z_2 + az_3)^2 - 2pz_2 (-z_1+z_2+z_3) \\
 =& \frac{p+q}{2} (z_1 + (a+1)z_2 + az_3)^2 -2qz_1z_2 \\
=& \frac{p+q}{2} (z_1-(a+1)z_2 -(a+2)z_3)^2 -2qz_3 (-z_1+z_2+z_3).
\end{aligned} 
\end{equation}
Recall that the functions $ P_j $ (resp. $ Q_j $) are expressed before Lemma \ref{lem10.12} in terms of functions $ \Theta_{j,n(p+q)} $ and $ \tph^{[np;0]} $ (resp. $ \tph^{[nq;0]} $). But the modular transformation of these functions are given by \eqref{1.3}, \eqref{1.4} and \eqref{1.15}, \eqref{1.16}. Applying \eqref{1.3} and \eqref{1.15} and using the first two (resp. last two) expressions for $ a(z|z) $ given by \eqref{10.15}, it is straightforward to derive the $ S $-transformation of $ P_j $ (resp. $ Q_j $). Formulas for the $ T $-transformation are immediate by \eqref{1.4} and \eqref{1.16}. 
\end{proof}
Recall that, by \cite{KW14}, Theorem 4.1, the function $ \hat{R}^- = \hat{R}^{(0)}_0 $ is $ SL(2, \ZZ) $-invariant of weight $ \frac{3}{2}. $ Hence Propositions 
\ref{prop10.19}, \ref{prop10.20},  along with Proposition \ref{prop10.21}, 
imply the following corollary. 
\begin{corollary}
\label{cor10.22}
The span of the set 
$ \lbrace  \tch^{-[P]}_\La \ (\mbox{resp. } \tch^{-[Q]}_\La)|\, \La\in \bigcup\limits^3_{j=0} P^K_{+, j}  \rbrace $ is $ SL_2 (\ZZ) $-invariant.
\end{corollary}

\section{QHR of integrable $ \wg $-modules, where $ \fg = D(2,1;a) $
, and a modular invariant family of big $N=4$ representations}

Throughout this section, $ \fg = D(2,1;a) $, and we use notation of the previous section.  

In order to construct a modular invariant family of modified (super)characters of big $ N = 4 $ superconformal algebras, we need also to consider the Ramond twisted $ \wg $-modules (cf. \cite{KW14}, Section 4). We need to choose the twist $ w_0, $ such that in the twisted QHR we have
\begin{equation}
\label{11.1}
s_{\alpha_1}=s_{\alpha_1+\alpha_2}=\half,\,\,s_{\alpha_1+\alpha_3}=s_{\alpha_1+\alpha_2+\alpha_3}=-\half,
\end{equation}
so that the generators
\[ G^{[\al_1]}_0, G^{[\al_1 + \al_2]}_0 \quad (\mbox{resp. } G^{[\al_1+\al_3]}_0, G^{[\al_1+\al_2+\al_3]}_0) \]
are annihilation (resp. creation) operators of the Ramond twisted big $ N=4 $ superconformal algebra (see \cite{KW05}). It is easy to see that this is the case if we choose
\begin{equation}
\label{11.2}
w_0 = t_{\xi} r_{\al_3}, \mbox{ where } \xi = \half \al^\vee_3.
\end{equation}
Note that in coordinates \eqref{10.06}, \eqref{10.07} we have for $ h = \tzzzt: $
\begin{equation}
\label{11.3}
\begin{aligned}
h + 2 \pi i \xi &= \left( \tau, z_1 +\half, z_2 + \half, z_3 - \half, 
t \right), \\w_0 (h) & = \left( z_3 - \tot, -z_1 +z_2+z_3 - \tot, z_1 + \tot, t - \frac{z_1-z_3}{2(a+1)} -\frac{\tau}{4(a+1)} \right).
\end{aligned}
\end{equation}

Recall that for $ \La = \La^{[K;m_2, m_3](j)} \in P^K_{+,j}, j = 0,1,2,3, $ the numerator of the normalized supercharacter of the $ \wg $-module $ \LLa $ is a simple expression in terms of the function
\begin{equation}
\label{11.4}
 B^-_\La := q^{\frac{|\La + \rho|^2}{2K}} \sum_{w \in \hat{W}^\#_0} \epsilon(w)w \frac{e^{\La + \rho}}{1-e^{-\beta_j}}. 
\end{equation}
Then the numerator of the normalized character of the $ \wg $-module is the same expression in terms of the function (see \cite{KW16}, Section 6)
\begin{equation}
\label{11.5}
 B^+_\La (h) =e^{-2 \pi i (\La + \rho | \xi)} B_\La (h + 2 \pi i \xi).
\end{equation}
Furthermore, the numerator of the normalized supercharacter (resp. character) of the corresponding Ramond twisted $ \wg $-module is the same expression in terms of the function 
\begin{equation}
\label{11.6}
 B_\La^{-, \tw} = B^-_\La (w_0(h))\, (\hbox{resp.}\, B^{+, \tw}_\La (h) = e^{-2 \pi i (\La + \rho | \xi)}B_\La^{-,\tw} (h + 2 \pi i \xi) ).
\end{equation}
Recall that Lemma \ref{lem10.9} provides two expressions of the functions $ F, $ and hence $ B, $ in terms of theta functions \eqref{1.1} and the mock theta functions $ \Phi_1 $ (see \eqref{1.8}), and we denoted by $ \tbe_P $ and $ \tbe_Q $ the functions, obtained from $ B $ replacing $ \Phi_1 $ by its modification $ \tph_1 $ in these two expressions respectively. Then the numerators of modified normalized supercharacters of the $ \wg $-modules $ \LLa $ are expressed in terms of functions $ P_j $ and $ Q_j $ by Lemmas 
\ref{lem10.14}--\ref{lem10.18}.

The numerators of modified normalized characters and twisted (super)characters of $ \LLa $ are expressed by the same formulas in terms of the following functions $ (j \in \ZZ): $
\[ \begin{aligned}
P^-_j  : = P_j, \ P^+_j (h) :=P^-_j (h + 2 \pi i \xi), \ P^{-, \tw}_j (h)  = P_j (w_0(h)), \ P^{+, \tw}_j (h) = P^{-, \tw}_j (h+ 2 \pi i \xi),
\end{aligned} \]
and similarly for $ P $ replaced by $ Q. $

It turns out that these functions are expressed in terms of theta functions \eqref{1.1} and the modifications of the mock theta functions $ \Psi^{[1, m; 0; \epsilon]}_{\epsilon',-\epsilon'; \epsilon'}, $ cf \eqref{1.10}. This is given by the following lemma. Its proof is straightforward. 
\begin{lemma}
\label{lem11.1}
For $ j \in \ZZ $ on has:
\begin{enumerate}
\item[(a)] 
\begin{enumerate}
\item[(i)] \[ \begin{aligned}
&P^-_j \tzzzt  = e^{2\pi i K t} \left( \Theta_{j;n(p+q)} (\tau, z_1 -(a+1)z_2-az_3) \tps^{[1, np; 0;0]}_{0,0;0} (\tau, z_1, -z_3, 0) \right. \\
& \left. -\Theta_{-j;n(p+q)} (\tau, z_1+(a-1)z_2+az_3) \tps^{[1, np;0,0]}_{0;0;0} (\tau, z_1-z_2-z_3, z_2, 0) \right),  
\end{aligned} \]
\item[(ii)] \[ 
\begin{aligned}
&Q^-_j \tzzzt = e^{2\pi i K t} \left( \Theta_{j;n(p+q)} (\tau, z_1 + (a+1)z_2 + az_3) \tps^{[1, nq; 0,0]}_{0,0;0} (\tau, z_1, -z_2, 0) \right. \\
&- \left. \Theta_{-j;n(p+q)} (\tau, z_1-(a+1)z_2-(a+2)z_3) \tps^{[1, nq;0;0]}_{0,0;0} (\tau, z_1-z_2-z_3, z_3,0) \right).
\end{aligned}  \]
\end{enumerate}
\item[(b)]
\begin{enumerate}
\item[(i)] \[ \begin{aligned}
&P^+_j \tzzzt = e^{2\pi i K t} \left( \Theta_{j;n(p+q)} (\tau, z_1-(a+1)z_2-az_3) \tps^{[1,np;0; \half]}_{0,0;0} (\tau, z_1, -z_3,0) \right. \\
& \left. -\Theta_{-j; n(p+q)} (\tau, z_1+(a-1)z_2+az_3) \tps^{[1,np;0;\half]}_{0,0;0} (\tau, z_1-z_2-z_3, z_2, 0) \right),
\end{aligned}
 \]
\item[(ii)] \[ \begin{aligned}
 & Q^+_j \tzzzt = e^{2\pi i K t} \left( \Theta_{j;n(p+q)} (\tau, z_1+(a+1)z_2+az_3) \tps^{[1, nq;0;\half]}_{0,0;0} (\tau, z_1, -z_2, 0) \right. \\
& \left. - \Theta_{-j;n(p+q)} (\tau, z_1-(a+1)z_2-(a+2)z_3) \tps^{[1, nq;0;\half]}_{0,0;0} (\tau, z_1-z_2-z_3, z_3, 0) \right). 
\end{aligned}
 \]
\end{enumerate}
\item[(c)]
\begin{enumerate}
\item[(i)] \[ \begin{aligned}
& P^{-, \tw}_j \tzzzt = \\
& -e^{2\pi i K t} \left( \Theta_{j;n(p+q)} (\tau, z_1-(a+1)z_2-az_3)  \tps^{[1,np;0;0]}_{\half, -\half; \half} (\tau, z_1, -z_3, 0) \right. \\
& \left. -\Theta_{-j;n(p+q)} (\tau, z_1+(a-1)z_2+az_3) \tps^{[1,np;0;0]}_{\half, -\half; \half} (\tau, z_1-z_2-z_3, z_2, 0) \right),
\end{aligned}
 \]
\item[(ii)] \[ \begin{aligned}
& Q^{-, \tw}_j \tzzzt = \\
& -e^{2\pi i K t} \left( \Theta_{j + n(p+q), n(p+q)} (\tau, z_1+(a+1)z_2 + az_3) \tps^{[1,nq;0;0]}_{\half, -\half; \half} (\tau, z_1, -z_2, 0)  \right. \\
& \left. - \Theta_{-j - n(p+q), n(p+q)}  (\tau, z_1 -(a+1)z_2 -(a+2)z_3) \tps^{[1,nq;0;0]}_{\half, -\half; \half} (\tau, z_1-z_2-z_3, z_3,0) \right).
\end{aligned}
 \]
\end{enumerate}
\item[(d)]
\begin{enumerate}
\item[(i)] \[ \begin{aligned}
& P^{+, \tw}_j \tzzzt = \\
& -e^{-2\pi i K t} \left( \Theta_{j;n(p+q)} (\tau, z_1-(a+1)z_2-az_3) \tps^{[1, np;0;\half]}_{\half, -\half;\half} (\tau, z_1, -z_3,0)\right. \\
& \left. - \Theta_{-j;n(p+q)} (\tau, z_1+(a-1)z_2 + az_3) \tps^{[1, np;0;\half]}_{\half, -\half;\half} (\tau, z_1-z_2-z_3, z_2, 0) \right), 
\end{aligned}
 \]
\item[(ii)] \[ \begin{aligned}
& Q^{+, \tw}_j \tzzzt = \\
& -(-1)^{j+np} e^{2\pi i K t} \left( \Theta_{j+n(p+q); n(p+q)} (\tau, z_1+(a+1)z_2 + az_3) \tps^{[1,nq;0;\half]}_{\half, -\half;\half} (\tau, z_1, -z_2, 0) \right. \\
& \left. - \Theta_{-j-n(p+q); n(p+q)} (\tau, z_1-(a+1)z_2 -(a+2)z_3) \tps^{[1,nq;0;\half]}_{\half, -\half;\half} (\tau, z_1-z_2-z_3, z_3, 0) \right). 
\end{aligned} \]
\end{enumerate}
\end{enumerate}

\end{lemma}

In view of these formulas, we introduce the following functions:
\[ 
\begin{aligned}
f_j^{[\epsilon, \epsilon']} \tzzzt & :=   e^{2 \pi i Kt} \left(  \Theta_{j;n(p+q)} (\tau, z_1-(a+1)z_2-az_3) \tps^{[1,np;0;\epsilon]}_{\epsilon', -\epsilon'; \epsilon'} (\tau, z_1, -z_3, 0) \right. \\
& \left. - \Theta_{-j;n(p+q)} (\tau, z_1 + (a-1)z_2+az_3) \tps^{[1,np;0;\epsilon]}_{\epsilon', -\epsilon'; \epsilon'} (\tau, z_1-z_2-z_3, z_2,0) \right), \\
g_j^{[\epsilon, \epsilon']} \tzzzt & := e^{2 \pi i K t} \left( \Theta_{j;n(p+q)} (\tau, z_1+(a+1)z_2+az_3) \tps^{[1,nq;0;\epsilon]}_{\epsilon', -\epsilon'; \epsilon'} (\tau, z_1, -z_2, 0) \right.  \\
& \left.  - \Theta_{-j;n(p+q)} (\tau, z_1-(a+1)z_2-(a+2)z_3) \tps^{[1,nq;0;\epsilon]}_{\epsilon', -\epsilon'; \epsilon'} (\tau, z_1-z_2-z_3, z_3, 0) \right). 
\end{aligned} \]
Then Lemma \ref{lem11.1} is rewritten as follows:
\begin{lemma}
\label{lem11.2}
\begin{enumerate}
\item[(a)]
\[ P^-_j = f_j^{[0,0]}, \ P^+_j = f_j^{[\half,0]}, \ Q^-_j = g_j^{[0,0]},  \ Q^+_j = (-1)^j g_j^{[\half,0]},   \] 
\item[(b)]
\[ P^{-, \tw}_j = -f_j^{[0,\half]}, \ P^{+, \tw}_j = -f_j^{[\half,\half]}, \ Q^{-, \tw}_j = -g^{[0,\half]}_{j+n (p+q)}, \ Q^{+, \tw}_j = -(-1)^{j+np} g_{j+n(p+q)}^{[\half,\half]}. \]
\end{enumerate}
\end{lemma}


Now we turn to the quantum Hamiltonian reduction of the integrable $ \wg $-modules $ \LLa, $ where $ \La \in \bigcup\limits^3_{j=0} P^K_{+,j}, $ where K is given by \eqref{10.03}.

Recall (see \cite{KRW03}, \cite{A05}) that the QHR associates to a $ \wg $-module $ \LLa $ of level $ K $ where $ \La = \sum^{3}_{i=0} m_i \La_i, $ a module $ H(\La) $ over the big $ N =4 $ superconformal algebra of Neveu-Schwarz type, such that $ H(\La) $ is either 0 or an irreducible positive energy module. Moreover, $ H(\La) = 0 $ iff $ m_0 \in \ZZ_{\geq 0}, $ and an irreducible module $ H(\La) $ is characterized by four numbers:

\begin{enumerate}
\item[($ \al $)] the central charge
\begin{equation}
\label{11.7}
c_K = \frac{6npq}{p+q} -3,
\end{equation}
\item[($ \beta $)] the lowest energy $ h_\La = \frac{(\La + 2 \rho | \La)}{2 K} -\left( \frac{\theta}{2}+ \La_0 \middle| \La \right), $
\item[($ \gamma $)] the two spins $ s_\La^{(i)} =\half (\La | \al^\vee_i), \ i = 2,3. $
\end{enumerate}
Finally, the (super)character of the irreducible module $ H(\La) $ is given by the following formula:
\begin{equation}
\label{11.8}
\begin{aligned}
  & \ch^\pm_{H(\La)} (\tau, y_1, y_2)  :=
  \tr^\pm_{H(\La)} q^{L_0 -\frac{c_K}{24}} e^{ \pi i (y_2\al^\vee_2 + y_3 \al^\vee_3)}
  \\
& = \left(\hat{R}^\pm \ch^\pm_\La \right) \left( \tau, \frac{\tau + y_2 + y_3}{2}, \frac{\tau - y_2 + y_3}{2}, \frac{\tau + y_2 - y_3}{2}, \tof   \right) \overset{b4}{R} \vphantom{R}^\pm (\tau, y_2, y_3)^{-1},
\end{aligned}
\end{equation}
where $ \overset{b4}{R} \vphantom{R}^+ = \overset{b4}{R} \vphantom{R}^{(\half)}_\half, \ \overset{b4}{R} \vphantom{R}^- = \overset{b4}{R} \vphantom{R}^{(0)}_\half, $ and 
\begin{equation}
\label{11.9}
\overset{b4}{R} \vphantom{R}^{ (\epsilon)}_{\epsilon'} (\tau, y_2, y_3) = \eta (\tau)^3 \frac{\vartheta_{11} (\tau, y_2) \vartheta_{11} (\tau, y_3)}{\vartheta_{1-2\epsilon', 1 - 2 \epsilon} \left(\tau, \frac{y_2+y_3}{2} \right) \vartheta_{1-2\epsilon', 1 - 2 \epsilon} \left(\tau, \frac{y_2-y_3}{2}\right)}
\end{equation}
for $ \epsilon, \epsilon' = 0 $ or $ \half. $

Similar results hold in the Ramond twisted case. Twisting the $ \wg  $-module $ \LLa $ by $ w_0,  $ defined by \eqref{11.2}, we obtain, by twisted QHR \cite{KW05}, a module $ H^\tw (\La) $ over the big $ N=4 $ superconformal algebra of Ramond type, such that properties, analogous to that of $ H(\La) $ hold with the following changes (see \cite{KW14}):
\begin{equation}
\label{11.10}
h^\tw_\La = \frac{(w_0 (\La) +2 \rho^\tw| w_0 (\La))}{2K} - \left( \half \theta + \La_0 \middle| w_0 (\La)\right) - \frac{1}{4}, 
\end{equation}

\begin{equation}
\label{11.11}
s^{(i), \tw}_\La =\half (w_0(\La) | \al^\vee_i) + \frac{1}{4} \sum_{\substack{\al \in \Delta \\ (\al| \theta)=1}} s_\al (\al| \al^\vee_i), 
\end{equation}
and the (super)character of an irreducible module $H^\tw(\Lambda)$ is given by the following formula:
\begin{equation}
\label{11.12}
\begin{aligned}
& \left( \overset{b4}{R} \vphantom{R}^{\pm,\tw} \ch^\pm_{H^\tw(\La)} \right) (\tau, y_2, y_3)  :=\tr^\pm_{H^\tw(\La)} q^{L^\tw_0 -\frac{c_K}{24}} e^{ \pi i (y_2\al^\vee_2 + y_3 \al^\vee_3)}\\
&  = \left( \hat{R} \vphantom{R}^{\pm,\tw} \ch^{\pm,\tw}_{\La} \right) \left( \tau, \frac{\tau + y_2+y_3}{2}, \frac{\tau-y_2+y_3}{2}, \frac{\tau + y_2 -y_3}{2}, \tof \right) \overset{b4}{R} \vphantom{R}^{\pm,\tw} (\tau, y_1, y_2)^{-1},
\end{aligned}
\end{equation}
where $ \overset{b4}{R} \vphantom{R}^{+,\tw} = \overset{b4}{R} \vphantom{R}^{(\half)}_0, \ \overset{b4}{R} \vphantom{R}^{-,\tw} = \overset{b4}{R} \vphantom{R}^{(0)}_0. $


It is straightforward to compute the characteristic numbers of the big $ N = 4 $ superconformal algebra modules $ H(\La) $ and $ H^\tw (\La). $

\begin{proposition}
\label{prop11.3}
Let $ \La = \La^{[K; m_2, m_3](j)}, \ j = 0,1,2,3, $ and assume that $ m_0 \notin \ZZ_{\geq 0}. $
\begin{enumerate}
\item[(a)] For $ j = 0,1 $  on has:
\[ h_\La = \frac{(m_2-m_3+np)^2}{4n(p+q)} + \frac{m_3}{2} - \frac{np^2}{4(p+q)}, \quad h^\tw_\La = \frac{(m_2-m_3+np)^2}{4n(p+q)} + \frac{npq}{4(p+q)} - \frac{1}{4}. \]
\item[(b)] For $ j = 2 $ one has: 
\[ h_\La = \frac{(m_2+m_3+2-np)^2}{4n(p+q)} + \frac{m_3}{2} - \frac{np^2}{4(p+q)}, \quad h^\tw_\La = \frac{(m_2 +m_3 +2 -np)^2}{4n(p+q)} + \frac{npq}{4(p+q)} -\frac{1}{4}. \]
\item[(c)] For $ j = 3 $ on has:
\[ h_\La = \frac{(m_2\!+m_3\! +2\! +\! np)^2}{4n(p+q)} - \frac{m_3}{2} - \frac{np^2}{4(p+q)} -1,  h^\tw_\La = \frac{(m_2\!+\!m_3\!+2\! +\! np)^2}{4n(p+q)} + \frac{npq}{4(p+q)} -m_3 -\frac{5}{4}. \]
\item[(d)] For all $ j $ one has:
\[ s^{(i)}_\La = \frac{m_i}{2}, \ i = 2,3; \quad s^{(2), \tw}_\La = \frac{m_2}{2}; \quad s^{(3),\tw}_\La = \frac{np-1-m_3}{2}. \]
\end{enumerate}
\end{proposition}
\begin{corollary}
\label{cor11.4}
The modules $ H(\La^{[K;m_2,m_3](0)}) $ and $ H(\La^{[K;m_2,m_3](1)}) $ are isomorphic, and the same holds for $ H^\tw. $
\end{corollary}

Using modular transformation formulas of the functions $ \vartheta_{ab} \tz, $ see e.g. \cite{KW14}, Proposition A.7, one obtains the following modular transformation formulas of the functions $ \overset{b4}{R} \vphantom{R}^{(\epsilon)}_{\epsilon'}  (\tau, y_1, y_2): $
\begin{equation}
\label{11.13}
\overset{b4}{R} \vphantom{R}^{(\epsilon)}_{\epsilon'} \left( -\frac{1}{\tau}, \frac{y_2}{\tau}, \frac{y_3}{\tau} \right) =  -(-1)^{(1-2\epsilon )(1-2 \epsilon')} (-i \tau)^{\frac{3}{2}} e^{\frac{\pi i }{2 \tau} (y^2_2+y^2_3)}   \overset{b4}{R} \vphantom{R}^{(\epsilon)}_{\epsilon'} (\tau, y_2, y_3),
\end{equation}
\begin{equation}
\label{11.14}
\overset{b4}{R} \vphantom{R}^{(\epsilon)}_{\epsilon'} (\tau +1, y_2, y_3) =
i^{2 \epsilon'} e^{\frac{\pi i }{4}} \overset{b4}{R} \vphantom{R}^{(|\epsilon-\epsilon'|)}_{\epsilon'} (\tau, y_2, y_3).
\end{equation}

In view of formulas \eqref{11.8} and \eqref{11.12}, we introduce the ``Hamiltonian reduction'' $ \varphi_{HR} $ 
of a function $ \varphi \tzzzt $ by the formula
\[ \varphi_{HR} (\tau, y_2, y_3) = \varphi \left( \tau, \frac{\tau+y_2+y_3}{2}, \frac{\tau - y_2 + y_3}{2}, \frac{\tau + y_2 - y_3}{2}, \tof \right) . \]
\begin{lemma}
\label{lem11.5}
The Hamiltonian reduction of $ f_j^{[\epsilon, \epsilon']} $ and $ g_j^{[\epsilon, \epsilon']} $ is as follows:
\begin{enumerate}
\item[(a)] 
\[ \begin{aligned}
& f_{j, HR}^{[\epsilon, \epsilon']} (\tau, y_2, y_3)=  \\
& \left( \Theta_{j+np; n(p+q)} - \Theta_{-j-np; n(p+q)} \right) (\tau, y_2) \tps^{[1; np; 0;\epsilon]}_{\half - \epsilon', \epsilon'-\half; \half - \epsilon'} \left( \tau, \frac{y_2+y_3}{2}, \frac{y_3-y_2}{2} \right). \\
\end{aligned}
 \]
\item[(b)]
\[ \begin{aligned}
&  g_{j, HR}^{[\epsilon, \epsilon']} (\tau, y_2, y_3)  = \\
& \left( \Theta_{j+nq; n(p+q)} - \Theta_{-j-nq; n(p+q)} \right) (\tau, y_3) \tps^{[1; nq; 0;\epsilon]}_{\half - \epsilon', \epsilon'-\half; \half - \epsilon'} \left( \tau, \frac{y_2+y_3}{2}, \frac{y_2-y_3}{2} \right). 
\end{aligned}
\]
\end{enumerate}
\end{lemma}
\begin{proof}
It is straightforward using the elliptic transformation properties \eqref{1.2} of theta functions and the definition of the functions $ \tps $, cf. \eqref{1.10}.
\end{proof}

In view of this lemma, we introduce the following functions $ (j \in \ZZ): $
\[ 
\begin{aligned}
& F_j^{[\epsilon, \epsilon']} (\tau, y_2, y_3) = \left( \Theta_{j ; n(p+q)} - \Theta_{-j; n(p+q)} \right) (\tau, y_2) \tps^{[1;np;0,\epsilon]}_{\epsilon', -\epsilon'; \epsilon'} \left( \tau, \frac{y_2+y_3}{2}, \frac{y_3-y_2}{2} \right), \\
& G_j^{[\epsilon, \epsilon']} (\tau, y_2, y_3) = \left( \Theta_{j ; n(p+q)} - \Theta_{-j; n(p+q)} \right) (\tau, y_3) \tps^{[1;nq;0,\epsilon]}_{\epsilon', -\epsilon'; \epsilon'} \left( \tau, \frac{y_2+y_3}{2}, \frac{y_2-y_3}{2} \right).
\end{aligned} \]

It is easy to see the following.

\begin{remark}
\label{rem11.6}
\[ 
\begin{aligned}
& F^{[\epsilon,\epsilon']}_{-j} =-  F^{[\epsilon,\epsilon']}_{j}, \quad G^{[\epsilon,\epsilon']}_{-j} = -G^{[\epsilon,\epsilon']}_{j}; \\
& F^{[\epsilon,\epsilon']}_{j+2kn(p+q)} = F^{[\epsilon,\epsilon']}_{j}, \quad     G^{[\epsilon,\epsilon']}_{j+2kn (p+q)} = G^{[\epsilon,\epsilon']}_{j} \ (k \in \ZZ); \\
& F^{[\epsilon,\epsilon']}_{n(p+q)+j} = -F^{[\epsilon,\epsilon']}_{n(p+q)-j}, \quad  G^{[\epsilon,\epsilon']}_{n(p+q)+j} =  -G^{[\epsilon,\epsilon']}_{n(p+q)-j}. \\
\end{aligned} \]
\end{remark}
Then Lemma \ref{lem11.5} is restated as follows:
\begin{lemma}
\label{lem11.7}
\[ f_{j, \HR}^{[\epsilon, \epsilon']} = F^{[\epsilon, \half-\epsilon']}_{j + np}, \quad g_{j, \HR}^{[\epsilon, \epsilon']} = G^{[\epsilon, \half-\epsilon']}_{j + nq}. \]
\end{lemma}

The modular transformations of $ F_j^{[\epsilon, \epsilon']} $ and $ G_j^{[\epsilon, \epsilon']} $ are given by the following:
\begin{lemma}
\label{lem11.8}
\begin{enumerate}
\item[(a)]
\[ 
\begin{aligned}
& F_j^{[\epsilon, \epsilon']} \left( - \frac{1}{\tau}, \frac{y_2}{\tau}, \frac{y_3}{\tau}\right) =\\
&  (-i \tau)^\frac{3}{2} e^{4 \pi i np \epsilon\epsilon'} \sqrt{\frac{2}{n(p+q)}} e^{\frac{\pi i n}{2 \tau}(qy^2_2+py^2_3)} \sum_{k=1}^{n(p+q)-1} \sin \frac{\pi jk}{n (p+q)}  F_k^{[\epsilon',\epsilon]} (\tau, y_2, y_3);\\
& G_j^{[\epsilon, \epsilon']} \left( - \frac{1}{\tau}, \frac{y_2}{\tau}, \frac{y_3}{\tau}\right) = \\
&  (-i \tau)^\frac{3}{2} e^{4 \pi i nq \epsilon\epsilon'} \sqrt{\frac{2}{n(p+q)}} e^{\frac{\pi i n}{2 \tau}(qy^2_2+py^2_3)} \sum_{k=1}^{n(p+q)-1} \sin \frac{\pi jk}{n (p+q)}  G_k^{[\epsilon',\epsilon]} (\tau, y_2, y_3).
\end{aligned} \] 
\item[(b)]
\[ 
\begin{aligned}
&F_j^{[\epsilon, \epsilon']} \left( \tau +1, y_2, y_3 \right)  =  e^{\frac{\pi i j^2}{2n(p+q)}} e^{-2 \pi i n p \epsilon'^2} F_j^{[|\epsilon-\epsilon'|, \epsilon']} (\tau, y_2, y_3); \\
& G_j^{[\epsilon, \epsilon']} \left( \tau +1, y_2, y_3 \right)  =e^{\frac{\pi i j^2}{2n(p+q)}} e^{-2 \pi i n q \epsilon'^2} G_j^{[|\epsilon-\epsilon'|, \epsilon']} (\tau, y_2, y_3). \\
\end{aligned} \]
\end{enumerate}
\end{lemma}
\begin{proof}
Using \eqref{1.3}, we obtain for a positive integer $ m $ and $ j \in \ZZ: $
\[ 
\begin{aligned}
\left( \Theta_{j,m} - \Theta_{-j,m}\right) \left( -\frac{1}{\tau}, \frac{z}{\tau}\right) = -i (-i\tau)^\frac{1}{2} \sqrt{\frac{2}{m}} e^{\frac{\pi i m}{2 \tau}z^2} \sum_{k=1}^{m-1} \sin \frac{\pi j k }{m} \left( \Theta_{k,m} - \Theta_{-k,m} \right)  \tz.
\end{aligned} \]
>From \eqref{1.17} and \eqref{1.18}, we obtain, for a positive integer $ m $
and $ \epsilon, \epsilon' = 0 $ or $ \half $: 
\[ 
\begin{aligned}
& \tps^{[1;m;0;\epsilon]}_{\epsilon', -\epsilon';\epsilon'} \left( -\frac{1}{\tau}, \frac{z_1}{\tau}, \frac{z_2}{\tau} \right) = \tau e^{4 \pi i m \epsilon\epsilon'} e^{\frac{2 \pi i m }{\tau} z_1z_2}  \tps^{[1;m;0;\epsilon']}_{\epsilon, -\epsilon; \epsilon} (\tau, z_1, z_2); \\
& \tps^{[1;m;0;\epsilon]}_{\epsilon', -\epsilon';\epsilon'} (\tau + 1, z_1, z_2) = e^{-2 \pi i m \epsilon'^2} \tps^{[1;m;0;|\epsilon-\epsilon'|]}_{\epsilon', -\epsilon'; \epsilon'} (\tau, z_1, z_2).
\end{aligned} \]
Lemma follows by applying these transformation formulas, along with \eqref{1.4}, to the expressions of the functions $ F_j^{[\epsilon, \epsilon']} $ and $ G_j^{[\epsilon, \epsilon']}, $ given before Remark \ref{rem11.6}.
\end{proof}

Now we return to the functions $ B^\pm_\La $ and $ B^{\pm, \tw}_\La, $ defined by \eqref{11.4}--\eqref{11.6}.
Recall from Section 10 that for $ \La = \La^{[K; m_2, m_3](1)} $ the function $ B_\La = B^{[K; m_1, m_3](1)} $ is a linear combination of functions 
$w( F^{[K; m_2, m_3](1)}), w \in W^\#_0, $ and that the function $ F^{[K; m_2, m_3](1)} $ is expressed in two ways in terms of the mock theta function $ \Phi_1, $ given by Lemma \ref{lem10.9}(a) and (b). Recall that we defined the modification $ \tbe^{[K;m_2, m_3](1)}_{P (\mbox{resp. }Q)} $ by replacing $ \Phi_1 $ by $ \tph_1  $ in these expressions. Furthermore, using Remark \ref{rem10.7}, we defined the modifications $ \tbe^{[K;m_2, m_3](j)}_{P (\mbox{resp. }Q)} $ for the remaining $ j = 0,2,3. $ We showed in Lemma \ref{lem10.13}  that all the functions $ \tbe^{[K;m_2, m_3](j)}_{P (\mbox{resp. }Q)}  $ are expressed in terms of functions $ P_j $ (resp. $ Q_j $), introduced before Lemma \ref{lem10.12}.

Using \eqref{11.4}--\eqref{11.6} we can now define modifications $ \tbe^{[K;m_2, m_3](j)\pm}_{P (\mbox{resp. }Q)}  $ and $ \tbe^{[K;m_2, m_3](j)\pm, \tw}_{P (\mbox{resp. }Q)}.  $ Using Lemmas \ref{10.13} and \ref{lem11.2}, we obtain simple expressions for all these modifications in terms of functions $ f_j^{[\epsilon, \epsilon']} $ (resp. $ g_j^{[\epsilon, \epsilon']} $), introduced before Lemma \ref{lem11.2}:

\begin{lemma}
\label{lem11.9}
\begin{enumerate}
\item[(a)]
\[ 
\begin{aligned}
\tbe^{[K;m_2, m_3](0)-}_P &= -f^{[0,0]}_{m_3-m_2-2np}, \quad \tbe^{[K;m_2, m_3](1)-}_P = f^{[0,0]}_{m_2-m_3},\\
\tbe^{[K;m_2, m_3](2)-}_P &= -f^{[0,0]}_{-m_2-m_3-2}, \quad \tbe^{[K;m_2, m_3](3)-}_P = -f^{[0,0]}_{m_2+m_3+2}.\\
\end{aligned} \]
\item[(b)]
\[ 
\begin{aligned}
\tbe^{[K;m_2, m_3](0)+}_P &= (-1)^{m_3} f^{[\half,0]}_{m_3-m_2-2np}, \quad \tbe^{[K;m_2, m_3](1)+}_P = -(-1)^{m_3} f^{[\half,0]}_{m_2-m_3},\\
\tbe^{[K;m_2, m_3](2)+}_P &= (-1)^{m_3} f^{[\half,0]}_{-m_2-m_3-2}, \quad \tbe^{[K;m_2, m_3](3)+}_P = (-1)^{m_3} f^{[\half,0]}_{m_2+m_3+2}.\\
\end{aligned} \]
\item[(c)]
\[ 
\begin{aligned}
\tbe^{[K;m_2, m_3](0)-, \tw}_P & = f^{[0, \half]}_{m_3-m_2-2np}, \quad \tbe^{[K;m_2, m_3](1)-, \tw}_P  = -f^{[0, \half]}_{m_2-m_3}, \\
\tbe^{[K;m_2, m_3](2)-, \tw}_P & = f^{[0, \half]}_{-m_2-m_3-2}, \quad \tbe^{[K;m_2, m_3](3)-, \tw}_P  = f^{[0, \half]}_{m_2+m_3+2}. \\
\end{aligned} \]
\item[(d)]
\[ 
\begin{aligned}
\tbe^{[K;m_2, m_3](0)+, \tw}_P &= -(-1)^{m_3} f^{[\half,\half]}_{m_3-m_2-2np}, \quad \tbe^{[K;m_2, m_3](1)+, \tw}_P = (-1)^{m_3} f^{[\half,\half]}_{m_2-m_3}, \\
\tbe^{[K;m_2, m_3](2)+, \tw}_P &= -(-1)^{m_3} f^{[\half,\half]}_{-m_2-m_3-2}, \quad 
\tbe^{[K;m_2, m_3](3)+, \tw}_P = -(-1)^{m_3} f^{[\half,\half]}_{m_2+m_3+2}. \\
\end{aligned} \]
\end{enumerate}
\end{lemma}

\begin{lemma}
\label{lem11.10}
The same formulas as in Lemma \ref{lem11.9} hold if in the LHS we replace $ P $ by $ Q, $ and exchange (2) with (3), and if in the RHS we replace $ p $ by $ q, \ f $  by $ g, $ and exchange $ m_2 $ with $ m_3; $ also, in the twisted case we should add $ n(p+q) $ to the subscript on the RHS, and in the $ +, \tw $ 
case multiply the RHS by $ (-1)^{np} $.
\end{lemma}

\begin{proposition}
\label{prop11.11}
Let $ \La = \La^{[K;0, m_3](1)} $ for $ m_3 \leq np -1 $ and $ \La' = \La^{[K; m_2,0](1)} $ for $ m_2 \leq nq-1 $ are weights from $ P^K_{+,1}. $ Then 
\begin{enumerate}
\item[(a)] The $ P $-modified (super)characters of
  big $N=4$ superconformal algebra
  modules, obtained by (twisted) QHR from the $ \wg $-modules $ \LLa $ and $ L(\La') $, are given by the followings formulas:
\[ 
\begin{array}{ll}
\overset{b4}{R}\vphantom{R}^-\tch^{-[P]}_{H(\La)} =  F^{[0,\half]}_{np-m_3}, & \overset{b4}{R}\vphantom{R}^-\tch^{-[P]}_{H(\La')} =  F^{[0,\half]}_{np+m_2}, \\
\overset{b4}{R}\vphantom{R}^+\tch^{+[P]}_{H(\La)} =  -(-1)^{m_3} F^{[\half,\half]}_{np-m_3}, & \overset{b4}{R}\vphantom{R}^+\tch^{+[P]}_{H(\La')} =  -F^{[\half,\half]}_{np+m_2}, \\
\overset{b4}{R}\vphantom{R}^{-,\tw}\tch^{-[P]}_{H^\tw(\La)} =  -F^{[0,0]}_{np-m_3}, & \overset{b4}{R}\vphantom{R}^{-,\tw}\tch^{-[P]}_{H^\tw(\La')} =  -F^{[0,0]}_{np+m_2}, \\
\overset{b4}{R}\vphantom{R}^{+,\tw}\tch^{+[P]}_{H^\tw(\La)} = (-1)^{m_3}  F^{[\half, 0]}_{np-m_3}, & \overset{b4}{R}\vphantom{R}^{\tw, +}\tch^{+[P]}_{H^\tw(\La')} =  F^{[\half,0]}_{np+m_2}. \\
\end{array} \]
\item[(b)] The $ Q $-modified (super)characters of  big $N=4$ superconformal algebra modules, obtained by (twisted) QHR from the $ \wg $-modules $ \LLa $ and $ L(\La') $, are given by the following formulas:
\[ 
\begin{array}{ll}
\overset{b4}{R}\vphantom{R}^-\tch^{-[Q]}_{H(\La)} = G^{[0, \half]}_{nq+m_3}, & \overset{b4}{R}\vphantom{R}^-\tch^{-[Q]}_{H(\La')} = G^{[0, \half]}_{nq-m_2}, \\
\overset{b4}{R}\vphantom{R}^+\tch^{+[Q]}_{H(\La)} = -G^{[\half, \half]}_{nq+m_3}, & \overset{b4}{R}\vphantom{R}^-\tch^{-[Q]}_{H(\La')} = -(-1)^{m_2} G^{[\half, \half]}_{nq-m_2}, \\
\overset{b4}{R}\vphantom{R}^{-,\tw}\tch^{-[Q]}_{H^\tw(\La)} = G^{[0, 0]}_{np-m_3}, & \overset{b4}{R}\vphantom{R}^{-,\tw}\tch^{-[Q]}_{H^\tw(\La')} = G^{[0, 0]}_{np+m_2}, \\
\overset{b4}{R}\vphantom{R}^{+,\tw}\tch^{+[Q]}_{H^\tw(\La)} = -(-1)^{np} G^{[\half, 0]}_{np-m_3}, & \overset{b4}{R}\vphantom{R}^{+,\tw}\tch^{+[Q]}_{H^\tw(\La')} = -(-1)^{m_2+np} G^{[\half, 0]}_{np+m_2}. \\
\end{array} \]
\item[(c)] The span of the $ P $ (resp. $ Q $)-modified (super)characters, given in (a) (resp. (b)) is $ SL_2(\ZZ) $-invariant with respect to the action 
\[ \varphi \big{|}_{\left(\begin{smallmatrix}
a& b \\
c & d\\
\end{smallmatrix}\right)} (\tau, y_2, y_3) = e^{-\frac{\pi ic }{2(c \tau +d)} ((qn-1)y_2^2 + (pn-1)y_3^2)} \varphi \left(\frac{a\tau +b}{c\tau +d}, 
\frac{y_2}{c\tau+d}, \frac{y_3}{c\tau +d}\right). \]
\end{enumerate}
\end{proposition}
\begin{proof}
Formulas in (a) and (b) are obtained by applying the Hamiltonian reduction $ \varphi \rightarrow \varphi_\HR $ to the formulas in Lemma \ref{lem10.18} (a), using Lemma \ref{lem11.5}, and Remark \ref{rem11.6}, and Lemmas \ref{lem11.9}, \ref{lem11.10}. Claim (c) follows from Lemma \ref{lem11.8} and formulas \eqref{11.13}, \eqref{11.14}.
\end{proof}

\begin{remark}
\label{rem11.12}
It follows from Proposition \ref{prop11.3} that the modules 
$H(\Lambda^{[K;m_2,m_3](j)})$ for $j=2,3$ 
are not isomorphic to those with $j=1$ (and $j=0$), and the same holds for $H^\tw$. Consequently, their (super)characters do not
lie in the span of the (super)characters for $j=1$. However, it
follows from Propositions 
\ref{prop10.19}  and \ref{prop10.20} that their P (resp Q)-type modifications
do lie in the space, described by Proposition \ref{prop11.11}. Thus, the span of P (resp. Q)-modified (super)characters of all modules 
$H(\Lambda)$ and   $H^\tw(\Lambda)$, where $\Lambda \in \bigcup\limits_{j=0}^{3} P^K_{+,j}$, is $SL_2(\ZZ)$-invariant.    
\end{remark}
\begin{remark}
\label{rem11.13}
It is shown in \cite{AKMPP} that $K=a$ is the collapsing level of the simple minimal $W$-algebra $ W_K (D (2,1; a), e_{-\theta}) $. Namely this vertex algebra is isomorpic to  the simple affine vertex algebra 
$V_{\frac{p}{q}-1}$, 
associated to the affine Lie algebra $\hat{sl}_2$ at level $\frac{p}{q}-1$. The integrable case corresponds to $q=n=1$. It is easy to see that in this case all big
$N=4$ modules, obtained by QHR of integrable $\fg$-modules, are of the form
 $H(\Lambda^{[K;0,m_3](1)})$ and  $H^\tw(\Lambda^{[K;0,m_3](1)})$ 
for $m_3\leq p-1$, $m_3\in \ZZ_{\geq 0}$. It is not difficult to see that their (super)characters are just the normalized characters of integrable highest weight $\hat{sl}_2$-modules of level $p-1$. Namely, we have for
$\Lambda^{[K;0,s](1)})$  ( $s=0,1,..., p-1$):
\[ ch^\pm_{H(\Lambda)} (\tau,y_2,y_3) = ch_{s,p-1}(\tau, y_3),   \,\,\, 
ch^\pm_{H^\tw(\Lambda)} (\tau,y_2,y_3) = ch_{p-s-1,p-1}(\tau, y_3),\]
where   
\[ ch_{j,p-1}(\tau,y)=\frac{\Theta_{j+1,p+1} (\tau, y) -\Theta_{-j-1,p+1} (\tau, y)}{\Theta_{1,2} (\tau, y) -\Theta_{-1,2} (\tau, y)}, \, j=0,1,...,p-1,\]
are the normalized characters of irreducible $V_{p-1}$-modules \cite{K90}.  
\end{remark}

Next, we give the modular transformations of functions $ f^{[\epsilon, \epsilon']}_j $ and $ g^{[\epsilon, \epsilon']}_j $ introduced before Lemma \ref{lem11.2}, which will be used in the next section. The proof is the same as that of Proposition \ref{prop10.21}.
\begin{proposition}
  \label{prop11.14}
In coordinates \eqref{10.06}, \eqref{10.07} one has:
\begin{enumerate}
\item[(a)] 
\[ 
\begin{aligned}
f^{[\epsilon, \epsilon']}_j \left( -\frac{1}{\tau}, \frac{z}{\tau}, t \right) = e^{4 \pi i n p \epsilon\epsilon'} \frac{(-i\tau)^\half \tau}{\sqrt{2n(p+q)}} e^{-\frac{\pi iK}{\tau}(z|z)} \sum_{k \in \ZZ / 2n(p+q)\ZZ} e^{-\frac{\pi i j k }{n(p+q)}} f_k^{[\epsilon', \epsilon]} \tzt, \\
g^{[\epsilon, \epsilon']}_j \left( -\frac{1}{\tau}, \frac{z}{\tau}, t \right) =e^{4 \pi i n q \epsilon\epsilon'} \frac{(-i\tau)^\half \tau}{\sqrt{2n(p+q)}} e^{-\frac{\pi iK}{\tau}(z|z)} \sum_{k \in \ZZ / 2n(p+q)\ZZ} e^{-\frac{\pi i j k }{n(p+q)}} g_k^{[\epsilon', \epsilon]} \tzt.\\
\end{aligned} \]
\item[(b)]
\[ 
\begin{aligned}
f^{[\epsilon, \epsilon']}_j  (\tau +1, z, t) = e^{-2 \pi i n p \epsilon'^2} e^{\frac{\pi i j^2}{2n(p+q)}} f_j^{[|\epsilon-\epsilon'|, \epsilon']} \tzt, \\
g^{[\epsilon, \epsilon']}_j (\tau +1, z, t) = e^{-2 \pi i n q \epsilon'^2} e^{\frac{\pi i j^2}{2n(p+q)}} g_j^{[|\epsilon-\epsilon'|, \epsilon']} \tzt. \\
\end{aligned} \]
\end{enumerate}
\end{proposition}

In conclusion of this section, we prove the following proposition.
\begin{proposition}
  \label{prop11.15}
The P(resp. Q)-modified (super)characters of the big $N=4$ superconformal algebras are holomorphic in $y_3$ (resp. $y_2$).
\end{proposition}


In order to prove this proposition, recall the mock theta function $ \varphi^{+[m,s]} \tuvt $ and its modification $ \tilde{\varphi}^{+[m,s]} = \varphi^{+[m,s]} -\half \varphi^{+[m,s]}_{\text{add}}, $ where $ m \in \zp, s \in \ZZ, $ see \cite{KW17}, Section 2. To simplify notation, let 
\[ \varphi^{[m,s]} (\tau, u, v) = \varphi^{+[m,s]} (\tau, u,v,0), \, \varphi^{[m,s]}_{\text{add}} (\tau, u, v) = \varphi^{+[m,s]}_{\text{add}} (\tau, u, v, 0), \, \tilde{\varphi}^{[m,s]} = \varphi^{[m,s]} - \half \varphi^{[m,s]}_{\text{add}}. \]
The functions $ F_j^{[\epsilon, \epsilon']} $ and $ G_j^{[\epsilon, \epsilon']} $ are written in terms of the functions $ \tilde{\varphi}^{[m,s]} $ as follows: 
\[ 
F_j^{[\epsilon, \epsilon']} (\tau, y_2, y_3) = q^{-np\epsilon'^2} e^{-2\pi i n p \epsilon' y_2} \left( \Theta_{j, n(p+q)} - \Theta_{-j, n(p+q)} \right) (\tau, y_2) \tilde{\varphi}^{[np; 0]} \left(\tau, - \frac{y_3}{2}, \frac{y_2}{2} + \epsilon' \tau + \epsilon \right) \]
and similarly for $ G $ with $ p $ switched with $ q $ and $ y_2 $ switched with $ y_3. $

Recall that, obviously, the function 
\[ \vartheta_{11} (\tau, v-u) \vartheta_{11} (\tau, v + u) \tilde{\varphi}^{[m,s]} (\tau, u, v) \]
is holomorphic in $ u $ (cf. \cite{KW17}),  hence
\[ \vartheta_{11} \left(\tau, \frac{y_2 + y_3}{2}\right) \vartheta_{11} \left( \tau, \frac{y_2-y_3}{2}\right) F(\text{resp. }G)_j^{[\epsilon, \epsilon']} (\tau, y_2, y_3)\, \mbox{is holomorphic in}\,  y_3 \,\mbox{(resp.}  y_2 ).\]

For simplicity, we shall consider only the case $ \epsilon = \epsilon' = 0, $ namely, we will show that $ \left( F (\text{resp. } G)_j^{[0,0]}
 (\overset{b4}{R}\vphantom{R}^{(0)}_0)^{-1} \right) (z, y_2, y_3) $ is holomorphic in $ y_3 $ (resp. $ y_2 $).
For that it suffices to prove the following lemma.
\begin{lemma}
\label{lem11.16}
\begin{enumerate}
\item [(a)] The function 
$\frac{\vartheta_{11} (\tau, v-u) \vartheta_{11} (\tau, v+u) \varphi^{[m,s]} (\tau, u, v)}{\vartheta_{11} (\tau, 2u)}$
is holomorphic in $ u. $
\item [(b)] The function 
$\frac{\varphi^{[m,s]}_{\text{add}} (\tau, u, v)}{\vartheta_{11}(\tau, 2 u)}$
is holomorphic in $ u. $
\end{enumerate}
\end{lemma}
The proof of this lemma uses two lemmas. 
\begin{lemma}
\label{lem11.17}
$ \varphi^{[m,s]} (\tau, u, v) = 0 $ for $ u = a \tau + b $ where $ a, b \in \half \ZZ. $
\end{lemma}
\begin{proof}
We use the following properties of $ \varphi^{[m,s]} $ (see e.g. \cite{KW14}):
\begin{equation}
\label{11.15}
\varphi^{[m,s]} (\tau, -u, v) = - \varphi^{[m,s]} (\tau, u, v),
\end{equation}
\begin{equation}
\label{11.16}
\varphi^{[m,s]} (\tau, u + j \tau, v) = q^{-mj^2} e^{-4 \pi i m j u } \varphi^{[m,s]} (\tau, u, v) \text{ for } j \in \ZZ, 
\end{equation}
\begin{equation}
\label{11.17}
\varphi^{[m,s]} (\tau, u+a, v+b) = \varphi^{[m,s]}(\tau,u,v) \text{ for } a,b \in \ZZ. 
\end{equation}
The claim for $ u = 0 $ follows from \eqref{11.15}. By \eqref{11.17} we have $ \varphi^{[m,s]} \left(\tau, -\half, v\right) = \varphi^{[m,s]} \left( \tau, \half, v \right)$, which, by \eqref{11.15} implies that
$ \varphi^{[m,s]} \left( \tau, \half, v \right) = 0. $
The proof of the claim for $ u = \tot $ (resp. $ u = \frac{\tau +1 }{2} $) is similar, using \eqref{11.15} and \eqref{11.16} (resp. and \eqref{11.17}). Thus, the claim holds for $ u = 0, \half, \tot, \frac{\tau +1}{2}. $ Using \eqref{11.16}, we deduce that the claim holds for $ u = a \tau, a \tau + \half $ where $ a \in \half \ZZ.  $ Finally, applying \eqref{11.17}, we get the claim for all $ u $ in question. 
\end{proof}
Note that Lemma \ref{lem11.16}(a) follows from Lemma \ref{lem11.17} since the set of zeros of $ \vartheta_{11} (\tau, 2u) $, as a function in $ u, $ is $ \half (\ZZ + \tau \ZZ). $
\begin{lemma}
\label{lem11.18} The function $ \left( \Theta_{-j,m} - \Theta_{j,m} \right) (\tau, 2u), $ where $ j \in \ZZ,  $ vanishes at $ u \in \half (\ZZ + \tau \ZZ). $
\end{lemma}
\begin{proof}
It is straightforward, using the first formula in \eqref{1.2} and the relation 
\[ \Theta_{j,m} (\tau, 2 a \tau) = q^{-ma^2} \Theta_{j+2am,m} (\tau, 0) \,\, 
\mbox{for} \,\, a \in \half \ZZ. \]
\end{proof}
Lemma \ref{lem11.16} (b) follows from Lemma \ref{lem11.18} since  \[ \varphi^{[m,s]}_\text{add} (\tau, u, v) = \sum_{j=s}^{s+2m-1} R_{j,m} (\tau, v) \left(\Theta_{-j,m} - \Theta_{j,m} \right) (\tau, 2 u).  \]

\section{Characters of integrable $ \wg $-modules in the case of $ \fg = D(2,1; -\frac{p}{p +1}) $}

As has been pointed out in Remark \ref{rem11.13}, the case $ \fg = D(2,1;a) $ with $ a = K = -\frac{p}{p+1}, $ where $ p \in \ZZ_{\geq 1} $, is very specials for the big $ N=4 $ superconformal algebras. Here we show that in this case the normalized supercharacters of integrable $ \wg $-modules form a modular invariant family before the modification. 

First, by Proposition \ref{prop10.1}, we have
\begin{lemma}
\label{lem12.1}
Let $ \fg = D(2,1;-\frac{p}{p+1}) $ and $ K = - \frac{p}{p+1}, $ where $ p \in \ZZ_{\geq 1}. $ Then
\[ 
\begin{aligned}
P^K_{+,0} & = \{  \La^{[K;0,s](0)}, \mbox{ where } s \in \ZZ_{\geq 0}, \ s \leq p-1     \} \\
P^K_{+,1} & = \{  \La^{[K;0,s](1)} \mbox{ where } s \in \ZZ_{\geq 0}, \ s \leq p \} \cup \{ \La^{[K;1,0](1)} \},\, 
P^K_{+,j}  = \emptyset \mbox{ for } j = 2,3. \\
\end{aligned} \]
\end{lemma}
Next we have  
\begin{lemma}
\label{lem12.2}
If $ K = a = -\frac{p}{p+1}, p \in \ZZ_{\geq 1}, $ then we have for $ \La \in P^K_{+, 0} \cup P^K_{+,1}: $
\begin{equation}
\label{12.1}
\hat{R}\vphantom{R}^- \ch_\La^- = B^-_\La.
\end{equation}
\end{lemma}

\begin{proof}
It follows from Lemma \ref{lem10.18} for the $ Q $-modification that this formula holds for the $ Q $-modified $ \ch_\La  $ and $ B^-_\La. $ But then this formula holds without modification since in the formula for $ Q_j $ before Lemma \ref{lem10.12} one can replace $ \tph^{[nq;0]} $ by $ \Phi^{[1;0]} $ by Remark \ref{rem1.3} since $ nq =1. $
\end{proof}

By Remark \ref{rem1.3} and Lemma \ref{lem11.10}, we can express the function $ B^-_\La $ (hence $ B^+_\La$ and $ B^{\pm, \tw}_\La $, by \eqref{11.5}, \eqref{11.6}) with $ \La $ as in Lemma \ref{lem12.1}, in terms of the functions $ g^{[\epsilon, \epsilon']}_j, $ introduced before Lemma \ref{lem11.2}, where $ \tps $ are replaced by $ \psi $ since $ nq =1: $
\begin{equation}
\label{12.2}
\begin{aligned}
B^\pm_\La &= \pm g^{[\epsilon, 0]}_{-s-2} \mbox{ if } \La = \La^{[K;D,s](0)}, \ = \mp g^{[\epsilon, 0]}_s \mbox{ if } \La = \La^{[K;0,s](1)}; \\
B^\pm_\La &= \mp g^{[\epsilon, 0]}_{-1} \mbox{ if } \La = \La^{[K;1,0](1)};\\
\end{aligned}
\end{equation}

\begin{equation}
\label{12.3}
\begin{aligned}
B^{\pm, \tw}_\La & = \mp (-1)^{2 \epsilon p} g^{[\epsilon, \half]}_{p-s-1} \mbox{ if } \La = \La^{[K;0,s](0)}, \ = \pm (-1)^{2 \epsilon p} g^{[\epsilon, \half]}_{p+s+1} \mbox{ if } \La = \La^{[K;0,s](1)}; \\
B^{\pm, \tw}_\La & = \mp (-1)^{2 \epsilon p} g^{[\epsilon, \half]}_{p} \mbox{ if } \La = \La^{[K;1,0](1)}. \\
\end{aligned}
\end{equation}

In all formulas $ \epsilon = \half $ for + and $ = 0 $ for $ -, $ and since $ n = q = 1, $ we have:
\begin{equation}
\label{12.4}
\begin{aligned}
& g^{[\epsilon, \epsilon']}_j \tzzzt = e^{-\frac{2 \pi i pt}{p+1}} \\
&\times \left( \Theta_{j, p+1} \left( \tau, z_1 + \frac{z_2-pz_3}{p+1}\right)  \Psi^{[1,1,0; \epsilon]}_{\epsilon', -\epsilon';\epsilon'} (\tau, z_1, -z_2)\right. \\ 
&\left. - \Theta_{-j, p+1} \left( \tau, z_1 - \frac{z_2+(p+2)z_3}{p+1}\right) \Psi^{[1,1,0; \epsilon]}_{\epsilon', -\epsilon';\epsilon'} (\tau, z_1 -z_2-z_3, z_3) \right).
\end{aligned}
\end{equation}
It follows from Remark 1.3 and the definition \eqref{1.10} of the function $ \Psi $ that
\begin{equation}
\label{12.5}
\Psi^{[1,1,0; \epsilon]}_{\epsilon', -\epsilon', \epsilon'} (\tau, z_1, z_2) = (-1)^{1-2 \epsilon} i \frac{\eta (\tau)^3 \vartheta_{11} (\tau, z_1+z_2)}{\vartheta_{1-2\epsilon', 1-2\epsilon} (\tau, z_1) \vartheta_{1-2\epsilon', 1-2\epsilon} (\tau, z_2)}.
\end{equation}

Finally recall that the normalized (super)denominator $ \hat{R}\vphantom{R}^\pm $ is given by \eqref{2.5} and that the twisted one is
$ \hat{R}\vphantom{R}^{\pm, \tw} = w_0 \hat{R}\vphantom{R}^\pm $. More explicitly, they are given in coordinates \eqref{10.06}, \eqref{10.07} by a uniform expression (cf. \cite{KW14}, Section 4)
\begin{equation}
\label{12.6}
\hat{R}\vphantom{R}^{(\epsilon)}_{\epsilon'} \tzzzt = -(-1)^{4 \epsilon \epsilon'} i \frac{\eta (\tau)^4 \vartheta_{11} (\tau, z_1-z_2) \vartheta_{11} (\tau, z_1-z_3) \vartheta_{11} (\tau, z_2 + z_3)}{\displaystyle{\prod_{u = z_1, z_2, z_3, z_1-z_2-z_3}} \vartheta_{1-2\epsilon', 1-2\epsilon} (\tau, u) },
\end{equation} 
where, as before, $ \epsilon = 0 $ 
(resp. $ = \half $) 
in + (resp. $ - $ case) and $ \epsilon' = 0 $ (resp. $ = \half $) in the non-twisted (resp. twisted) case. 

Formulas \eqref{12.4}--\eqref{12.6} imply in the cases $ n = q =1: $
\begin{equation}
\label{12.7}
\begin{aligned}
& \frac{g_j^{[\epsilon, \epsilon']} \tzzzt}{\hat{R}\vphantom{R}^{(\epsilon)}_{\epsilon'} (\tau, z_1, z_2, z_3)} = \frac{e^{-\frac{2 \pi i p t}{p+1}}}{\eta (\tau) \vartheta_{11} (\tau, z_1-z_3) \vartheta_{11} (\tau, z_2+z_3)} \\
&\times \left( (-1)^{1-2\epsilon'} \Theta_{j, p+1} \left( \tau, z_1 + \frac{1}{p+1}z_2 - \frac{p}{p+1}z_3 \right) \vartheta_{1-2\epsilon', 1-2\epsilon} (\tau, z_3) \vartheta_{1-2\epsilon', 1-2\epsilon} (\tau, z_1-z_2-z_3) \right. \\
& \left. -(-1)^{2 \epsilon (1-2\epsilon')} \Theta_{-j, p+1}  \left( \tau, z_1 - \frac{1}{p+1}z_2 - \frac{p+2}{p+1}z_3 \right) \vartheta_{1-2\epsilon', 1-2\epsilon} (\tau, z_1) \vartheta_{1-2\epsilon', 1-2\epsilon} (\tau, z_2)   \right).
\end{aligned}
\end{equation}
We thus obtain the following:
\begin{proposition}
\label{prop12.3}
In the case $ n = q = 1 $, the (super)characters of integrable $ \hat{D} (2,1;a) $-modules are given by the following formulas:
\[ 
\begin{aligned}
\ch_{\La^{[K;0,s ](0)}}^\pm &= \pm \frac{g^{[\epsilon,0]}_{-s -2}}{\hat{R}\vphantom{R}^{(\epsilon)}_0} \quad (0 \leq s  \leq p-1), \\
\ch_{\La^{[K;0,s](1)}}^\pm &=  \mp \frac{g^{[\epsilon,0]}_{s }}{\hat{R}\vphantom{R}^{(\epsilon)}_0} \quad (0 \leq s  \leq p),\,\,\,\,
\ch_{\La^{[K;1,0](1)}}^\pm  =  \mp \frac{g^{[\epsilon,0]}_{-1}}{\hat{R}\vphantom{R}^{(\epsilon)}_0};\\
\ch_{\La^{[K;0,s ](0)}}^{\pm, \tw } & = \mp (-1)^{2 p \epsilon} \frac{g^{[\epsilon, \half]}_{-s +p-1}}{\hat{R}\vphantom{R}^{(\epsilon)}_\half} \quad (0 \leq s  \leq p-1),  \\ 
\ch_{\La^{[K;0,s](1)}}^{\pm, \tw } & =\pm (-1)^{2 p \epsilon}  \frac{g^{[\epsilon, \half]}_{s +p+1}}{\hat{R}\vphantom{R}^{(\epsilon)}_\half} \quad (0 \leq s  \leq p),   \,\,\,\,
\ch_{\La^{[K;1,0](1)}}^{\pm, \tw }  = \pm(-1)^{2 p \epsilon} \frac{g^{[\epsilon, \half]}_{p}}{\hat{R}\vphantom{R}^{(\epsilon)}_\half}, \\
\end{aligned} \]
where $ \epsilon = 0 $ (resp. $ \half $) in the $ - $ (resp. +) case, and the RHS are given by \eqref{12.7}.
\end{proposition}
 

The modular transformations of the $ \hat{D}(2,1;a) $-denominators \eqref{12.6} follow from that of the Jacobi forms $ \vartheta_{ab} \tz $ (see e.g. \cite{KW14}, Appendix):
\begin{equation}
\label{12.8}
\begin{aligned}
\hat{R}\vphantom{R}^{(\epsilon)}_{\epsilon'} \left(-\frac{1}{\tau}, \frac{z}{\tau}\right) &= i (-i \tau )^\frac{3}{2} \hat{R}\vphantom{R}^{(\epsilon')}_{\epsilon} \tz,\quad \hat{R}\vphantom{R}^{(\epsilon)}_{\epsilon'} (\tau +1, z)  = e^{\frac{\pi i }{12}} \hat{R}\vphantom{R}^{(|\epsilon-\epsilon'|)}_{\epsilon'} \tz. 
\end{aligned}
\end{equation}

The modular transformation of supercharacters in the case $ n=q=1 $ are obtained from Proposition \ref{prop12.3}, Proposition \ref{prop11.14} and \eqref{12.8}.

\begin{proposition}
\label{prop12.4}
The modular transformations of normalized supercharacters of integrable $ \hat{D}\left(2,1;-\frac{p}{p+1} \right) $-modules of level $ K = -\frac{p}{p+1} $ are as follows:
\begin{enumerate}
\item[(a)]
\[ 
\begin{aligned}
\ch^-_{\La^{[K;0,s](0)}} \left( -\frac{1}{\tau}, \frac{z}{\tau}, t\right)  = & \frac{e^{-\frac{\pi i K}{\tau} (z|z)}}{\sqrt{2(p+1)}} \left( \sum_{s'=0}^{p-1} e^{-\frac{\pi i}{p+1} (s+2)(s'+2)} \ch^-_{\La^{[K;0,s'](0)}} \tzt \right. \\
& \left. - \sum_{s'=0}^{p} e^{\frac{\pi i }{p+1}(s+2)s'} \ch^-_{\La^{[K;0,s'](1)}} \tzt - e^{-\frac{\pi i }{p+1}(s+2)} \ch^-_{\La^{[K;1,0](1)}} \tzt \vphantom{\sum_{s'=0}^{p}}\right).
\end{aligned} \]
for $ s = 0,1, \ldots, p-1;  $ 
\[ 
\begin{aligned}
\ch^-_{\La^{[K;0,s](1)}} \left( -\frac{1}{\tau}, \frac{z}{\tau}, t\right) = & \frac{-e^{-\frac{\pi i K}{\tau}(z|z)}}{\sqrt{2(p+1)}} \left(  \sum_{s'=0}^{p-1} e^{\frac{\pi i }{p+1}s(s'+2)} \ch^-_{\La^{[K;0,s'](0)}} \tzt  \right. \\
& \left. -   \sum_{s'=0}^{p} e^{-\frac{\pi i }{p+1}ss'} \ch^-_{\La^{[K;0,s'](1)}} \tzt - e^{\frac{\pi i }{p+1}s}    \ch^-_{\La^{[K;1,0](1)}}  \tzt   \right)
\end{aligned} \]
for $ s = 0,1, \ldots, p ;  $
\[ 
\begin{aligned}
\ch^-_{\La^{[K;1,0](1)}} \left( -\frac{1}{\tau}, \frac{z}{\tau}, t\right) = & \frac{-e^{-\frac{\pi i K}{\tau}(z|z)}}{\sqrt{2(p+1)}} \left(  \sum_{s'=0}^{p-1} e^{-\frac{\pi i }{p+1}(s'+2)} \ch^-_{\La^{[K;0,s'](0)}} \tzt  \right. \\
& \left. -   \sum_{s'=0}^{p} e^{\frac{\pi i }{p+1}s'} \ch^-_{\La^{[K;0,s'](1)}} \tzt - e^{-\frac{\pi i }{p+1}}    \ch^-_{\La^{[K;1,0](1)}}  \tzt   \right).
\end{aligned} \]
\item[(b)]
\[ 
\begin{aligned}
ch^-_{\La^{[K;0,s](0)}} (\tau +1, z, t) &= e^{\frac{\pi i }{2(p+1)}(s+2)^2-\frac{\pi i }{12}} \ch^-_{\La^{[K;0,s](0)}} \tzt, \\
ch^-_{\La^{[K;0,s](1)}} (\tau +1, z, t) &= e^{\frac{\pi i }{2(p+1)}s^2-\frac{\pi i }{12}} \ch^-_{\La^{[K;0,s](1)}} \tzt, \\
ch^-_{\La^{[K;1,0](1)}} (\tau +1, z, t) &= e^{\frac{\pi i }{2(p+1)}-\frac{\pi i }{12}} \ch^-_{\La^{[K;1,0](1)}} \tzt. \\
\end{aligned} \]
\end{enumerate}
\end{proposition}

\begin{remark}
\label{rem12.5}
Let $ \La $ be one of the weights appearing in Proposition \ref{prop12.3}. Then there exists a unique $ j \bmod 2(p+1) \ZZ, $ such that $ \ch^-_\La = \pm  g_j^{[0,0]}/\hat{R}\vphantom{R}^{(0)}_0, $ and we let $ \La = \La^j. $ Using Proposition \ref{prop12.4}, Verlinde's formula gives the following fusion rules for the supercharacters:
\[ N_{\La^i, \La^j, \La^k} = 1 \mbox{ if } i + j + k \in 2(p+1)\ZZ, \mbox{ and } = 0 \mbox{ otherwise. }\]
\end{remark}

\end{document}